\documentclass[11pt]{amsart}
\usepackage{amsmath, amssymb, tikz, tikz-cd} 
\usepackage[colorlinks=true, linkcolor=red!80!black, urlcolor=purple, citecolor=blue!70!black]{hyperref}
\usepackage{geometry}
\usepackage{fullpage}
\usepackage{mathrsfs}
\usepackage{graphicx,color}
\usepackage[all,cmtip]{xy}
\usepackage{enumitem}
\usepackage{verbatim}
\usepackage{bbm}
\usepackage{bm}
\usepackage{mathtools}
\usepackage{tabu}
\usepackage{comment}

\usepackage[mathscr]{eucal}

\newtheorem{theorem}{Theorem}[section]
\newtheorem{proposition}[theorem]{Proposition}
\newtheorem{corollary}[theorem]{Corollary}

 \newtheorem{lemma}[theorem]{Lemma}

\newtheorem{prop}[theorem]{Proposition}
\newtheorem{thm}[theorem]{Theorem}
\newtheorem{lem}[theorem]{Lemma}
\newtheorem{cor}[theorem]{Corollary}

\newtheorem{thm-defn}[theorem]{Theorem-Definition}

\theoremstyle{definition}
\newtheorem{definition}[theorem]{Definition}
\newtheorem{remark}[theorem]{Remark}
\newtheorem{rem}[theorem]{Remark}

\newtheorem{defn}[theorem]{Definition}
\newtheorem{defn-prop}[theorem]{Definition-Proposition}

\numberwithin{equation}{section}

\newcommand{\bP}{\mathbb{P}}
\newcommand{\bC}{\mathbb{C}}
\newcommand{\bQ}{\mathbb{Q}}
\newcommand{\bZ}{\mathbb{Z}}
\newcommand{\bG}{\mathbb{G}}
\newcommand{\bL}{\mathbb{L}}
\newcommand{\PGL}{\mathrm{PGL}}
\newcommand{\Hilb}{\mathrm{Hilb}}
\newcommand{\calX}{\mathcal{X}}
\newcommand{\calL}{\mathcal{L}}
\newcommand{\calO}{\mathcal{O}}

\newcommand{\Fut}{\mathrm{Fut}}
\newcommand{\Aut}{\mathrm{Aut}}
\newcommand{\vol}{\mathrm{vol}}
\newcommand{\ord}{\mathrm{ord}}
\newcommand{\lct}{\mathrm{lct}}
\newcommand{\Supp}{\mathrm{Supp}}
\newcommand{\tT}{\widetilde{T}}
\newcommand{\cO}{\mathcal{O}}
\newcommand{\bA}{\mathbb{A}}
\newcommand{\sF}{\mathscr{F}}
\newcommand{\Proj}{\mathrm{Proj}}
\newcommand{\oM}{\overline{M}}
\newcommand{\GIT}{\mathrm{GIT}}
\newcommand{\K}{\mathrm{K}}
\newcommand{\cY}{\mathcal{Y}}
\newcommand{\cD}{\mathcal{D}}
\newcommand{\cV}{\mathcal{V}}

\newcommand{\CC}{\mathbb{C}}
\newcommand{\bR}{\mathbb{R}}

\newcommand{\calF}{\mathcal{F}}

\newcommand{\calE}{\mathcal{E}}
\newcommand{\Sym}{\mathrm{Sym}}

\newcommand{\Hom}{\mathrm{Hom}}

\newcommand{\Pic}{\mathrm{Pic}}

\newcommand{\Spec}{\mathrm{Spec}\;}

\renewcommand{\arraystretch}{1.2}

\newcommand{\PP}{\mathbb{P}}
\newcommand{\cS}{\mathcal{S}}
\newcommand{\cF}{\mathcal{F}}
\newcommand{\cG}{\mathcal{G}}

\newcommand{\SL}{\mathrm{SL}}
\newcommand{\calD}{\mathcal{D}}

\newcommand{\calN}{\mathcal{N}}

\newcommand{\calT}{\mathcal{T}}

\newcommand{\calZ}{\mathcal{Z}}

\newcommand{\cP}{\mathcal{P}}

\newcommand{\sExt}{\mathcal{E}xt}

\usepackage{graphicx}
\newcommand{\sslash}{\mathbin{/\mkern-6mu/}}

\newcommand{\Chow}{\operatorname{Chow}}

\newcommand{\cX}{\mathcal X}
\newcommand{\cT}{\mathcal T}

\newcommand{\Ext}{\mathrm{Ext}}

\newcommand{\hvol}{\widehat{\mathrm{vol}}}

\newcommand{\CM}{\mathrm{CM}}

\newcommand{\kst}{\mathrm{kst}}

\newcommand{\cM}{\mathcal M}
\newcommand{\cN}{\mathcal N}
\newcommand{\cZ}{\mathcal Z}
\newcommand{\cL}{\mathcal L}
\newcommand{\cI}{\mathcal{I}}
\newcommand{\cK}{\mathcal{K}}

\newcommand{\tsigma}{\tilde{\sigma}}

\newcommand{\fm}{\mathfrak{m}}

\newcommand{\calC}{\mathcal C}

\newcommand{\ocM}{\overline{\mathcal{M}}}
\newcommand{\ocK}{\overline{\mathcal{K}}}

\newcommand{\cE}{\mathcal{E}}
\newcommand{\Hdg}{\mathrm{Hodge}}

\newcommand{\sM}{\mathscr{M}}
\newcommand{\oK}{\overline{K}}
\newcommand{\red}{\mathrm{red}}

\newcommand{\bfP}{\mathbf{P}}
\newcommand{\bfA}{\mathbf{A}}
\newcommand{\fM}{\mathfrak{M}}
\newcommand{\osM}{\overline{\sM}}
\newcommand{\ofM}{\overline{\fM}}

\newcommand{\sX}{\mathscr{X}}
\newcommand{\sY}{\mathscr{Y}}
\newcommand{\sZ}{\mathscr{Z}}
\newcommand{\sD}{\mathscr{D}}

\newcommand{\hsF}{\widehat{\sF}}

\newcommand{\sC}{\mathscr{C}}
\newcommand{\sS}{\mathscr{S}}
\newcommand{\fp}{\mathfrak{p}}
\newcommand{\slc}{\mathrm{slc}}
\newcommand{\tcK}{\widetilde{\mathcal{K}}}

\newcommand{\tX}{\widetilde{X}}
\newcommand{\tY}{\widetilde{Y}}
\newcommand{\tS}{\widetilde{S}}
\newcommand{\tE}{\widetilde{E}}
\newcommand{\tH}{\widetilde{H}}
\newcommand{\sT}{\mathscr{T}}
\newcommand{\sH}{\mathscr{H}}
\newcommand{\PE}{\mathbb{P}\mathcal{E}}
\newcommand{\tPE}{\widetilde{\mathbb{P}\mathcal{E}}}
\newcommand{\cHom}{\mathcal{H}om}
\newcommand{\bfV}{\mathbf{V}}
\newcommand{\bfH}{\mathbf{H}}
\newcommand{\Stab}{\mathrm{Stab}}
\newcommand{\bmu}{\bm{\mu}}
\newcommand{\fS}{\mathfrak{S}}
\newcommand{\hH}{\widehat{H}}
\newcommand{\tV}{\widetilde{V}}
\newcommand{\trho}{\tilde{\rho}}
\newcommand{\tfM}{\widetilde{\mathfrak{M}}}
\newcommand{\tlambda}{\tilde{\lambda}}
\newcommand{\fY}{\mathfrak{Y}}
\newcommand{\ofY}{\overline{\mathfrak{Y}}}
\newcommand{\osY}{\overline{\mathscr{Y}}}

\newcommand{\tL}{\widetilde{L}}
\newcommand{\tW}{\widetilde{W}}
\newcommand{\sU}{\mathscr{U}}

\newcommand{\tpsi}{\tilde{\psi}}
\newcommand{\tSigma}{\widetilde{\Sigma}}
\newcommand{\lnorm}[1]{\lVert#1\rVert_2}
\newcommand{\mnorm}[1]{\lVert#1\rVert_{\rm m}}
\newcommand{\KD}[1]{{\textcolor{violet}{[Kristin: #1]}}}

\newcommand{\YL}[1]{{\textcolor{blue}{[Yuchen: #1]}}}

\title{K-stability and birational models of moduli of quartic K3 surfaces}

\author[Ascher]{Kenneth Ascher}
\address{Department of Mathematics, University of California, Irvine, CA, 92697, USA}
\email{kascher@uci.edu}
\author[DeVleming]{Kristin DeVleming}
\address{Department of Mathematics and Statistics,
University of Massachusetts, Amherst, MA 01003-9305, USA}
\email{kdevleming@umass.edu}
\author[Liu]{Yuchen Liu}
\address{Department of Mathematics, Northwestern University, Evanston, IL 60208, USA}
\email{yuchenl@northwestern.edu}
\date{\today}

\begin{document}

\maketitle

\begin{abstract}
We show that the K-moduli spaces of log Fano pairs $(\bP^3, cS)$ where $S$ is a quartic surface interpolate between the GIT moduli space of quartic surfaces and the Baily-Borel compactification of moduli of quartic K3 surfaces as $c$ varies in the interval $(0,1)$. We completely describe the wall crossings of these K-moduli spaces. As the main application, we verify  Laza-O'Grady's prediction on the Hassett-Keel-Looijenga program for quartic K3 surfaces. We also obtain the  K-moduli compactification of quartic double solids, and classify all Gorenstein canonical Fano degenerations of $\bP^3$.
\end{abstract}
\setcounter{tocdepth}{1}
\tableofcontents

\section{Introduction}

An important question in algebraic geometry is to construct geometrically meaningful compact moduli spaces for polarized K3 surfaces. The global Torelli theorem indicates that the coarse moduli space $\fM_{2d}$ of primitively polarized K3 surfaces with du Val singularities of degree $2d$ is isomorphic, under the period map, to  the arithmetic quotient $\sF_{2d}=\sD_{2d}/\Gamma_{2d}$ of a Type IV Hermitian symmetric domain $\sD_{2d}$ as the period domain. The  space $\sF_{2d}$ has a natural Baily-Borel compactification $\sF_{2d}^*$, but it is well-known that  $\sF_{2d}^*$ does not carry a nicely behaved universal family. Thus it is a natural problem to compare $\sF_{2d}^*$ with other geometric compactifications, e.g. those coming from geometric invariant theory (GIT), via the period map.

In particular, it is natural to ask if there exists a modular way to resolve the (birational) period map. When the degree $2d=2$, there is a birational period map between the GIT quotient of sextic plane curves and $\sF_{2}^*$, since a generic such K3 is the double cover of $\bP^2$ ramified along a sextic. By work of Looijenga and Shah, this map can be resolved by considering either a partial Kirwan desingularization of the GIT quotient, or via a small partial resolution of $\sF_2^*$ \cite{Sha80, Loo86}. A realization of Laza-O'Grady (based on work of Looijenga), is that an alternate systematic approach to this problem is via interpolating the Proj of $R(\sF_2, \lambda + \beta \Delta)$, where $\lambda$ is the Hodge line bundle on $\sF_2$, the divisor $\Delta$ is some geometrically meaningful Heegner divisor, and $\beta$ varies between $0$ and $1$  (see e.g. \cite{LO18b}).

When the degree $2d=4$ (for simplicity, denoted by $\fM=\fM_{4}$ and $\sF=\sF_4$), a distinguished geometric compactification is given by the GIT moduli space $\ofM^{\GIT}$ of quartic surfaces in $\bP^3$. There is a birational period map $\fp: \ofM^{\GIT}\dashrightarrow \sF^*$ with much more complicated exceptional loci as compared to the degree two case. %An explicit modular resolution of this map has remained open, and so one of the goals of this paper is to resolve this.
%As mentioned above, 
In a series of papers \cite{LO19, LO18b, LO18a}, Laza and O'Grady proposed a systematic way to resolve such period maps (when $\sF$ is an arithmetic quotient of a Type IV Hermitian symmetric domain) via a sequence of explicit birational transformations governed by the Heegner divisors in $\sF^*$, and predict that they satisfy a natural interpolation. Motivated by the Hassett-Keel program -- running the log minimal model program on $\overline{M}_g$ to interpolate between different birational models of the moduli space of curves (see e.g. \cite{hassett2013log}),
they named this program the \emph{Hassett-Keel-Looijenga program}. In \cite{LO18a}, Laza and O'Grady verified their proposal for the the moduli of hyperelliptic quartic K3 surfaces, which is an $18$-dimensional divisor in $\sF^*$, but their prediction for the $19$-dimensional space $\sF$ has remained open. One of the main purposes of this paper is to completely verify their prediction for $\sF$ using the recently constructed moduli spaces of log Fano pairs from the theory of K-stability. We note that the analogous question in the case of EPW sextics remains open, and it would be interesting to try to use K-moduli to study their compactifications.

For a rational number $c\in (0,1)$, the pair $(\bP^3, cS)$ is a log Fano pair, where $S\subset\bP^3$ is a smooth quartic surface. Thus, K-stability provides a natural framework to construct geometrically meaningful compactifications of moduli of quartic K3 surfaces. In recent years, the algebro-geometric theory of constructing projective K-moduli spaces of log Fano pairs has been completed as a combination of the important works \cite{Jia17, LWX18, CP18, BX18, ABHLX19, BLX19, Xu19, XZ19, XZ20, BHLLX20, LXZ21}. Meanwhile, when we vary the coefficient $c$, the K-moduli spaces of $\bQ$-Gorenstein smoothable log Fano pairs display  wall-crossing phenomena as established in \cite{ADL19} (see also \cite{Zho21b}). 

In this paper, we show that the K-moduli compactifications of log Fano pairs $(\bP^3, cS)$ where $S$ is a smooth quartic surface interpolate naturally between the GIT moduli space $\ofM^{\GIT}$ and the Baily-Borel compactification $\sF^*$ as $c$ varies in the interval $(0,1)$. As a result, we resolve the period map $\fp:\ofM^{\GIT}\dashrightarrow \sF^*$ where all intermediate birational models have a modular meaning as they parametrize certain K-polystable log Fano pairs. Furthermore, using the positivity of the log CM line bundle \cite{CP18, Pos19, XZ19}, we confirm the prediction by Laza and O'Grady on the Hassett-Keel-Looijenga program for moduli space of quartic K3 surfaces \cite{LO18b, LO19}. 

We first fix some notation. Let $\sM^{\circ}$ and $\fM^{\circ}$ be the Deligne-Mumford stack and coarse moduli space of quartic surfaces $S\subset \bP^3$ with du Val singularities, respectively. Let $\sF$ be the locally symmetric variety parametrizing periods of all polarized K3 surfaces of degree $4$ with du Val singularities. The global Torelli theorem implies that the period map $\fp: \fM^{\circ}\hookrightarrow \sF$ is an open immersion of quasi-projective varieties. 
Let $\sF^*$ be the Baily-Borel compactification of $\sF$. Let $\lambda$ be the Hodge line bundle over $\sF$. By \cite{LO19}, there are two Heegner divisors $H_h$ and $H_u$ of $\sF$, which  parametrize hyperelliptic and unigonal quartic K3 surfaces respectively, such that $\fp(\fM^{\circ})=\sF\setminus (H_h\cup H_u)$. 
Let $\osM^{\GIT}$ and $\ofM^{\GIT}$ be the  GIT moduli stack and space of quartic surfaces in $\bP^3$, respectively. 
\begin{thm}\label{mthm:Kmod}
For $c\in (0,1)\cap \bQ$, let $\osM^{\K}_c$ (resp. $\ofM^{\K}_c$) be the K-moduli stack (resp. K-moduli space) parametrizing K-semistable (resp. K-polystable) log Fano pairs $(X,cD)$ admitting a $\bQ$-Gorenstein smoothing to $(\bP^3, cS)$ where $S$ is a quartic surface.
\begin{enumerate}
    \item For any $c\in (0,\frac{1}{3})\cap \bQ$, there are  isomorphisms $\osM^{\K}_c\cong \osM^{\GIT}$ and $\ofM^{\K}_c\cong\ofM^{\GIT}$.
    \item For any $c\in (0,1)\cap \bQ$, the section ring $R(\sF, c\lambda+(1-c)\Delta^{\K})$ is finitely generated where $\Delta^{\K}:= \frac{1}{4}H_h+\frac{9}{8}H_u$. Moreover, there is an isomorphism  $\ofM_c^{\K}\cong \Proj ~ R(\sF, c\lambda+(1-c)\Delta^{\K})$ where the log CM line bundle on  $\ofM_c^{\K}$ is proportional to $\cO(1)$ on the $\Proj$ up to a positive constant. 
    \item For $0<\epsilon\ll 1$, the K-moduli space $\ofM_{1-\epsilon}^{\K}$ is isomorphic to Looijenga's $\bQ$-Cartierization $\widehat{\sF}$ of $\sF^*$ associated to $H_h$ and $H_u$. Moreover, the Hodge line bundle on $\ofM_{1-\epsilon}^{\K}$ is semiample and its ample model is isomorphic to $\sF^*$.
    \item There are $9$ K-moduli walls for $c\in (0,1)$. Among them, $2$ walls are divisorial contractions: contracting a strict transform of $H_h$ to the double quadric surface $[2Q]$ when $c=\frac{1}{3}$, and contracting a strict transform of $H_u$ to the tangent developable surface $[T]$ of a twisted cubic curve  when $c=\frac{9}{13}$, respectively.  The remaining $7$ walls are flips. 
\end{enumerate}

For a detailed description of the K-moduli wall crossings, see Theorem \ref{thm:k-moduli-walls}.
\end{thm}

That is, by varying the coefficient $c$, the K-moduli spaces $\ofM^K_c$ provide a natural interpolation between the GIT quotient for quartic surfaces and the Baily-Borel compactification, and explicitly resolve the period map. We note that a special case of Theorem \ref{mthm:Kmod}(1) was proved earlier in \cite[Theorem 1.2]{GMGS} and \cite[Theorem 1.4]{ADL19} (see also \cite{Zho21a}). We also note that the two walls which are divisorial contractions are actually weighted blowups of Kirwan type (see Remarks \ref{rmk:Kirwan1stpaper} and \ref{rmk:weightedblowup}).

Since these K-moduli spaces $\ofM_c^{\K}$ provide birational models of $\sF$, we are able to confirm Laza-O'Grady's prediction on the Hassett-Keel-Looijenga program for $\sF$ \cite{LO19, LO18b} by modifying $\ofM_c^{\K}$  and checking ampleness of Laza-O'Grady's line bundle. Indeed, we prove a more general finite generation result and  describe a wall-chamber structure for the full-dimensional subcone of $N^1_{\bR}(\sF)$ generated by $\lambda$, $H_h$, and $H_u$.

\begin{thm}\label{mthm:fg}
For any $a,b\in \bQ_{>0}$, the section ring $R(\sF, \lambda+\frac{1}{2}(aH_h+bH_u))$ is finitely generated, which yields a projective birational model $\sF(a,b):=\Proj ~ R(\sF, \lambda+\frac{1}{2}(aH_h+bH_u))$ of $\sF$. These $\sF(a,b)$'s have a wall-chamber structure where the walls are $a=a_i$ or $b=1$ with
\[
(a_1, a_2, \cdots, a_8)= \left(\frac{1}{9}, \frac{1}{7}, \frac{1}{6}, \frac{1}{5}, \frac{1}{4}, \frac{1}{3}, \frac{1}{2}, 1\right).
\]
Moreover, we have the following description of $\sF(a,b)$. Here we assume $0<\epsilon\ll 1$.
\begin{enumerate}
    \item If $a\in (0, \frac{1}{9})$ and $b\in (0,1)$, then $\sF(a,b)\cong \hsF$.
    \item If $a,b\in [1,+\infty)$, then $\sF(a,b)\cong \ofM^{\GIT}$.
    \item The birational map $\sF(1-\epsilon,b)\to\sF(1,b)$ is a divisorial contraction of the strict transform of $H_h$ to a point, and $\sF(1,b)\cong \sF(a,b)$ for any $a>1$.
    \item The birational map $\sF(a,1-\epsilon)\to\sF(a,1)$ is a divisorial contraction of the strict transform of $H_u$ to a point, and $\sF(a,1)\cong \sF(a,b)$ for any $b>1$.
    \item If $1\leq i\leq 7$, then birational maps $\sF(a_i-\epsilon,b)\rightarrow\sF(a_i,b)\leftarrow \sF(a_i+\epsilon,b)$ form a flip whose flipping locus (resp. flipped locus) is the strict transform of $Z^j$ (resp. of $W_{j-1})$ where $j=\begin{cases} 9-i & \textrm{if }i\geq 4;\\ 10-i & \textrm{if } i\leq 3. \end{cases}$ Here $Z^j\subset\sF$ is a tower of Shimura subvarieties of codimension $j$ (see \eqref{eq:Z^j-construct}), and $W_{i}\subset\ofM^{\GIT}$ is a tower of $i$-dimensional subvarieties (see \eqref{eq:stratification}).
\end{enumerate}

\end{thm}

\begin{cor}\label{cor:LO}
Laza-O'Grady's prediction for the $19$-dimensional locally symmetric variety $\sF$ \cite[Prediction 5.1.1]{LO19} holds.
\end{cor}

 We note that partial results toward Laza-O'Grady's prediction were obtained in \cite{LO19, LO18b}. In \cite{LO18a} the $18$-dimensional case of their prediction was confirmed (see also \cite{ADL20} for a different approach). In \cite{LO18a} the authors used an intricate and subtle variation of GIT argument, motivated by their previous arithmetic and hodge theoretic computations in \cite{LO19}.

As a consequence of Theorem \ref{mthm:Kmod}, we give an explicit description of the K-moduli space of quartic double solids, i.e. del Pezzo threefolds of degree $2$. The smooth quartic double solids are previously known to be K-stable \cite{Der16}. Note that this K-moduli space displays similar behavior to the K-moduli space of del Pezzo surfaces of degree $1$ \cite{OSS16} as both are two-step birational modifications (a blow-up followed by a small contraction) of GIT moduli spaces, while K-moduli spaces of del Pezzo threefolds/fourfolds of degree $3$ or $4$ are identical to GIT moduli spaces \cite{SS17, LX19, Liu20}.

\begin{thm}\label{mthm:doublesolid}
Let $\ofY$ be the K-moduli space of quartic double solids. Then the seminormalization of $\ofY$ is isomorphic to $\ofM_{\frac{1}{2}}^{\K}$. Moreover, it fits into the following diagram
\[
\ofM^{\GIT}\xleftarrow{\rho} \widehat{\fM}^{\GIT}\xrightarrow{\psi} \ofM_{\frac{1}{2}}^{\K}\xrightarrow{\iota} \ofY
\]
where $\rho$ is a divisorial contraction of a birational transform of $H_h$ to the point parametrizing the double quadric surface $[2Q]$, $\psi$ is a small contraction of a rational curve (the strict transform of $W_1$) to a point $p$, and $\iota$ is the seminormalization obtained by taking fiberwise double covers, where $\iota(p)$ represents the toric $\bQ$-Fano threefold $(x_2^4=x_3x_4)\subset\bP(1,1,2,4,4)_{[x_0,\cdots,x_4]}$.
\end{thm}

Another interesting consequence is a classification of all Gorenstein canonical Fano degenerations of $\bP^3$. Here $X_h$ is the projective anti-canonical cone over $\bP^1\times\bP^1$,  and $X_u$ is a Gorenstein $\bQ$-Fano threefold constructed in Section \ref{sec:construction}. Their notation is chosen so that $X_h$ (resp. $X_u$) contains a general hyperelliptic (resp. unigonal) quartic K3 surface as its anti-canonical divisor.

\begin{thm}\label{mthm:gordeg}
    Let $X$ be a Gorenstein canonical Fano variety that admits a $\bQ$-Gorenstein smoothing to $\bP^3$. Then $X$ is isomorphic to $\bP^3$, $X_h$, $\bP(1,1,2,4)$, or $X_u$.
\end{thm}

\subsection*{Sketch of proofs} We sketch the proofs of Theorems \ref{mthm:Kmod} and \ref{mthm:fg}. First of all, by \cite{GMGS, ADL19} we know that $\ofM_\epsilon^{\K}\cong \ofM^{\GIT}$.  If $S$ is a quartic surface in $\bP^3$ with semi-log canonical (slc) singularities (also called insignificant limit singularities), then $(\bP^3, S)$ is a K-semistable log Calabi-Yau pair, and $\bP^3$ is K-polystable. Hence by interpolation of K-stability, the log Fano pair $(\bP^3, cS)$ is K-semistable for any $0<c<1$.  Therefore, the birational map $\ofM_c^{\K}\dashrightarrow \ofM^{\GIT}$ is isomorphic over the open subset $\fM^{\slc}$ parametrizing quartic surfaces with slc singularities. Thus in order to describe wall-crossings of $\ofM_c^{\K}$, we only need to understand the K-polystable replacements of $\ofM^{\GIT}\setminus \fM^{\slc}$ parametrizing quartic surfaces with significant limit singularities. From the GIT of quartic surfaces \cite{Sha81, LO18b}, we know that $\ofM^{\GIT}\setminus \fM^{\slc}=W_8 \sqcup \{[T]\}$ where $T$ is the tangent developable surface of a twisted cubic curve, and $W_8$ is the largest subvariety of $\ofM^{\GIT}$ in the tower $W_i$ (see \eqref{eq:stratification}). Indeed, the K-polystable replacements of $[T]$ (resp. $W_i$) precisely correspond to unigonal (resp. hyperelliptic) quartic K3 surfaces.

In Section \ref{sec:unigonal}, we study the K-stability of $(\bP^3, cT)$ and its K-polystable replacements. Using equivariant K-stability from \cite{Zhu20}, we show that the K-semistable threshold of $(\bP^3, T)$, i.e. the largest $c$ where $(\bP^3, cT)$ is K-semistable, is equal to $\frac{9}{13}$. Then we construct the K-polystable replacement $(X_u, \frac{9}{13} T_0)$ of $(\bP^3, \frac{9}{13}T)$ by explicit birational geometry. Here $X_u$ is constructed as a particular Gorenstein $\bQ$-Fano threefold that contains all unigonal quartic K3 surfaces as anti-canonical divisors (see Section \ref{sec:construction}). Then using Paul-Tian criterion type arguments and the deformation theory of Gorenstein toric threefold singularities \cite{Altmann}, we show that the K-moduli wall crossing at $c=\frac{9}{13}$ near $[T]$ is a divisorial contraction whose exceptional divisor, birational to $H_u$, is the GIT moduli space of $(X_u, S)$ where $S$ is a Weierstrass elliptic surface.

In Section \ref{sec:hyperelliptic}, we study the K-polystable replacements of the tower $W_i$ in $\ofM^{\GIT}$. This is the trickiest part of the proof. Our motivation comes from \cite{ADL20} where we show that the K-moduli compactification $\oK_c$ of $(\bP^1\times\bP^1, cC)$ where $C\in |\cO(4,4)|$ is identical to the VGIT moduli space of slope $t=\frac{3c}{2c+2}$. By taking fiberwise double covers, we obtain a family of K-moduli spaces birational to $H_h$. However, these K-moduli spaces parametrize surface pairs rather than threefold pairs. Nevertheless, we notice that a hyperelliptic quartic K3 surface $S$ as a double cover of $\bP^1\times \bP^1$ (resp. of $\bP(1,1,2)$) naturally embeds into the cone $X_h$ (resp. $\bP(1,1,2,4)$) as an anti-canonical divisor. Then using a cone construction, a covering trick, and interpolation (see Section \ref{sec:cone} for more details), we show that $\oK_{\frac{3c-1}{4}}$ admits a closed embedding into $\ofM_c^{\K}$ for $c>\frac{1}{3}$ whose image $H_{h,c}$ is a birational transform of $H_h$ (see Theorem \ref{thm:H_h-embed}). Then we construct K-polystable replacements of $W_i$ by first embedding all of $\bP^3$, $X_h$, $\bP(1,1,2,4)$ into $\bP(1^4,2)$ as weighted hypersurfaces of degree two, and then finding a particular $1$-PS (coming from VGIT in \cite{LO18a}) that degenerates $(\bP^3, cS)$ to a K-polystable pair in $H_{h,c}$ (see Theorem \ref{thm:Kpsreplace-W}). Then we use deformation theory to classify exceptional loci after the walls. In particular, all K-moduli spaces $\ofM_c^{\K}$ are isomorphic outside of the loci $H_{h,c}$ and $H_{u,c}$. We give a complete description of all wall-crossings of $\ofM_c^{\K}$ in Theorem \ref{thm:k-moduli-walls}. 

Finally, in Section \ref{sec:proofs} we prove the main theorems. We observe that $\sF\dashrightarrow \ofM_c^{\K}$ is a birational contraction by Theorem \ref{thm:k-moduli-walls}. The upshot to show $\ofM_c^{\K}\cong \Proj ~R(\sF, c\lambda+(1-c)\Delta^{\K})$ is to use ampleness of log CM line bundles \cite{XZ19}, and to compute the variation of  log CM line bundles which interpolate between the Hodge line bundle and the absolute CM line bundle (see \eqref{eq:CM-Hodge}). Then we perform necessary gluing operations and birational modifications on $\ofM_c^{\K}$ to obtain $\sF(a,b)$ (see Definition \ref{def:F(a,b)} and Proposition \ref{prop:F(a,b)proper}), and show that the pushforward of $\lambda+\frac{a}{2}H_h +\frac{b}{2}H_u$ is ample. %using ampleness of log CM line bundles. 

\subsection*{Prior and related works}
Compactifying the moduli space $\sF_2$ of degree 2 K3 surfaces is a well studied problem. Recall that a general K3 surface of degree two can be realized as a double cover of $\bP^2$ branched along a sextic curve. As such, there is a natural birational period map between the GIT quotient of plane sextics and the Baily-Borel compactification. Shah \cite{Sha80} constructed a partial Kirwan desingularization of the GIT quotient which provides a compactification of $\sF_2$ with a set-theoretic map to $\sF_2^*$. Work of Looijenga \cite{Loo86, Loo03} shows that Shah's compatification is a $\mathbb{Q}$-factorialization of $\sF_2^*$ and additionally resolves the birational period map.  In fact, this case serves as a major motivation for the Hassett-Keel-Looijenga (HKL) program.
The case of degree 2 was revisited, in terms of Koll\'ar-Shepherd-Barron (KSB) stable pairs, by Laza \cite{laza2012ksba}, and more recently studied from the viewpoint of toroidal compactifications in work of Alexeev-Engel-Thompson \cite{AET} (see also the more recent \cite{AE21}). 

As mentioned above, the Hasset-Keel-Looijenga program was proposed by Laza-O'Grady \cite{LO19} for Type IV locally symmetric varieties associated to the lattice $U^2\oplus D_{N-2}$. It has been verified in the case of $N=18$ for hyperelliptic quartic K3 surfaces by Laza-O'Grady using variation of GIT \cite{LO18a}, and partial results for $N=19$, i.e. moduli of quartic K3 surfaces were obtained in \cite{LO19, LO18b}. 

The wall-crossing phenomenon for K-moduli spaces of log Fano pairs with varying coefficients was systematically investigated in \cite{ADL19}. One  novelty of this strategy is to naturally connect well-studied moduli spaces, such as GIT, moduli of curves, and K3 surfaces, through birational maps between a sequence of K-moduli spaces. In our previous works \cite{ADL19, ADL20}, we carried out this strategy for the Hassett-Keel-Looijenga program for $\sF_2$ and the hyperelliptic Heegner divisor $H_h$ of $\sF_4$. In this paper, we use this novel approach of wall-crossing for K-moduli to solve the problem of Laza-O'Grady (HKL program for $\sF_4$). One of the key benefits of this approach is that it gives a direct solution to a problem which was posed from an entirely different point of view, namely considering the Proj of a ring of automorphic forms and interpolating based on some arithmetic predictions.

Finally, we mention that moduli of pairs $(\bP^3, cH)$ have been studied from the point of view of KSB stable pairs by DeVleming \cite{dV} where $H$ is a surface in $\bP^3$ of degree $d\geq 5$.

%\textcolor{blue}{YL: I think it is a good idea to compare our results to other works. Maybe this can be a subsection or a remark. For instance, Shah, Looijenga, Laza-O'Grady, Kristin's thesis, Alexeev-Engel-Thompson, Alexeev-Engel, etc.}

\subsection*{Acknowledgements}
We would like to thank Dori Bejleri, Justin Lacini, Zhiyuan Li, Andrea Petracci, David Stapleton, Xiaowei Wang, and Chenyang Xu for helpful discussions, and Yuji Odaka for useful comments.  We thank the referee for their helpful comments and suggestions. The authors were
supported in part by the American Insitute of Mathematics as part of the AIM SQuaREs
program. Research of KA was supported in part by the NSF grant DMS-2140781 (formerly DMS-2001408). Research of YL was supported in part by the NSF grant DMS-2148266 (formerly DMS-2001317). 

\section{Preliminaries on K-stability and K-moduli}

Throughout this paper, we work over the field of complex numbers $\bC$. We refer the background of singularities of pairs, such as Kawamata log terminal (klt), purely log terminal (plt), log canonical (lc), and semi-log canonical (slc), to the standard references \cite{KM98, Kol13}. All schemes are assumed to be of finite type over $\bC$.

\subsection{K-stability}

\subsubsection{Fujita-Li's valuative criteria}
\begin{defn}
Let $X$ be a normal variety. Let $D$ be an effective $\bQ$-divisor on $X$. We say $(X,D)$ is a \emph{log pair} if $K_X+D$ is $\bQ$-Cartier. A log pair $(X,D)$ is called \emph{log Fano} if $X$ is projective and $-K_X-D$ is ample. A log pair $(X,D)$ is called \emph{log Calabi-Yau} if $X$ is projective and $K_X+D\sim_{\bQ} 0$. If $(X,0)$ is a klt log Fano pair, then we call $X$ a \emph{$\bQ$-Fano variety}.
\end{defn}

We refer the definitions of test configurations and  K-(poly/semi)stability of log Fano pairs to  \cite[Section 2.1]{ADL19}. Here we use Fujita-Li's valuative criteria as alternative definitions.

\begin{defn}
Let $(X,D)$ be a log pair. We call $E$ a \emph{prime divisor over} $X$ if there is a proper birational morphism $\mu: Y\to X$ from a normal variety $Y$ such that $E$ is a prime divisor on $Y$. We define the \emph{log discrepancy} of $E$ with respect to $(X,D)$ as
\[
A_{(X,D)}(E):= 1+ \textrm{coeff}_E(K_Y-\mu^*(K_X+D)).
\]
If, in addition, $(X,D)$ is a log Fano pair, then we define the \emph{pseudo-effective threshold}, the  \emph{expected vanishing order} (also known as the \emph{$S$-functional}), and the \emph{$\beta$-invariant} of $E$ with respect to $(X,D)$ as
\begin{align*}
T_{(X,D)}(E) & := \sup \{t\in \bR_{\geq 0}\mid \mu^*(-K_X-D)-tE\textrm{ is big}\},\\ 
S_{(X,D)}(E)&:=\frac{1}{\vol_X(-K_X-D)}\int_0^{T_{(X,D)}(E)} \vol_X(-K_X-D-tE)dt,\\
\beta_{(X,D)}(E)&:=A_{(X,D)}(E)-S_{(X,D)}(E).
\end{align*}
Here $\vol_X(-K_X-D-tE):=\vol_Y(\mu^*(-K_X-D)-tE)$. The \emph{$\alpha$-invariant} and the \emph{stability threshold} (also known as the \emph{$\delta$-invariant}) of a klt log Fano pair $(X, D)$ are defined as
\[
\alpha(X,D):= \inf_{E} \frac{A_{(X,D)}(E)}{T_{(X,D)}(E)}\quad \textrm{and}\quad
\delta(X,D):= \inf_{E} \frac{A_{(X,D)}(E)}{S_{(X,D)}(E)},
\]
where the infima run over all prime divisors $E$ over $X$. 
\end{defn}

Next, we recall Fujita-Li's valuative criteria for K-(semi)stability and uniform K-stability.

\begin{thm}[\cite{Fuj19, Li17, BX18}, see also \cite{FO16, BJ17}]\label{thm:valuative}
    A klt log Fano pair $(X,D)$ is
    \begin{enumerate}
        \item K-semistable if and only if $\beta_{(X,D)}(E)\geq 0$ for any prime divisor $E$ over $X$, or equivalently, $\delta(X,D)\geq 1$;
        \item K-stable if and only if $\beta_{(X,D)}(E)> 0$ for any prime divisor $E$ over $X$;
        \item uniformly K-stable if and only if $\delta(X,D)>1$.
    \end{enumerate}
\end{thm}

Note that a K-semistable log Fano pair is always klt by \cite{Oda13b}. By a recent result of Liu-Xu-Zhuang \cite{LXZ21}, K-stability is equivalent to uniform K-stability  for any klt log Fano pair.

If $X$ is a $\bQ$-Fano variety, $D$ is an effective $\bQ$-Cartier $\bQ$-divisor on $X$, and $c\in\bQ_{>0}$, then we say $(X,D)$ is \emph{$c$-K-(poly/semi)stable} if $(X,cD)$ is a K-(poly/semi)stable log Fano pair.

\subsubsection{Special degenerations and plt blow-ups}

Recall that a test configuration $(\cX,\cD;\cL)$ of a klt log Fano pair $(X,D)$ is called \emph{special} if $(\cX,\cX_0+\cD)$ is plt and $\cL\sim_{\bQ}-l(K_{\cX/\bA^1}+\cD)$ for some $l\in \bZ_{>0}$. In this case, we call $(\cX_0,\cD_0)$ a \emph{special degeneration} of $(X,D)$ and denote by $(X,D)\rightsquigarrow (\cX_0,\cD_0)$. By adjunction we know that $(\cX_0,\cD_0)$ is also a klt log Fano pair. 

Next, we recall a result of Li-Wang-Xu which gives a characterization for K-polystability in terms of special degenerations.

\begin{thm}[\cite{LWX18}]\label{thm:lwx-polystable}
    A K-semistable log Fano pair $(X,D)$ is K-polystable if and only if any K-semistable special degeneration of $(X,D)$ is isomorphic to itself.
\end{thm}

\begin{defn}\label{defn:specialplt}
Let $(X,D)$ be a klt log pair.
Let $E$ be a prime divisor over $X$. 
\begin{enumerate}

    \item (\cite{Fuj19a}) We say $E$ is of \emph{plt type} over $(X,D)$ if there exists a birational morphism $\mu: Y\to X$ from a normal projective variety $Y$ such that
    \begin{itemize}
        \item $E$ is a $\bQ$-Cartier prime divisor on $Y$, and $-E$ is $\mu$-ample;
        \item $\mu|_{Y\setminus E}: Y\setminus E\to X\setminus \mu(E)$ is an isomorphism;
        \item $(Y, E+\mu_*^{-1} D)$ is plt.
    \end{itemize}
    Such a morphism $\mu$ is called a \emph{plt blow-up}.
    \item \cite{Pro00, Xu14} A plt type divisor $E$ over $(X,D)$ with center $\mu(E)=x$ being a closed point is called a \emph{Koll\'ar component} over the singularity $x\in (X,D)$. 
    \item We say $E$ is a \emph{special divisor} over $(X,D)$ if there exists a special test configuration $(\cX,\cD;\cL)$ of $(X,D)$ and a positive integer $d$ such that $\ord_{\cX_0}|_{\bC(X)}=d\cdot \ord_E$.
    %Note that $E$ is not necessarily exceptional over $X$.
\end{enumerate}
 
\end{defn}

\begin{lem}[Zhuang]\label{lem:specialplt}
Any special divisor over a klt log Fano pair $(X,D)$ is of plt type.
\end{lem}

\begin{proof}
By Zhuang's Theorem \cite[Theorem 4.12]{Xu20}, if $E$ is a special divisor over $(X,D)$, then there exists a $\bQ$-complement $D^+$ of $(X,D)$ such that $E$ is the only lc place of $(X,D^+)$. Thus for $0<\epsilon\ll 1$, the log pair $(X, (1-\epsilon)D+\epsilon D^+)$ is klt where $E$ has log discrepancy less than $1$. Thus by \cite{BCHM10} there exists a birational morphism $\mu: Y\to X$ from a normal projective variety $Y$ such that the first two conditions in Definition \ref{defn:specialplt}(2) hold. Moreover, since $E$ is the only lc place of $(X,D^+)$, we know that $(Y, E+\mu_*^{-1} D^+)$ is plt. Thus $(Y, E+\mu_*^{-1} D)$ is also plt. 
\end{proof}

\subsubsection{Equivariant K-stability}

The following theorem is essentially due to Zhuang \cite{Zhu20}. The equivalence between (i) and (iii) was also proved by Fujita \cite{Fuj19a} when $G$ is trivial.

\begin{thm}[Zhuang]\label{thm:kss-plt}
    Let $(X,D)$ be a klt log Fano pair. Let $G$ be an algebraic group acting on $(X,D)$. Then the following are equivalent.
    \begin{enumerate}[label=(\roman*)]
        \item $(X,D)$ is K-semistable;
        \item $\beta_{(X,D)}(E)\geq 0$ for any $G$-invariant special divisor $E$ over $(X,D)$;
        \item $\beta_{(X,D)}(E)\geq 0$ for any $G$-invariant prime divisor $E$ of plt type over $(X,D)$.
    \end{enumerate}
\end{thm}

\begin{proof}
The (i)$\Rightarrow$(iii) part is a consequence of Theorem \ref{thm:valuative}. The (iii)$\Rightarrow$(ii) part follows from Lemma \ref{lem:specialplt}. So we focus on the (ii)$\Rightarrow$(i) part. 
Assume to the contrary that $(X,D)$ is K-unstable. By \cite[Theorem 1.2]{LXZ21} and \cite[Theorem 1.2(2)]{BHLLX20}, there exists a non-trivial special test configuration $(\cX,\cD)$ of $(X,D)$ that minimizes the bi-valued invariant $\left( \frac{\Fut(\cX , \cD )}{ \mnorm {\cX,\cD} }, \frac{\Fut(\cX , \cD )}{ \lnorm{\cX,\cD}}\right)$ under the lexicographic order among all special test configurations, where $\mnorm{\cdot}$ and $\lnorm{\cdot}$ represents the minimum norm and the $L^2$-norm respectively (see \cite[Section 2.3]{BHLLX20} for definitions). Moreover, such a minimizing special test configuration $(\cX,\cD)$ is unique up to rescaling. The minimizing property implies $\Fut(\cX,\cD)<0$ as $(X,D)$ is K-unstable. Every $g\in G$ induces a pull-back $(\cX_g, \cD_g)$ of the test configuration $(\cX,\cD)$ with the same $\Fut(\cdot)$, $\mnorm{\cdot}$, and $\lnorm{\cdot}$. Since a non-trivial rescaling must change the norms, the uniqueness of $(\cX,\cD)$ implies that $(\cX_g,\cD_g)\cong (\cX,\cD)$ for every $g\in G$. Thus $(\cX,\cD)$ is $G$-equivariant. %Let $(\cX,\cD;\cL)$ be any non-trivial $G$-equivariant special test configuration of $(X,D)$. 
Let $E$ be a special divisor over $(X,D)$ such that $\ord_{\cX_0}|_{\bC(X)}=d\cdot \ord_E$ for some $d\in \bZ_{>0}$. Then $\Fut(\cX,\cD)=d\cdot\beta_{(X,D)}(E)\geq 0$ by \cite{Fuj19}, a contradiction to $\Fut(\cX,\cD)<0$. Thus $(X,D)$ is K-semistable. 
%The equivariant version of \cite{LX14} implies that $(X,D)$ is $G$-equivariantly K-semistable, hence is K-semistable by \cite{Zhu20}.
\end{proof}

\begin{prop}\label{prop:G-kps}
Let $(X,D)$ be a K-semistable, but not K-polystable, log Fano pair. Let $G$ be a reductive group acting on $(X,D)$. Then there exists a $G$-invariant special divisor $E$ over $(X,D)$ with $\beta_{(X,D)}(E)=0$ which induces a K-polystable special degeneration of $(X,D)$. 
\end{prop}

\begin{proof}
By \cite[Theorem 1.3]{LWX18} let $(X_0,\Delta_0)$ be the unique K-polystable special degeneration of $(X,D)$. Then the second paragraph of \cite[Proof of Corollary 4.11]{Zhu20} implies that there exists a non-trivial $G$-equivariant special test configuration $(\cX,\cD)$ of $(X,D)$ such that $(\cX_0,\cD_0)\cong (X_0,D_0)$. Since $(X_0,D_0)$ is K-polystable, we have $\Fut(\cX,\cD)=0$. From \cite{Fuj19} we know that there exists a special divisor $E$ over $(X,D)$ and $d\in \bZ_{>0}$ such that $\ord_{\cX_0}|_{\bC(X)}=d\cdot\ord_E$, and $0=\Fut(\cX,\cD)=d\cdot \beta_{(X,D)}(E)$. Since $(\cX,\cD)$ is $G$-equivariant, we know that $E$ is $G$-invariant. Thus the proof is finished.
%By \cite{Zhu20} we know that $(X,D)$ is $G$-equivariantly K-semistable but not $G$-equivariantly K-polystable.
%Using the $G$-equivariant version of \cite{LX14}, we can generalize \cite[Theorem 1.3]{LWX18} to the $G$-equivariant setting, i.e. there exists a non-trivial $G$-equivariant special test configuration $(\cX,\cD;\cL)$ of $(X,D)$ such that $(\cX_0,\cD_0)$ is $G$-equivariantly K-polystable and $\Fut(\cX,\cD;\cL)=0$. Hence $(\cX_0,\cD_0)$ is K-polystable by \cite{Zhu20}. From \cite{Fuj19} we know that there exists a special divisor $E$ over $(X,D)$ and $d\in \bZ_{>0}$ such that $\ord_{\cX_0}|_{\bC(X)}=d\cdot\ord_E$, and $0=\Fut(\cX,\cD;\cL)=d\cdot \beta_{(X,D)}(E)$. Thus the proof is finished.
\end{proof}

\subsubsection{Almost log Calabi-Yau pairs}
The following definition is equivalent to the original definition by \cite{Oda13b}.

\begin{defn}
A log Calabi-Yau pair $(X,D)$ is called \emph{K-semistable} if $(X,D)$ is log canonical. 
\end{defn}

% \begin{lem}
% Let $X$ be a $\bQ$-Fano variety of dimension at least $2$. Let $D\sim_{\bQ} -K_X$ be an effective $\bQ$-Cartier Weil divisor. Assume that $(X,D)$ is plt. Let $c\in (0,1)$ be a rational number. Then $(X,cD)$ is K-polystable if and only if $(X,cD)$ is K-stable.\footnote{YL: Consider replace by the next thm.}
% \end{lem}

% \begin{proof}
% The ``if'' part is obvious, so we only need to prove the ``only if'' part. Since $D\sim_{\bQ} -K_X$ is ample and $\dim(X)\geq 2$, we know that $\Supp(D)$ is connected. Thus $D$ is a prime divisor on $X$ by the plt assumption on $(X,D)$. Assume to the contrary that $(X,cD)$ is K-polystable but not K-stable. Then by definition, there exists a non-trivial product test configuration $(\cX,c\cD)/\bA^1$ of $(X,cD)$ such that $\Fut(\cX,c\cD)=0$. Since $(X,D)$ is plt, \cite[Proposition 2.13]{ADL19} implies that $(X,tD)$ is K-polystable for any rational number $t\in (c,1)$. Hence we have $\Fut(\cX, t\cD)=0$. According to \cite{Fuj19}, the test configuration $\cX$ induces a prime divisor $E$ over $X$ such that $\beta_{(X,tD)}(E)=0$ for any $t\in (c,1)$. Since $A_{(X,tD)}(E)=A_X(E)-t\ord_E(D)$ and $S_{(X,tD)}(E)=(1-t)S_X(E)$, by letting $t\to 1$ we get $A_{(X,D)}(E)=0$.  Since $(X,D)$ is plt and $D$ is irreducible, we know that $E=D$. Hence simple computation shows that 
% \[
% A_{(X,cD)}(D)=1-c>\frac{1-c}{\dim(X)+1} =S_{(X,cD)}(D).
% \]
% This is a contradiction. Thus the proof is finished.
% \end{proof}

The following theorem can be viewed as an algebraic analogue of \cite[Corollary 1]{JMR16}. A different proof can be obtained by applying \cite[Lemma 5.3]{Zho21b}.

\begin{thm}\label{thm:uKs-almost-logCY}
Let $X$ be a $\bQ$-Fano variety of dimension $n\geq 2$. Let $D\sim_{\bQ} -K_X$ be an effective $\bQ$-Cartier Weil divisor. Assume that $(X,D)$ is plt. Then there exists $\epsilon_1\in (0,1)$ depending only on $n$, such that for any rational number $c\in (1-\epsilon_1,1)$, the log Fano pair $(X,cD)$ is uniformly K-stable. In particular, $\Aut(X,D)$ is a finite group.
\end{thm}

\begin{proof}
Consider the pair $(X, \frac{m-1}{m}D)$ where $m\in \bZ_{>0}$. By \cite[Corollary 3.5]{BLX19} based on Birkar's boundedness of complements \cite[Theorem 1.8]{Bir19}, there exists $N\in \bZ_{>0}$ depending only on $n$, such that either $(X, \frac{m-1}{m}D)$ is uniformly K-stable, or 
\begin{equation}\label{eq:almost-logCY}
    \delta(X, \frac{m-1}{m}D)=\inf_E \frac{A_{(X, \frac{m-1}{m}D)}(E)}{S_{(X, \frac{m-1}{m}D)}(E)},
\end{equation}
where $E$ runs over lc places of $N$-complements $\Delta_m^+$ of $(X, \frac{m-1}{m}D)$ satisfying $\Delta_m^+\geq \frac{m-1}{m}D$. Since $N\Delta_m^+$ is a Weil divisor, we know that $\Delta_m^+\geq D$ as long as $m> N$. Since $\Delta_m^+\sim_{\bQ}-K_X\sim_{\bQ} D$, we have that $\Delta_m^+=D$ for $m>N$. 

We claim that $(X, \frac{N}{N+1}D)$ is uniformly K-stable. If not, from $n\geq 2$ and ampleness of $D$ we know that $D$ is connected. Since $\Delta_{N+1}^+=D$ and $(X,D)$ is plt, the prime divisor $D$ is the only lc place of $N$-complements $\Delta_{N+1}^+$. Hence \eqref{eq:almost-logCY} implies that $\delta(X,\frac{N}{N+1}D)$ is computed by $E=D$, and simple computation shows that
\[
A_{(X,\frac{N}{N+1}D)}(D)=\frac{1}{N+1} \quad\textrm{and} \quad S_{(X,\frac{N}{N+1}D)}(D)=\frac{1}{(n+1)(N+1)}.
\]
This contradicts the assumption that $\delta(X,\frac{N}{N+1}D)\leq 1$. Thus we prove the claim which implies the uniform K-stability of $(X,cD)$  for any $c\in [\frac{N}{N+1},1)$ by \cite[Proposition 2.13]{ADL19}. The finiteness of $\Aut(X,D)$ follows from \cite[Corollary 1.3]{BX18}.
\end{proof}

\subsection{CM line bundles}
The CM line bundle of a flat family of polarized projective varieties is a functorial line bundle over the base which was introduced algebraically by Tian \cite{Tia97}.
The following definition of CM line bundles is due to Paul and Tian \cite{PT06, PT09} using the Knudsen-Mumford expansion (see also \cite{FR06}). We use the concept of relative Mumford divisors from \cite{Kol18, Kol19}; see also \cite[Definition 2.7]{ADL20}.

\begin{defn}[log CM line bundle]\label{defn:logCM}
Let $f:\cX\to B$ be a proper flat morphism of connected schemes. Assume that $f$ has $S_2$ fibers of pure dimension $n$. Let $\cL$ be an $f$-ample line bundle on $\cX$. Let $\cD:= \sum_{i=1}^k c_i\cD_i$ be a relative Mumford $\bQ$-divisor on $\cX$ over $B$ where each $\cD_i$ is a relative Mumford divisor and $c_i \in [0, 1] \cap \bQ$ . We also assume that each $\cD_i$ is flat over $B$ (see Remark \ref{rmk:flatness}).

%Let $\cD_i$ ($i\in \{1,2,\cdots, k\}$) be a closed subscheme of $\cX$ such that $f|_{\cD_i}:\cD_i\to T$ is flat of pure dimension $n-1$. Let $c_i\in[0,1]$ be rational numbers.

A result of Knudsen-Mumford \cite{KM76} says that there exist line bundles $\lambda_j=\lambda_{j}(\cX,\cL)$ on $B$ such that for all $k$,
\[
\det f_!(\cL^k)=\lambda_{n+1}^{\binom{k}{n+1}}\otimes\lambda_n^{\binom{k}{n}}\otimes\cdots\otimes\lambda_0.
\]
By flatness, the Hilbert polynomial $\chi(\cX_b,\cL_b^k)=a_0 k^n+a_1 k^{n-1}+ O(k^{n-2})$ for any $b\in B$. 
Then the \emph{CM line bundle} and the \emph{Chow line bundle} of the data $(f:\cX\to B,\cL)$ are defined as
\[
\lambda_{\CM,f,\cL}:=\lambda_{n+1}^{\mu+n(n+1)}\otimes\lambda_n^{-2(n+1)},\qquad \lambda_{\Chow,f,\cL}:=\lambda_{n+1}.
\]
where $\mu:=\frac{2a_1}{a_0}$.
The \emph{log CM $\bQ$-line bundle} of the data $(f:\cX\to B, \cL,\cD)$ is defined as
\[
 \lambda_{\CM,f,\cD,\cL}:=\lambda_{\CM,f,\cL}-\frac{n(\cL_b^{n-1}\cdot\cD_b)}{(\cL_b^n)}\lambda_{\Chow,f,\cL}+(n+1)\lambda_{\Chow,f|_{\cD},\cL|_{\cD}},
\]
where $(\cL_b^{n-1}\cdot\cD_b):=\sum_{i=1}^k c_i (\cL_b^{n-1}\cdot\cD_{i,b})$ and $\lambda_{\Chow,f|_{\cD},\cL|_{\cD}}:=\bigotimes_{i=1}^k\lambda_{\Chow,f|_{\cD_i},\cL|_{\cD_i}}^{\otimes c_i}$. 
\end{defn}

\begin{remark}\label{rmk:flatness}
In Definition \ref{defn:logCM} we assumed that each $\calD_i$ are flat over $B$. This is guaranteed in our setting -- $\mathbb{Q}$-Gorenstein smoothable log Fano families over reduced base schemes -- by \cite[Proposition 2.12]{ADL20}. 
\end{remark}

% \begin{defn}\label{defn:qgorfamily} Let $f:\cX\to T$ be a proper flat morphism between normal varieties. Let $\cD$ be an effective $\bQ$-divisor on $\cX$. We say $f:(\cX,\cD)\to T$ is a \emph{$\bQ$-Gorenstein flat family of log Fano pairs} if the following conditions hold:
% \begin{itemize}
% \item $f$ has normal, connected fibers;
% \item $\Supp(\cD)$ does not contain any fiber;
% \item $-(K_{\cX/T}+\cD)$ is $\bQ$-Cartier and $f$-ample.
% \end{itemize}
% We define the \emph{CM $\bQ$-line bundle} of $f:(\cX,\cD)\to T$ to be  $\lambda_{\CM,f,\cD}:=l^{-n}\lambda_{\CM,f,\cD,\cL}$, where $\cL:=-l(K_{\cX/T}+\cD)$ is an $f$-ample Cartier divisor on $\cX$ for some $l\in\bZ_{>0}$. 
% \end{defn}

\subsection{K-moduli of $\bQ$-Fano varieties}
We first recall the moduli stack of $\bQ$-Fano varieties. 

\begin{defn}
A \emph{$\bQ$-Fano family} is a morphism $f:\cX\to B$ between schemes such that 
\begin{enumerate}
    \item $f$ is projective and flat of pure relative dimension $n$ for some positive integer $n$;
    \item the geometric fibers of $f$ are $\bQ$-Fano varieties;
    \item $-K_{\cX/B}$ is $\bQ$-Cartier and $f$-ample;
    \item $f$ satisfies Koll\'ar's condition.
\end{enumerate}
\end{defn}

We recall the following definition from \cite{BHLLX20}. 

\begin{definition}\cite[Section 4.1]{BHLLX20}
Let $n$ be a positive integer and $V$ a positive rational number. We define the moduli pseudo-functor $\cM_{n,V}^{\rm Fano}$ that sends a scheme $B$ to 
\[
\cM_{n,V}^{\rm Fano}(B):= \{\textrm{$\bQ$-Fano families } f:\cX\to B \mid \dim(X_b)=n,~(-K_{X_b})^n=V \textrm{ for all }b\in S  \},
\]

Fix an $0<\epsilon \leq 1$, and let 
${ \cM_{n,V}^{\delta \geq\epsilon } \subseteq \cM^{\rm Fano}_{n,V}}$ denote the subfunctor
defined by 
\[
 \cM^{\delta \geq \epsilon }_{n,V}(B) := \{ [\cX \to B] \in \  \cM^{\rm Fano}_{n,V}(B) \, \vert \, \delta(\cX_{\overline{b}}) \geq \epsilon \, \text{ for all }\, b\in B \}.\]
 We also define 
 \[
  \cM^{\rm Kss}_{n,V}(B) := \{ [\cX \to B] \in \  \cM^{\rm Fano}_{n,V}(B) \, \vert \, \cX_{\overline{b}}\textrm{ is K-semistable for all }\, b\in B \}.
 \]
 Then by Theorem \ref{thm:valuative} we know that $\cM_{n,V}^{\rm Kss} = \cM^{\delta \geq 1 }_{n,V}$.

 By \cite[Section 4.1]{BHLLX20}, the pseudo-functor $\cM^{\delta \geq \epsilon}_{n,V}$ is represented by an Artin stack of finite type with affine diagonal (indeed, a quotient stack $[Z/\PGL_{m+1}]$ where $Z$ is a quasi-projective scheme). The CM $\bQ$-line bundle on $\cM_{n,V}^{\delta \geq \epsilon}$ is defined as the CM $\bQ$-line bundle of its universal family.
\end{definition}

By \cite{BL18b, BLX19}, we know that for a $\bQ$-Fano family $\cX\to B$, the function $b\mapsto \min\{1, \delta(\cX_{\bar{b}})\}$ is constructible and lower semi-continuous. Thus for any $0<\epsilon <\epsilon'\leq 1$ there are canonical open immersions $\cM_{n,V}^{\delta \geq \epsilon'}\hookrightarrow \cM_{n,V}^{\delta \geq \epsilon}$.

%This pseudo-functor is represented by an Artin stack (c.f. \cite[8.3]{Kol17} for general background on moduli of strongly polarized varieties). 

%It follows from \cite{Kol17} that $\cM_{n,V}^{\rm Fano}$ is represented by an Artin stack locally of finite type, which we also denote by $\cM_{n,V}^{\rm Fano}$ as abuse of notation. The CM $\bQ$-line bundle on $\cM_{n,V}^{\rm Fano}$ is defined as the CM $\bQ$-line bundle of its universal family.

The following result, known as the K-moduli theorem, is a combination of many recent important algebraic works  \cite{Jia17, LWX18, CP18, BX18, ABHLX19, BLX19, Xu19, XZ19, XZ20, BHLLX20, LXZ21}.

\begin{thm}[K-moduli theorem]\label{thm:k-moduli}
    Let $n$ be a positive integer and $V$ a positive rational number. Then there exists an Artin stack $\cM_{n,V}^{\rm Kss}$ of finite type with affine diagonal parametrizing K-semistable $\bQ$-Fano varieties of dimension $n$ and volume $V$. Moreover, $\cM_{n,V}^{\rm Kss}$ admits a projective good moduli space $M_{n,V}^{\rm Kps}$ parametrizing K-polystable $\bQ$-Fano varieties, and the CM $\bQ$-line bundle on $\cM_{n,V}^{\rm Kss}$ descends to an ample $\bQ$-line bundle on $M_{n,V}^{\rm Kps}$.
\end{thm}

We call $\cM_{n,V}^{\rm Kss}$ and $M_{n,V}^{\rm Kps}$ a \emph{K-moduli stack} and a \emph{K-moduli space}, respectively.

In this paper, we are mainly interested in the $\bQ$-Gorenstein smoothable case. 

\begin{defn}
%Let $n$ and $V$ be positive integers. Let $\cM_{n,V}^{\rm sm}$ be the open substack of $\cM_{n,V}^{\rm Fano}$ that parametrizes smooth Fano manifolds of dimension $n$ and volume $V$. Let $\ocM_{n,V}^{\rm sm}$ be the Zariski closure of $\cM_{n,V}^{\rm sm}$ in $\cM_{n,V}^{\rm Fano}$ with reduced structure.

Let $n$ and $V$ be positive integers, and fix any $0 < \epsilon \leq 1$. Let $\cM_{n,V}^{{\rm sm}, \delta\geq \epsilon}$ be the open substack of $\cM^{\delta \geq \epsilon}_{n,V}$ parametrizing smooth Fano varieties $X$ of dimension $n$ and volume $V$ with $\delta(X) \geq \epsilon$. Let $\ocM_{n,V}^{{\rm sm}, \delta \geq \epsilon}$ be the Zariski closure of $\cM_{n,V}^{{\rm sm}, \delta \geq \epsilon}$ in  $\cM^{\delta \geq \epsilon}_{n,V}$ with reduced structure.  We call $\ocM_{n,V}^{{\rm sm}, \delta\geq \epsilon}$ a \emph{moduli stack of $\bQ$-Gorenstein smoothable $\bQ$-Fano varieties}. A $\bQ$-Fano variety $X$ is called \emph{$\bQ$-Gorenstein smoothable} if $[X]\in \ocM_{n,V}^{{\rm sm},\delta\geq \epsilon}$ for some $n,V,\epsilon$.
\end{defn}

%By the Kodaira-Nakano vanishing theorem and boundedness of Fano manifolds, it is clear that $\cM_{n,V}^{\rm sm}$ is a smooth Artin stack of finite type: with boundedness and the Kodaira-Nakano vanishing theorem, we may construct $\cM_{n,V}^{\rm sm}$ using the theory of strongly polarized varieties in \cite[Section 8.3, Corollary 8.23]{Kol17}.  We call $\ocM_{n,V}^{\rm sm}$ a \emph{moduli stack of $\bQ$-Gorenstein smoothable $\bQ$-Fano varieties}. A $\bQ$-Fano variety $X$ is called \emph{$\bQ$-Gorenstein smoothable} if $[X]\in \ocM_{n,V}^{\rm sm}$ for some $n,V$.
%\end{defn}

Let $\ocM_{n,V}^{\rm sm, Kss}:=\ocM_{n,V}^{{\rm sm}, \delta \geq 1}$ be a reduced closed substack of $\cM_{n,V}^{\rm Kss}$. %Since $\delta \geq 1$ (see Theorem \ref{thm:valuative}) for any object parametrized by $\cM_{n,V}^{\rm Kss}$, we see that our definition does not depend on the choice of $\epsilon \leq 1$ from above.
According to Theorem \ref{thm:k-moduli}, the stack $\ocM_{n,V}^{\rm sm,Kss}$ admits a projective good moduli space $\oM_{n,V}^{\rm sm,Kps}$ as a reduced closed subscheme of $M_{n,V}^{\rm Kps}$. Note that prior to the algebraic approach in Theorem \ref{thm:k-moduli}, it was shown using analytic methods that there exists a proper good moduli space $\oM_{n,V}^{\rm sm,Kps}$ of $\ocM_{n,V}^{\rm sm,Kss}$ by \cite{LWX19} (see also \cite{Oda15}).

\subsubsection{$\bQ$-Gorenstein smoothable log Fano pairs}

We will consider the following class of pairs. 

\begin{definition}\label{defn:qgorsmoothable}
 Let $c,r$ be positive rational numbers such that $c<\min\{1, r^{-1}\}$. 
 A log Fano pair $(X,cD)$ is \emph{$\bQ$-Gorenstein smoothable} if there exists a $\bQ$-Fano family $\pi:\cX\to C$ over a pointed smooth curve $(0\in C)$ and a relative Mumford divisor $\cD$ on $\cX$ over $C$ such that the following holds:
 \begin{itemize}
  \item $\cD$ is $\bQ$-Cartier, $\pi$-ample, and  $\cD\sim_{\bQ,\pi}-rK_{\cX/C}$;
  \item Both $\pi$ and $\pi|_{\cD}$ are smooth morphisms
  over $C\setminus\{0\}$;
  \item $(\cX_0,c\cD_0)\cong (X,cD)$, in particular $X$ has klt singularities.
 \end{itemize} 
 A \emph{$\bQ$-Gorenstein smoothable log Fano family} $f:(\cX,c\cD)\to B$ over a reduced scheme $B$ consists of a $\bQ$-Fano family $f:\cX\to B$ and a $\bQ$-Cartier relative Mumford divisor $\cD$ on $\cX$ over $B$, such that 
 all fibers $(\cX_b,c\cD_b)$ are $\bQ$-Gorenstein smoothable log Fano pairs, and $\cD\sim_{\bQ,f} -rK_{\cX}$.
\end{definition}

For a $\bQ$-Gorenstein smoothable log Fano family, we  define its Hodge line bundle as follows.

\begin{defn}\label{defn:hodge}
For $c,r\in \bQ_{>0}$ with $cr<1$, let $f:(\cX,c\cD)\to B$ be a $\bQ$-Gorenstein smoothable log Fano family over a reduced scheme $B$ where  $\cD\sim_{\bQ,f} -rK_{\cX/B}$. The \emph{Hodge $\bQ$-line bundle} $\lambda_{\Hdg,f,r^{-1}\cD}$ is defined as the $\bQ$-linear equivalence class of $\bQ$-Cartier $\bQ$-divisors on $T$ such that
\[
K_{\cX/B}+r^{-1}\cD\sim_{\bQ}f^*\lambda_{\Hdg,f,r^{-1}\cD}.
\]
\end{defn}

In \cite{ADL19} we define the Artin stacks $\cK\cM_{\chi_0, r, c}$ and prove that they admit proper good moduli spaces $KM_{\chi_0, r,c}$, where the projectivity of such K-moduli spaces is proven by Xu and Zhuang \cite{XZ19}. Note that the K-moduli theorem also holds for all log Fano pairs without the $\bQ$-Gorenstein smoothable assumption as a generalization of Theorem \ref{thm:k-moduli} (see e.g. \cite[Theorem 1.3]{LXZ21}), though we restrict to the $\bQ$-Gorenstein smoothable case in this article.

 \begin{thm}[{\cite[Theorem 3.1 and Remark 3.25]{ADL19} and \cite{XZ19}}]\label{thm:ADL19-moduli}
 Let $\chi_0$ be the Hilbert polynomial of an anti-canonically polarized Fano manifold. Fix $r\in\bQ_{>0}$ and a rational number $c\in (0,\min\{1,r^{-1}\})$. Consider the following moduli pseudo-functor over reduced schemes $B$:
\[
\cK\cM_{\chi_0,r,c}(B)=\left\{(\cX,\cD)/B\left| \begin{array}{l}(\cX,c\cD)/B\textrm{ is a $\bQ$-Gorenstein smoothable log Fano family,}\\ \cD\sim_{B,\bQ}-rK_{\cX/B},~\textrm{each fiber $(\cX_b,c\cD_b)$ is K-semistable,}\\ \textrm{and $\chi(\cX_b,\cO_{\cX_b}(-kK_{\cX_b}))=\chi_0(k)$ for $k$ sufficiently divisible.}\end{array}\right.\right\}.
\]
Then there exists a reduced Artin stack  $\cK\cM_{\chi_0,r,c}$ (called a \emph{K-moduli stack}) of finite type over $\bC$ representing the above moduli pseudo-functor. In particular, the $\bC$-points of $\cK\cM_{\chi_0,r,c}$ parametrize K-semistable $\bQ$-Gorenstein smoothable log Fano pairs $(X,cD)$ with Hilbert polynomial $\chi(X,\cO_X(-mK_X))=\chi_0(m)$ for sufficiently divisible $m$ and $D\sim_{\bQ}-rK_X$.

Moreover, the Artin stack $\cK\cM_{\chi_0,r,c}$ admits a good moduli space $KM_{\chi_0,r,c}$ (called a \emph{K-moduli space}) as a projective reduced scheme of finite type over $\bC$, whose closed points parametrize K-polystable log Fano pairs, and the CM $\bQ$-line bundle on $\cK\cM_{\chi_0,r,c}$  descends to an ample $\bQ$-line bundle on $KM_{\chi_0,r,c}$.
\end{thm}

\begin{lem}\label{lem:forgetful-smooth}
Let $n$ and $V$ be positive integers. Let $r$ be a positive rational number. Let $\chi_0$ be the Hilbert polynomial of an anti-canonically polarized Fano manifold of dimension $n$ and volume $V$. 
Then there exists $\epsilon_0\in (0,1]$ depending only on $n$ and $r$ such that the forgetful map $\cK\cM_{\chi_0,r,c}\to \ocM^{{\rm sm},\delta\geq \epsilon_0}$ that assigns $[(X,D)]\mapsto [X]$ is well-defined for every $c\in (0, \min \{1, r^{-1}\})$.
Moreover, if $r$ is an integer, then this forgetful map  is a smooth morphism with connected or empty fibers.
%let $\sU \subset \ocM^{{\rm sm},\delta\geq \epsilon_0}$ be an open substack such that $\sU$ is normal. Let $\sV\subset \cK\cM_{\chi_0,r,c}$ be the open substack whose $\bC$-points are $[(X,D)]\in \cK\cM_{\chi_0,r,c}$ satisfying $[X]\in \sU$. Then the forgetful morphism $\sV\to \sU$ that assigns $[(X,D)]\mapsto [X]$ is a smooth morphism with connected fibers.
\end{lem}

\begin{proof}
We fix a positive integer $n$ and a positive rational number $r$. Since smooth Fano manifolds of dimension $n$ are bounded \cite{KMM92, Cam92}, there are finitely many choices of $V$ and $\chi_0$. By \cite[Theorem 1.2]{ADL19} we know that the collection of $n$-dimensional $\bQ$-Fano varieties $X$ such that $[(X,D)]\in \cK\cM_{\chi_0, r, c}$ for some $D$, $\chi_0$, $V$, and $c\in (0,\min\{1,r^{-1}\})$ is bounded. Thus by \cite[Theorem A]{BJ17} and \cite[Proposition 5.3]{BL18b}, there exists $\epsilon_0\in (0,1]$ depending only on $n$ and $r$ such that $\delta(X)\geq \epsilon_0$ for every $X$ in this collection. This shows that the forgetful map $\cK\cM_{\chi_0,r,c}\to \ocM^{{\rm sm},\delta\geq \epsilon_0}$ is well-defined.

Next, we assume $r\in \bZ_{>0}$. For the last statement, following the last paragraph of the proof of \cite[Theorem 2.21]{ADL20}, it suffices to show that for any $\bQ$-Gorenstein smoothable $\bQ$-Fano variety $X$ and any effective Weil divisor $D$ on $X$ satisfying $D\sim_{\bQ} -rK_X$, we have $D\sim -rK_X$. Let $\pi:(\cX,\cD)\to B$ be a $\bQ$-Gorenstein smoothing over a pointed curve $0\in B$ with $(\cX_0, \cD_0)\cong (X,D)$ and $\cD\sim_{B,\bQ} -rK_{\cX/B}$. Since $\pi$ is smooth over $B\setminus\{0\}$, the class group of $\cX_b$ is torsion free for $b\in B\setminus\{0\}$. Thus we have $\cD|_{\cX\setminus\cX_0}\sim_B -rK_{\cX/B}|_{\cX\setminus\cX_0}$ as $r$ is an integer. Since $\cX_0$ is integral and $\cX_0\sim_B 0$, we know that $\cD\sim_B -rK_{\cX/B}$. This implies that $\cD_0\sim -rK_{\cX_0}$. The proof is finished.
\end{proof}

\section{Geometry and moduli of quartic K3 surfaces}

%\textcolor{blue}{Review K3 moduli, Laza-O'Grady, and ADL20.}

\subsection{Geometry of quartic K3 surfaces} Our goal in this section is to review the Hassett-Keel-Looijenga program (Section \ref{sec:laza-ogrady}). Before doing so, we introduce some terminology from the geometry of quartic K3 surfaces as studied by Mayer in \cite{Mayer}. A \emph{K3 surface} $S$ is a connected projective surface with du Val singularities such that $\omega_S\cong \cO_S$ and $H^1(S,\cO_S)=0$. A \emph{polarized K3 surface} $(S, L_S)$ consists of a K3 surface $S$ and an ample line bundle $L_S$ which is primitive. A \emph{quartic K3 surface} is a polarized K3 surface $(S,L_S)$ of degree $4$, i.e. $(L_S^2)=4$. %Let $S$ be a K3 surface and let $h$ be a primitive ample class on $S$ with $h^2 = 4$. 
Consider the map $S\dashrightarrow |L_S|^\vee\cong \bP^3$ induced by the linear system $|L_S|$. 

\begin{definition}\label{def:mayerK3} Generically, the linear system $|L_S|$ defines an isomorphism onto a quartic surface in $\bP^3$ with du Val singularities.
\begin{enumerate}
    \item We say that $S$ is \emph{hyperelliptic} if $|L_S|$ induces a  $2:1$ map onto a quadric surface in $\bP^3$. In this case $S$ is isomorphic to a double cover of $\bP^1 \times \bP^1$ or $\bP(1,1,2)$ ramified along a $(4,4)$ curve or a degree $8$ curve, respectively. 
    \item We say that $S$ is \emph{unigonal} if $|L_S|$ defines a rational map from $S$ onto a twisted cubic curve in $\bP^3$ with general fiber a smooth elliptic curve.\end{enumerate}
    By \cite{Mayer}, we know that any quartic K3 surface belongs to one of the three classes above.
\end{definition}

\subsection{K-moduli of quartic surfaces}
We define the K-moduli stacks $\osM_c^{\K}$ and spaces $\ofM_c^{\K}$. 
\begin{defn}
Let $\chi_0$ be the Hilbert polynomial of $(\bP^3, \cO_{\bP^3}(4))$. Let $c\in (0,1)\cap \bQ$ be a rational number. We define the K-moduli stacks $\osM_c^{\K}$ and spaces $\ofM_c^{\K}$ as
\[
\osM_c^{\K}:=\cK\cM_{\chi_0, 1, c}\quad \textrm{and}\quad \ofM_c^{\K}:=KM_{\chi_0,1,c}. 
\]
By Theorem \ref{thm:ADL19-moduli} we know that $\osM_c^{\K}$ is a reduced Artin stack of finite type, and $\ofM_c^{\K}$ is a reduced projective scheme. 
\end{defn}

\begin{lem}\label{lem:L-construct}
Let $X$ be a $\bQ$-Fano variety in $\ocM^{{\rm sm}, \delta\geq \epsilon}_{3,64}$ for some $\epsilon\in (0,1]$. Then $X$ admits a $\bQ$-Gorenstein smoothing to $\bP^3$. Moreover, there exists an ample $\bQ$-Cartier Weil divisorial sheaf $L$ on $X$ such that the following conditions hold.
\begin{enumerate}
    \item $L^{[m]}$ is Cohen-Macaulay for any $m\in \bZ$;
    \item $\omega_X\cong L^{[-4]}$ and $(L^3)=1$;
    \item $h^i(X, L^{[m]})=h^i(\bP^3, \cO_{\bP^3}(m))$ for any $m\in \bZ$ and $i\geq 0$;
\end{enumerate}
\end{lem}

\begin{proof}
From the Iskovskikh-Mori-Mukai classfication of smooth Fano threefolds \cite{IP99}, we know that $\bP^3$ is the only smooth Fano threefold with anti-canonical volume $64$. Hence $X$ admits a $\bQ$-Gorenstein smoothing $\pi:\cX\to B$ over a smooth pointed curve $0\in B$ such that $\cX_0\cong X$ and $\cX_b\cong\bP^3$ for $b\in B\setminus \{0\}$. Denote by $\cX^\circ:=\cX\setminus \cX_0$ and $B^\circ:=B\setminus \{0\}$. After a quasi-finite base change of $\pi$, we may assume that $\cX^\circ\cong \bP^3\times B^\circ$. Let $\cL^\circ$ be a Weil divisor on $\cX^\circ$ in  $|\cO_{\bP^3_{B^\circ}}(1)|$. Let $\cL$ be the Zariski closure of $\cL^\circ$ in $\cX$. Since $4\cL^\circ\sim_{B} -K_{\cX^\circ/B^\circ}$ and $\cX_0$ is a Cartier prime divisor, we know that $4\cL\sim_{B}-K_{\cX/B}$, in particular $\cL$ is $\bQ$-Cartier. Since $\cX$ is klt and $\cL$ is $\bQ$-Cartier, the sheaf $\cO_{\cX}(m\cL)$ is Cohen-Macaulay for any $m\in\bZ$ by \cite[Corollary 5.25]{KM98}. Let $L:=\cL|_{\cX_0}$, then $L$ is a $\bQ$-Cartier Weil divisor on $X$. Moreover, we have $\cO_{\cX}(m\cL)\otimes\cO_{\cX_0}$ and $\cO_{X}(mL)$ are isomorphic on a big open subset of $\cX_0$, hence they are isomorphic everywhere since $\cO_{\cX}(m\cL)\otimes\cO_{\cX_0}$ is Cohen-Macaulay. Thus part (1) is proved.

For part (2), notice that  $4\cL\sim_{B}-K_{\cX/B}$ implies $4L\sim -K_X$. We have $(L^3)=1$ since $(-K_X)^3=64$. Part (3) follows from Kawamata-Viehweg vanishing similar to \cite[Proof of Theorem 3.1]{Liu20}.
\end{proof}

From now on, we fix a number $\epsilon_0\in (0,1]$ from Lemma \ref{lem:forgetful-smooth} with $n=3$ and $r=1$.

Next we recall a result connecting K-stability and GIT stability as a special case of \cite[Theorem 1.2]{GMGS} and \cite[Theorem 1.4]{ADL19} (see also \cite{Zho21a}).

\begin{thm}[\cite{GMGS, ADL19}]\label{thm:K=GIT-smallc}
    There exists $\epsilon_2\in (0,1)$ such that for any rational number $c\in (0, \epsilon_2)$, a quartic surface $S\subset\bP^3$ is GIT (poly/semi)stable if and only if $(\bP^3, cS)$ is K-(poly/semi)stable.
\end{thm}

\begin{lem}\label{lem:ADE-K}
Let $S\subset \bP^3$ be a quartic surface.
\begin{enumerate}
    \item If $S$ has only ADE singularities, then $(\bP^3, cS)$ is K-stable for any  $c\in [0,1)\cap\bQ$.
    \item If $S$ is semi-log canonical, then $(\bP^3, cS)$ is K-semistable for any $c\in [0,1]\cap\bQ$.
\end{enumerate}
\end{lem}

\begin{proof}
We first prove part (1). Since $S$ has ADE singularities, it is GIT stable by \cite[Theorem 2.4]{Sha81}. Hence Theorem \ref{thm:K=GIT-smallc} implies that  $(\bP^3, \epsilon S)$ is K-stable for $0<\epsilon\ll 1$. Moreover, by adjunction we have that $(\bP^3, S)$ is plt. Hence part (1) follows from \cite[Proposition 2.13]{ADL19}.

Next we prove (2). In fact, inversion of adjunction implies that $(\bP^3, S)$ is log canonical since $S$ is slc. Then the result follows from \cite[Proposition 2.13]{ADL19} since $\bP^3$ is K-polystable. 
\end{proof}

Recall that $\sM^\circ$ and $\fM^\circ$ denote the modui stack and coarse moduli space of ADE quartic surfaces in $\bP^3$, respectively. Denote by $\sM^{\slc}$ the open substack of the GIT moduli stack $\osM^{\GIT}$ parametrizing quartic surfaces $S$ that are semi-log canonical (i.e. $(\bP^3, S)$ is log canonical).

\begin{prop}\label{prop:k-moduli-irred}
For any rational number $c\in (0,1)$, both $\osM_c^{\K}$ and $\ofM_c^{\K}$ are irreducible. Moreover, there are open immersions $\sM^\circ \hookrightarrow \sM^{\slc}\hookrightarrow\osM_c^{\K}$ whose images in $\osM_c^{\K}$ are saturated open substacks. Taking good moduli spaces yields open immersions $\fM^{\circ}\hookrightarrow\fM^{\slc}\hookrightarrow \ofM_c^{\K}$ where $\fM^{\slc}$ parametrizes GIT polystable slc quartic surfaces $S\subset \bP^3$. In particular, both $\sM^\circ$ and $\sM^{\slc}$ are saturated open substacks of $\osM^{\GIT}$.
\end{prop}

\begin{proof}
By Lemma \ref{lem:L-construct} we know that $\ocM_{3,64}^{{\rm sm},\delta\geq \epsilon_0}$ is irreducible. Thus $\osM_c^{\K}$ is irreducible since the forgetful map $\osM_c^{\K}\to \ocM_{3,64}^{{\rm sm},\delta\geq \epsilon_0}$ is smooth with connected fibers by Lemma \ref{lem:forgetful-smooth}. By Lemma \ref{lem:ADE-K} (2), we know that any $[S]\in \sM^{\slc}$ satisfies that $(\bP^3, cS)$ is K-semistable for any $c\in (0,1)$. Thus by openness of klt and lc (see \cite[Corollary 7.6]{SingularitiesOfPairs}), we know that both $\sM^\circ$ and $\sM^{\slc}$ are dense open substacks of $\osM_c^{\K}$. 

Next, we show saturatedness of  $\sM^\circ$ in $\osM_c^{\K}$. If $S$ is an ADE quartic surface, then Lemma \ref{lem:ADE-K} implies that $(\bP^3, cS)$ is K-stable for any $c\in (0,1)$. Thus all $\bC$-points in $\sM^\circ$ are closed with finite stabilizers, which implies that  $\sM^\circ$ is saturated in $\osM_c^{\K}$. 

Finally, we show saturatedness of  $\sM^\slc$ in $\osM_c^{\K}$. Let $S$ be a slc quartic surface. Since $(\bP^3,c S)$ is K-semistable for any $c\in (0,1)$ from the above discussion, by Theorem \ref{thm:K=GIT-smallc} we know that $S$ is GIT semistable. Let $S_0$ be the unique GIT polystable quartic surface in the orbit closure of $S$. Then by Theorem \ref{thm:K=GIT-smallc} we know that $(\bP^3, \epsilon S_0)$ is K-polystable for $0<\epsilon\ll 1$. Denote by $(\bP^3\times \bA^1, \cS)$ the test configuration of $(\bP^3, S)$ degenerating to $(\bP^3, S_0)$. Hence we have $\Fut(\bP^3\times\bA^1, \epsilon\cS)=0$. Since $\Fut$ is linear in coefficients, we know that $\Fut(\bP^3\times\bA^1, c\cS)=0$ for any $c\in (0,1)$. Hence $(\bP^3, cS_0)$ is K-semistable for any $c\in (0,1)$ by \cite[Lemma 3.1]{LWX18} which implies that $S_0$ is slc. Thus interpolation of K-stability \cite[Proposition 2.13]{ADL19} implies that $(\bP^3, cS_0)$ is the unique K-polystable degeneration of $(\bP^3, cS)$ for any $c\in (0,1)$. Since $S_0\in \sM^\slc$, we have that $\sM^\slc$ is saturated in $\osM_c^{\K}$. The last statement follows from  $\osM_{\epsilon}^{\K}\cong \osM^{\GIT}$ by Theorem \ref{thm:K=GIT-smallc}.
\end{proof}

\begin{defn}
The \emph{K-semistable threshold} of a quartic surface $S\subset\bP^3$ is defined as
\[
\kst(\bP^3, S):= \sup\{c\in [0,1]\mid (\bP^3, cS)\textrm{ is K-semistable}\}.
\]
If $S$ is GIT semistable, then by Theorem \ref{thm:K=GIT-smallc} and \cite[Theorem 3.15]{ADL19} we know that $\kst(\bP^3,S)\in (0,1]$ is a rational number, and the supremum is a maximum.
\end{defn}

 By Lemma \ref{lem:ADE-K}, we know that $\kst(\bP^3, S)=1$ if and only if $S$ is slc. If, in addition, $S$ is GIT polystable, then by interpolation of K-stability \cite[Proposition 2.13]{ADL19} and Theorem \ref{thm:K=GIT-smallc}, we know that for any $c\in [0, \kst(\bP^3,S))$, the log pair $(\bP^3, cS)$ is K-polystable. If a GIT polystable quartic surface $S$ is not slc, i.e. $\kst(\bP^3, S)<1$, then $(\bP^3, \kst(\bP^3, S) S)$ is K-semistable but not K-polystable by \cite[Proposition 3.18]{ADL19}.
 
% \footnote{YL: add a wall crossing result from ADL19.}

The following theorem follows directly from \cite[Theorem 1.2]{ADL19} and Proposition \ref{prop:k-moduli-irred}.

\begin{thm}\label{thm:ADL19-wall}
There exist rational numbers
$0=c_0<c_1<c_2<\cdots<c_k=1$
such that for each $0\leq i\leq k-1$, both the K-moduli stack $\osM_{c}^{\K}$ and the K-moduli space $\ofM_c^{\K}$ are independent of the choice of the rational number $c\in (c_i, c_{i+1})$. Moreover, for each $1\leq i\leq k-1$ and $0<\epsilon\ll 1$ we have open immersions 
\[
\osM_{c_i-\epsilon}^{\K} \hookrightarrow \osM_{c_i}^{\K} \hookleftarrow \osM_{c_i+\epsilon}^{\K}.
\]
which induce projective birational morphisms
\[
\ofM_{c_i-\epsilon}^{\K} \xrightarrow{\phi_i^-} \ofM_{c_i}^{\K} \xleftarrow{\phi_i^+} \ofM_{c_i+\epsilon}^{\K}.
\]
In addition, all the above morphisms have local VGIT presentations in terms of \cite[(1.2)]{AFS17}.
\end{thm}

\subsection{GIT stratification of quartic surfaces}\label{sec:GIT-quartic}

By Proposition \ref{prop:k-moduli-irred}, the GIT moduli space $\ofM^{\GIT}$ has an open subset $\fM^{\slc}$ parametrizing GIT polystable quartic surfaces with slc singularities. 
From \cite{Sha81, LO18b}, we know that the complement $\ofM^{\GIT}\setminus \fM^{\slc}$ (denoted by $\fM^{IV}$ therein) has two connected components, where one of them is an isolated point $\{[T]\}$ representing the tangent developable surface $T$ (see Definition \ref{def:tangentdevelopable}), and the other component
has the following stratification
\begin{equation}\label{eq:stratification}
(\ofM^{\GIT}\setminus \fM^{\slc})\setminus \{[T]\}=W_8\supset W_7\supset W_6\supset W_4\supset W_3\supset W_2\supset W_1\supset W_0.
\end{equation}
Here $W_i$ is an $i$-dimensional closed integral subvariety of $\ofM^{\GIT}$, and $W_0=\{[2Q]\}$ is the single point representing the double quadric surface. For each $i\in \{0,1,2,3,4,6,7,8\}$, we denote by $W_i^{\circ}:= W_i\setminus W_{i-1}$ when $i\not\in \{0, 6\}$, $W_6^{\circ}:= W_6\setminus W_4$, and $W_0^{\circ}=W_0$.

The following result follows from the classification of Shah \cite[S-4.3 on Page 282]{Sha81} (see also \cite[Section 4.3]{LO18b}).

\begin{thm}[\cite{Sha81}]\label{thm:shah}
Let $[S]\in W_i^{\circ}$ for $i\in \{0,1,2,3,4,6,7,8\}$. Then in suitable projective coordinates $[x_0,x_1,x_2,x_3]$ of $\bP^3$, the equation of $S$ has the form $q^2+g=0$ where $q$ and $g$ are given as follows.
\begin{align*}
q & = x_0 x_2 + x_1^2 +a x_3^2,\\
g & = \begin{cases} x_3^3 (x_0 + \beta_1(x_1, x_2, x_3)) +
x_2 ( x_3^2 f_1(x_1,x_2) + x_2 x_3 g_1(x_1, x_2) + x_2^2 h_1(x_1,x_2)), & \textrm{if } i\geq 3;\\
x_3^3 l_1(x_1,x_3), & \textrm{if } i\leq 2.
\end{cases}
\end{align*}
Here $\beta_1$, $f_1$, $g_1$, $h_1$, and $l_1$ are homogeneous linear polynomials in corresponding variables.
More precisely, we have the following classification.
\begin{enumerate}
    \item $[S]\in W_8^\circ$ if and only if  $x_2\nmid h_1$;
    \item $[S]\in W_7^\circ$ if and only if $x_2 \mid h_1$ and $x_2\nmid g_1$;
    \item $[S]\in W_6^\circ$ if and only if $x_2\mid h_1\neq 0$ and $x_2\mid g_1$;
    \item $[S]\in W_4^\circ$ if and only if $h_1= 0$, and either $x_2\mid g_1\neq 0$ or $x_2\nmid f_1$;
    \item $[S]\in W_3^\circ$ if and only if $h_1=g_1=0$ and $x_2\mid f_1\neq 0$;
    \item $[S]\in W_2^\circ$ if and only if $x_3\nmid l_1$;
    \item $[S]\in W_1^\circ$ if and only if $x_3\mid l_1\neq 0$;
    \item $[S]\in W_0^\circ$ if and only if $g= 0$ and $a\neq 0$.
\end{enumerate}
\end{thm}

\subsection{Laza-O'Grady and the Hassett-Keel-Looijenga Program}\label{sec:laza-ogrady}
The moduli space of quartic K3 surfaces can be constructed as a Type IV locally symmetric variety $\sF$, and comes with a natural Baily-Borel compactification $\sF^*$. The global Torelli theorem for K3 surfaces implies that the period map $\fp: \ofM^{\GIT} \dashrightarrow \sF^*$ is birational.  Building off of previous work of Shah \cite{Sha80} and Looijenga \cite{Loo031,Loo03}, in a series of papers \cite{LO19, LO18b, LO18a} Laza and O'Grady propose a conjectural method to resolve the period map $\fp$ whenever $\sF$ is a Type IV locally symmetric variety associated to a lattice of the form $U^2 \oplus D_{N-2}$. When $N = 18$ this is the case of hyperelliptic quartic K3 surfaces, and when $N = 19$ this is the case of quartic K3 surfaces. 

Recall that Baily-Borel showed that (for any $N$) one has $\sF^* \cong \Proj~ R(\sF, \lambda)$, where $\lambda$ denotes the hodge line bundle. Based on observations of Looijenga, Laza-O'Grady predict that in many cases  there is an isomorphism $\ofM^{\GIT} \cong \Proj~ R(\sF, \lambda + \Delta)$, for some geometrically meaningful boundary divisor $\Delta$ depending on $\sF$. 
Moreover, they predict that more generally the rings $R(\sF, \lambda + \beta \Delta)$ are finite generated and so the  schemes $\sF(\beta) = \Proj ~ R(\sF, \lambda + \beta \Delta)$  interpolate between $\ofM^{\GIT}$ and $\sF^*$. Their work also predicts the location of the walls, i.e. the values $\beta$ where the moduli spaces change. 

\subsubsection{Hyperelliptic quartic K3 surfaces}

When $N=18$, i.e. the \emph{hyperelliptic case}, Laza and O'Grady confirm their conjecture in \cite{LO19}. By Definition \ref{def:mayerK3}, this case occurs when the K3 surface is a double cover of $\bP^1 \times \bP^1$ ramified along a $(4,4)$ curve. If we let $\ofM_{(4,4)}^{\GIT}$ denote the GIT quotient of $(4,4)$ curves on $\bP^1 \times \bP^1$, then Laza and O'Grady show that the period map $\fp: \ofM_{(4,4)}^{\GIT} \dashrightarrow \sF^*(18)$ can be resolved via a series of explicit wall crossings arising from variation of GIT. Let $H_{h,18} \subset \sF(18)$ denote the divisor which parametrizes periods of hyperelliptic K3s which are the double cover of a quadric cone.

If $\mathrm{Reg}(\fp) \subset \ofM_{(4,4)}^{\GIT}$ denotes the regular locus of $\fp$, then $\fp(\mathrm{Reg}(\fp)) \cap \sF(18) \cong  \sF(18) \setminus H_{h, 18}$. Laza and O'Grady prove that $R(\beta) := R(\sF(18), \lambda + \beta \cdot \frac{H_{h, 18}}{2})$ is a finitely generated $\mathbb{C}$-algebra, and that $\sF_{18}(\beta) := \Proj ~ R(\beta)$ is a projective variety which interpolates between $\sF_{18}(0) \cong \sF(18)^*$, the Baily-Borel compactification, and $\sF_{18}(1) \cong \ofM_{(4,4)}^{\GIT}$, the GIT quotient. Moreover, the period map can be explicitly described as a composition of elementary birational maps. The first step $\sF_{18}(\epsilon) \to \sF_{18}(0)$ can be realized as the $\mathbb{Q}$-factorialization of $\sF(18)^*$, which fails to be $\mathbb{Q}$-factorial along $H_{h,18}$. The remainder of the birational transformations are flips, finally followed by a divisorial contraction $\sF_{18}(1-\epsilon) \to \ofM_{(4,4)}^{\GIT}$.   

In \cite{ADL20}, we show that the wall crossings resolving the period map $\fp$ can be interpreted as wall-crossings in a suitable K-moduli space of log Fano pairs. Let $\ocK_{c}$ denote the connected component of the moduli stack parametrizing K-semistable log Fano pairs admitting $\bQ$-Gorenstein smoothings to $(\bP^1 \times \bP^1, cC)$ where $C$ is a $(4,4)$ curve. Denote the good moduli space of $\ocK_c$ by $\oK_c$. Then, varying the weight $c$, the K-moduli spaces $\oK_c$ interpolate between $\ofM_{(4,4)}^{\GIT}$ and $\sF(18)^*$. Moreover, the explicit intermediate spaces constructed in \cite{LO19} using variation of GIT are all isomorphic to K-moduli spaces, and the walls coincide.

\subsubsection{General conjecture}

In fact, Laza and O'Grady conjecture that a similar behavior that is shown in \cite{LO19} is true for many classes of varieties whose moduli space can be constructed as a Type IV locally symmetric variety. They expect that, as in the hyperelliptic case, if $\fp: \ofM^{\GIT} \dashrightarrow \sF^*$ represents the relevant birational period map from a GIT compactification to Baily-Borel, then $\fp(\mathrm{Reg}(\fp) )\cap \sF \cong  \sF \setminus \Delta$, for some geometrically meaningful divisor $\Delta$. Moreover, $R(\beta) := R(\sF, \lambda + \beta \cdot \Delta)$ should be finitely generated and so $\sF(\beta) := \Proj~ R(\beta)$ should interpolate between $\ofM^{\GIT}$ and $\sF^*$. Moreover, they conjecture that interpolating from $\sF^*$ to $\ofM^{\GIT}$ should consist of birational transformations related to $\Delta$. More precisely, let $\mathscr{H} := \pi^{-1} \Supp \Delta$, where $\pi: \sD \to \sF$. Then $\mathscr{H}$ is a union of hyperplane sections of $\sD$, and has a stratification by closed subsets, where the stratification is given by the number of independent sheets of $\mathscr{H}$ containing the general point. Then, the stratification of $\mathscr{H}$ induces a stratification of $\Supp \Delta$.

With this in mind, the prediction for resolving $\fp$ is as follows: $\mathbb{Q}$-factorialize $\Delta$, followed by a series of explicit flips of strata inside $\Delta$, followed by a divisorial contraction of the strict transform of $\Delta$ to obtain $\ofM^{\GIT}$.  

\subsubsection{Quartic K3 surfaces}
Now let $\ofM^{\GIT}$ denote the GIT moduli space of quartic surfaces. We recall some notation from \cite{LO19}. For quartic K3 surfaces, we have that the K3 lattice $\Lambda \cong U^2 \oplus D_{17}$. Consider 
\[ \sD = \{ \left[ \sigma \right] \in \bP(\Lambda \otimes \bC) \mid \sigma^2 = 0 , (\sigma + \overline{\sigma})^2 > 0\}^+,\] where the superscript $+$ indicates that we have taken one of the two connected components. Let $\mathrm{O}^+(\Lambda)$ denote the subgroup of isometries $\mathrm{O}(\Lambda)$ of $\Lambda$ which fixes $\sD$. Then $\sF \cong \sD / \mathrm{O}^+(\Lambda)$ is the period space for quartic K3 surfaces (see \cite[Secion 1.2]{LO19}). 

\begin{definition}Thet \emph{hyperelliptic divisor} $H_h \subset \sF$ is the image of $v^\perp \cap \sD$ for $v \in \Lambda$ such that $q(v) = -4$ and $\mathrm{div}(v)=2$.  The \emph{unigonal divisor} $H_u \subset \sF$ is the image of $v^\perp \cap \sD$ for $v \in \Lambda$ such that $q(v) = -4$ and $\mathrm{div}(v)=4$.\end{definition}

\begin{remark}
Note that $\fp(S,L_S) \in H_h$ (resp. $H_u$) if and only if $(S,L_S)$ is hyperelliptic (resp. unigonal) in the sense of Definition \ref{def:mayerK3}.
\end{remark}

For quartic K3 surfaces, Laza-O'Grady predict that the regular locus of $\fp$ is the complement of $H_u$ and $H_h$. If $\Delta = (H_u + H_h)/2$, they predict that the predicted critical $\beta$ values are $\beta \in \{ 1, \frac{1}{2}, \frac{1}{3}. \frac{1}{4}, \frac{1}{5}, \frac{1}{6}, \frac{1}{7}, \frac{1}{9}, 0\}$ (see \cite[Prediction 5.1.1]{LO19}). Recall, in \eqref{eq:stratification}, we discussed a stratification of $\ofM^{\GIT}$. They predict that, under $\fp$, this stratification is related to a stratification of $\sF^*$. This is made more precise as follows. 
Let $\Delta^{(k)} \subset \Supp \Delta$ be the $k$-th stratum of the stratification defined above, and consider 
\begin{equation}\label{eq:Z^j-construct}
    Z^9 \subset Z^8 \subset Z^7 \subset Z^5 \subset Z^4 \subset Z^3 \subset Z^2 \subset Z^1 = H_u \cup H_h \subset \sF,
\end{equation}
 where
$Z^k = \Delta^{(k)}$ for $k \leq 5$, and then
\begin{itemize}
    \item $Z^7 = \mathrm{Im} \sF(\mathrm{II}_{2, 10} \oplus A_2) \hookrightarrow \sF$,
    \item $Z^8 = \mathrm{Im} \sF(\mathrm{II}_{2, 10} \oplus A_1) \hookrightarrow \sF$, and
    \item $Z^9 = \mathrm{Im} \sF(\mathrm{II}_{2, 10}) \hookrightarrow \sF$.
\end{itemize}

Then, they predict that each birational map occurring at one of the critical values above corresponds to a flip with center $Z^k$, and that each $Z^k$ is replaced by $W_{k-1} \subset \ofM^{\GIT}$.

\section{Tangent developable surface and unigonal K3 surfaces}\label{sec:unigonal}

In this section, we show that if $T$ denotes the tangent developable surface (see Definition \ref{def:tangentdevelopable}), then the pair $(\bP^3, cT)$ is K-polystable if and only if $c<\frac{9}{13}$. Moreover, the $c$-K-polystable replacements of $(\bP^3, T)$ for $c>\frac{9}{13}$ are log pairs $(X_u, S)$ where $X_u$ is a Gorenstein canonical Fano threefold (see Proposition \ref{prop:X_u-construct} for the construction), and $S$ is a GIT polystable unigonal K3 surface.

\subsection{Destabilizing divisor}
\begin{definition}\cite[Pages 30 - 31]{EL19}\label{def:tangentdevelopable}
The \emph{tangent developable} to the twisted cubic curve $C_0$ in $\bP^3$ is the surface $T \subset \bP^3$ defined to be the union of all the embedded projective tangent lines to $C_0$.  The surface $T$ is a quartic surface that has cuspidal singularities along $C_0$, and whose normalization is $\bP^1 \times \bP^1$ such that the diagonal $\Delta_{\bP^1} \subset \bP^1 \times \bP^1$ is the preimage of $C_0$. 
\end{definition}

Consider the group $G:=\PGL(2,\bC)$ acting linearly on $\bP^3$ where $C_0$ is $G$-invariant, since clearly $\Aut(\bP^3, C_0)\cong \Aut(C_0)\cong G$. A simple analysis shows that there are precisely three $G$-orbits in $\bP^3$, which are $C_0$, $T\setminus C_0$, and $\bP^3\setminus T$.

\begin{prop}\label{prop:normal-bundle}
Let $\iota:\bP^1 \xrightarrow{\cong} C_0$ be a parametrization. Then  $\iota^*\cN_{C_0/\bP^3}\cong \cO_{\bP^1}(5)\oplus\cO_{\bP^1}(5)$. Moreover, there exists a $G$-equivariant sub-line bundle $\cN_1$ of $\cN_{C_0/\bP^3}$ such that $\iota^*\cN_1\cong\cO_{\bP^1}(4)$. We denote by $\cN_2:= \cN_{C_0/\bP^3}/\cN_1$ the quotient line bundle. Then locally analytically along $C_0$ we may split $\cN_{C_0/\bP^3}$ as $\cN_1\oplus\cN_2$ such that the surface $T$ has analytic equation $(y^2=x^3)$ where $\cN_1=\langle\partial/\partial x \rangle$ and $\cN_2=\langle\partial/\partial y\rangle$.
\end{prop}

\begin{proof}
By \cite{EVdV81},  $\iota^*\cN_{C_0/\bP^3}\cong \cO_{\bP^1}(5)\oplus \cO_{\bP^1}(5)$.  Denote by $\cO_{C_0}(m):=\cO_{\bP^3}(m)\otimes \cO_{C_0}$, then we have $\iota^*\cO_{C_0}(m)\cong \cO_{\bP^1}(3m)$. Denote by $\cI_{C_0}$ the ideal sheaf of $C_0$ in $\bP^3$. Since $C_0$ is projectively normal, we have a short exact sequence
\[
0\to H^0(\bP^3, \cI_{C_0}(2))\to H^0(\bP^3, \cO_{\bP^3}(2))\to H^0(C_0, \cO_{C_0}(2))\to 0.
\]
Hence a dimension computation shows that $h^0(\bP^3, \cI_{C_0}(2))=3$. Since $G=\PGL(2,\bC)$ has fundamental group $\bZ/2\bZ$, the line bundle $\cO_{\bP^3}(2)$ has a natural $G$-linearization. Thus $G$ acts on $H^0(\bP^3, \cI_{C_0}(2))$. From the classification of $G$-orbits, we know that there does not exist any $G$-invariant quadric surface in $\bP^3$. Thus $H^0(\bP^3, \cI_{C_0}(2))$ is a $3$-dimensional irreducible representation of $G$. Next we consider the restriction map
\[
r: H^0(\bP^3, \cI_{C_0}(2))\to H^0(C_0, (\cI_{C_0}/\cI_{C_0}^2)(2)).
\]
It is clear that $r$ is non-zero and $G$-equivariant. Since $\cI_{C_0}/\cI_{C_0}^2$ is the conormal bundle $\cN_{C_0/\bP^3}^{\vee}$, we have that $\iota^*(\cI_{C_0}/\cI_{C_0}^2)(2)\cong \cO_{\bP^1}(1)^{\oplus 2}$ which implies that $h^0(C_0, (\cI_{C_0}/\cI_{C_0}^2)(2)) = 4$. Since $G$ is reductive, the cokernel of $r$ provides a non-zero $G$-invariant section $s \in  H^0(C_0, (\cI_{C_0}/\cI_{C_0}^2)(2))$. Since $(\cI_{C_0}/\cI_{C_0}^2)(2)\cong \cHom(\cN_{C_0/\bP^1},\cO_{C_0}(2))$, the section $s$ induces a  non-zero $G$-equivariant morphism $\cN_{C_0/\bP^3}\to \cO_{C_0}(2)$ which has to be surjective since $G$ acts transitively on $C_0$. Thus we define $\cN_2:=\cO_{C_0}(2)$ and $\cN_1$ as the kernel of the surjection $\cN_{C_0/\bP^3}\twoheadrightarrow\cN_2$. Computations on degrees show that $\iota^*\cN_1\cong \cO_{\bP^1}(4)$ and $\iota^*\cN_2\cong \cO_{\bP^1}(6)$.

It is clear that $T$ has equation $(y^2=x^3)$ in some analytic coordinate $(x,y,z)$ of $\bP^3$. Moreover, the tangent vector $\partial/\partial x$ spans a $G$-invariant sub-line bundle $\cN_1'$ of $\cN_{C_0/\bP^3}$. Since the $G$-action on $C_0\cong\bP^1$ is transitive, we know that either $\cN_1'=\cN_1$ or $\cN_1'\oplus \cN_1\cong \cN_{C_0/\bP^3}$. The latter case is not possible since $\cO_{\bP^1}(4)\cong \cN_1\hookrightarrow\cN_{C_0/\bP^3}\cong \cO_{\bP^1}(5)^{\oplus 2}$ does not split. Thus we have $\langle\partial/\partial x\rangle=\cN_1'=\cN_1$. 
\end{proof}

\begin{thm}\label{thm:tangent-kst}
Let $T$ be the tangent developable surface of twisted cubic curve $C_0$ in $\bP^3$. Let $c\in [0,1)$ be a rational number.
Then $(\bP^3, cT)$ is K-semistable (resp. K-polystable) if and only if $c\leq \frac{9}{13}$ (resp. $c<  \frac{9}{13}$).
\end{thm}

\begin{proof}
We first show the statement for K-semistability. For the ``only if'' part, we use Fujita-Li's valuative criteria Theorem \ref{thm:valuative}. 
%Here we give a valuative criterion computation of the divisor $E_0$ over $(\bP^3, cT)$. 
Denote by $\mu: Y_0\to \bP^3$ the $(2,3)$-weighted blow up of $\bP^3$ along the twisted cubic curve $C_0$ in the local coordinates $(x,y)$ defined by $\cN_1$ and $\cN_2$ from Proposition \ref{prop:normal-bundle}. Let $E_0$ be the exceptional divisor of $\mu$. 
Then we know that $E_0\cong \Proj_{\bP^1} \Sym (\cE_2\oplus \cE_3)$ where $\cE_2:=\cO_{\bP^1}(-4)\cong \iota^*\cN_1^\vee$, $\cE_3:=\cO_{\bP^1}(-6)\cong \iota^*\cN_2^\vee$, $\Sym(\cE_2\oplus\cE_3)$ is a $\bZ_{\geq 0}$-graded $\cO_{\bP^1}$-algebra, and each $\cE_i$ has degree $i$.
Denote by $\tT$ the strict transform of $T$ in $Y_0$. From the local computation that the $(2,3)$ blow-up  normalizes the cusp, we see that $\tT$ is the normalization of $T$, so by \cite[Page 31]{EL19}, $\tT\cong\bP^1\times\bP^1$. Denote the two pencil of rulings on $\tT$ by $R$ and $C$ respectively, such that $\mu_*R$ is a tangent line of $C_0$ and $\mu_*C$ is a conic curve. Denote by  $H:=\mu^*\cO(1)$. Then we have the following relations:
\begin{equation}\label{eq:intersection-tT}
(H\cdot R)=1, \quad (H\cdot C)=2, \quad (E_0\cdot R)=1, \quad (E_0\cdot C)=1.
\end{equation}
Moreover, we have $\tT=4H-6E_0$ and $-K_{Y_0}=4(H-E_0)$.
%Hence we see that $(-K_{Y_0}\cdot L)=0$, $(-K_{Y_0}\cdot C)=4$.

\textbf{Claim.} $H-E_0$ (and hence $-K_{Y_0}$) is nef and big on $Y_0$.
%$L$ generates an extremal ray of the Mori cone of $Y_0$.

\textbf{Proof of claim.} We first show that $H-E_0$ is nef. Indeed, since $4H-4E_0\sim \tT+2E_0$, it suffices to show that $(H-E_0)|_{\tT}$ and $(H-E_0)|_{E_0}$ are both nef. Since $\tT\cong\bP^1\times \bP^1$ with $R$ and $C$ as two rulings, \eqref{eq:intersection-tT} implies that $((H-E_0)\cdot R)=0$, $((H-E_0)\cdot C)=1$. Hence $(H-E_0)|_{\tT}$ is nef. Denote by $F$ the fiber of the $\bP^1$-bundle $E_0\to C_0$. Let $\Delta_{\tT}:=\tT\cap E_0$. It is clear that $E_0\cong \bP^1\times \bP^1$ with $F$ and $\Delta_{\tT}$ as the two rulings. Since $\Delta_{\tT}\sim R+C$ in $\tT$, \eqref{eq:intersection-tT} implies that $((H-E_0)\cdot \Delta_{\tT})=1$. Moreover, we know that $(H\cdot F)=(\cO_{\bP^3}(1)\cdot \mu_*F)=0$ and $-E_0|_{E_0}\sim\cO_{E_0}(1)$, hence $(-E_0\cdot F)=1$. Thus $((H-E_0)\cdot F)=1$. This implies that $(H-E_0)|_{E_0}$ is also nef.

Next we show that $H-E_0$ is big. This is done by computing the self intersection number of $H-tE_0$ for $t\in [0,1]$ and observing that the computation in the previous paragraph also shows that $H-tE_0$ is nef for $t \in [0,1]$. 
%Assuming the claim, we see that $H-E_0$ is a nef divisor which has zero intersection number with $L$. Later on we will see that it is also big. Hence by base-point-free theorem we know that $H-E_0$ is semiample. 
%Next we compute the self intersection number of $H-tE_0$. 
It is clear that 
\[
H|_{\tT}\sim 2R+C, \quad E_0|_{\tT}\sim R+C,\quad 
\tT|_{\tT}\sim 2R-2C.
\]
Hence 
\[
(H^3)=1, \quad (H^2\cdot \tT)=4, \quad (H\cdot \tT^2)=-2,\quad (\tT^3)=-8.
\]
Using the fact that $E_0=\frac{1}{6}(4H-\tT)$, we see 
\[
(H-tE_0)^3= ((1-\frac{2}{3}t)H +\frac{t}{6}\tT)^3= 1-\frac{3}{2}t^2+\frac{2}{3}t^3.
\]
When $0\leq t\leq 1$ this is also the volume of $H-tE_0$ since it is nef. Thus $\vol_{Y_0}(H-E_0)=\frac{1}{6}>0$ which implies that $H-E_0$ is big. The claim is proved.

Next we compute the $S$-invariant of $E_0$. Since $-K_{Y_0}$ is nef and big, and $Y_0$ has only quotient singularities which are klt, we know that $-K_{Y_0}$ (and hence $H-E_0$) is semi-ample. The ample model of $H-E_0$ gives a divisorial contraction $g: Y_0\to Y_0'$ which contracts $\tT$ to a smooth rational curve by contracting $R$ to a point.  By computation (using the intersection theory above), we have \[ g^*g_*(H - tE_0) = H- t E_0 + \left( \frac{1-t}{2} \right) \tT = H- tE_0 + \left( \frac{1-t}{2} \right) (4H - 6E_0) = (3-2t)(H-E_0). \]

Therefore, if $1< t < \frac{3}{2}$, the divisor $H-tE_0$ is big and by Zariski decomposition we know that
\[
\vol_{Y_0}(H-tE_0)=\vol_{Y_0}(g^*g_*(H-tE_0))=\vol((3-2t)(H-E_0))=\frac{(3-2t)^3}{6}.
\]
Therefore,
\[
\int_{0}^{\infty}\vol(H-tE_0)dt=\int_0^1 (1-\frac{3}{2}t^2+\frac{2}{3}t^3)dt+\int_1^{3/2}\frac{(3-2t)^3}{6}dt=\frac{11}{16}.
\]
Thus we know that 
\[
A_{(\bP^3, cT)}(E_0)=5-6c, \quad S_{(\bP^3,cT)}(E_0)=\frac{11}{16}(4-4c).
\]
So if $(\bP^3, cT)$ is K-semistable, then Theorem \ref{thm:valuative} implies that $5-6c\geq \frac{11}{16}(4-4c)$ which is equivalent to $c\leq \frac{9}{13}$. Thus we have shown the ``only if'' part for K-semistability. 

Next, we show the ``if'' part for K-semistability. By interpolation of K-stability \cite[Proposition 2.13]{ADL19}, it suffices to show that $(\bP^3, \frac{9}{13}T)$ is K-semistable. Consider the $G=\PGL(2,\bC)$-action on $\bP^3$ which preserves the twisted cubic curve $C_0$. By Theorem \ref{thm:kss-plt}, it suffices to show that $\beta_{(\bP^3, \frac{9}{13}T)}(F)\geq 0$ for any $G$-invariant prime divisor $F$ of plt type over $(\bP^3, \frac{9}{13}T)$. Since there are only three $G$-orbits in $\bP^3$: $C_0$, $T\setminus C_0$, and $\bP^3\setminus T$, we know that either $F=T$ or $F$ is centered at $C_0$. The first case is easy. For the second case, notice that localizing $(\bP^3, \frac{9}{13}T)$ at the generic point of $C_0$ produces a singularity analytically isomorphic to $(\bA^2, \frac{9}{13}(y^2=x^3))$. Hence the $(2,3)$-weighted blow up in $(x,y)$ produces the only $G$-invariant  divisor $E_0$ of plt type centered at $C_0$ by Lemma \ref{lem:exceptional-sing}. Therefore, the above computations show that $\beta_{(\bP^3, \frac{9}{13}T)}(E_0)= 0$. Thus $(\bP^3, \frac{9}{13}T)$ is K-semistable. This finishes the proof for K-semistability. 

Finally, we show the statement for K-polystability. Since $T$ is GIT polystable by \cite{Sha81, LO18b}, Theorem \ref{thm:K=GIT-smallc} implies that $(\bP^3, \epsilon T)$ is GIT polystable for $0<\epsilon\ll 1$. Thus \cite[Proposition 2.13]{ADL19} implies that $(\bP^3, cT)$ is K-polystable  when $0\leq c<\frac{9}{13}$. Moreover, since  $(\bP^3, cT)$ is K-unstable for any $c>\frac{9}{13}$ from the ``only if'' part for K-semistability, \cite[Proposition 3.18]{ADL19} implies that $(\bP^3, \frac{9}{13}T)$ is not K-polystable. Thus the proof is finished.
\end{proof}

\begin{lem}\label{lem:exceptional-sing}
The klt singularity $0\in (\bA^2, \frac{9}{13}(y^2=x^3))$ admits a unique plt blow-up given by the $(2,3)$-weighted blow-up in $(x,y)$.
\end{lem}

\begin{proof}
Let $E\subset Y\xrightarrow{\pi} \bA^2$ be the $(2,3)$-weighted blow-up in $(x,y)$. Let $\Gamma$ be the different divisor on $E$, i.e.  $K_E+\Gamma=(K_Y+E+\pi_*^{-1}D)|_E$ where $D=\frac{9}{13}(y^2=x^3)$. Then it is not hard to see that $(E,\Gamma)\cong (\bP^1, \frac{1}{2}[0]+\frac{9}{13}[1]+\frac{2}{3}[\infty])$. Hence  $\alpha(E,\Gamma)=(1-\max\{\frac{1}{2}, \frac{9}{13}, \frac{2}{3}\})(2-\deg(\Gamma)) =\frac{24}{11}>1$, which means that $(E,\Gamma)$ is an exceptional log Fano pair, i.e. there is at most exceptional divisor over it with log discrepancy one (see \cite[Definition 4.1]{Pro00}). Thus  \cite[Section 4]{Pro00} implies that $E$ is the unique Koll\'ar component over $0\in (\bA^2, D)$. 
\end{proof}

\subsection{Construction of $X_u$}\label{sec:construction}

Let $\cE:=\cO_{\bP^1}\oplus \cO_{\bP^1}(-4)\oplus \cO_{\bP^1}(-6)$ be a rank $3$ vector bundle over $\bP^1$. Denote by $\cE_i$ the $i$-th direct summand line bundle of $\cE$ for $i=1,2,3$. Denote by $\PE:=\Proj_{\bP^1} \Sym \,\cE$ where $\Sym\,\cE$ is a $\bZ_{\geq 0}$-graded $\cO_{\bP^1}$-algebra such that each line bundle $\cE_i$ has degree $i$. Thus $p:\PE\to \bP^1$ is a $\bP(1,2,3)$-bundle. It is clear that 
\begin{align*}
p_*\cO_{\PE}(1)&\cong\cO_{\bP^1},\\ p_*(\cO_{\PE}(2)\otimes p^*\cO_{\bP^1}(4)) & \cong\cO_{\bP^1}(4)\oplus \cO_{\bP^1},\\ p_*(\cO_{\PE}(3)\otimes p^*\cO_{\bP^1}(6)) & \cong \cO_{\bP^1}(6)\oplus \cO_{\bP^1}(2)\oplus \cO_{\bP^1}.
\end{align*}
Let $x\in H^0(\PE,\cO_{\PE}(1))$, $y\in H^0(\PE,\cO_{\PE}(2)\otimes p^*\cO_{\bP^1}(4))$, and $z\in H^0(\PE,\cO_{\PE}(3)\otimes p^*\cO_{\bP^1}(6))$ be non-zero sections in the last direct summand of the right-hand-side in each isomorphism above. 

By the Euler sequence for weighted projective bundles (see \cite[Ex. III.8.4]{Har77} and \cite{Dol82}), we have 
\begin{equation}\label{eq:-K_PE}
\cO(-K_{\PE})\sim \pi^*(\cE_1^\vee\otimes\cE_2^\vee\otimes\cE_3^\vee) \otimes \pi^*\cO_{\bP^1}(-K_{\bP^1})\otimes \cO_{\PE}(1+2+3)\cong p^*\cO_{\bP^1}(12) \otimes \cO_{\PE}(6).    
\end{equation}
It is clear that
\[
p_*\cO_{\PE}(6)\cong \cE_3^{\otimes 2}\oplus (\cE_1\otimes\cE_2\otimes\cE_3)\oplus (\cE_1^{\otimes 3}\otimes \cE_3) \oplus \cE_2^{\otimes 3}\oplus (\cE_1^{\otimes 2}\otimes\cE_2^{\otimes 2})\oplus  (\cE_1^{\otimes 3}\otimes\cE_2)\oplus \cE_1^{\otimes 6}.
\]
Thus any section $s\in H^0(\PE, \cO_{\PE}(-K_{\PE}))$ can be uniquely expressed as
\begin{equation}\label{eq:PE-anti-can}
s= az^2 +f_2xyz+f_6 x^3 z + by^3 + f_4 x^2 y^2 + f_8 x^4 y + f_{12} x^6,
\end{equation}
where $f_j \in H^0(\bP^1, \cO_{\bP^1}(j))$. 

Let $H_{\cE}\in |\cO_{\PE}(1)|$ be the Weil divisor defined by $(x=0)$. 
Let $T_{\cE}$ be the anti-canonical divisor on $\PE$ defined by $(z^2-y^3=0)$. Then $T_{\cE}$ and $H_{\cE}$ intersect transversally along a smooth curve $C_{\cE}=V(x, z^2-y^3)$ which is a section of $\PE\to \bP^1$. Let $h:\tPE\to \PE$ be the $(2,1)$-weighted blow up along the divisors $(T_{\cE}, H_{\cE})$. Let $\tE_{\cE}$ be the exceptional divisor of $h$. Denote by $\tH_{\cE}$ and $\tT_{\cE}$ the strict transforms of $H_{\cE}$ and $T_{\cE}$ in $\tPE$ respectively. It is clear that $\tH_{\cE}\cong H_{\cE}\cong \bP^1\times\bP^1$.

%Let $Y:=\Proj_{\bP^1} \Sym\,\pi_*\cO_S(3\Sigma)$ where we assign degrees $1,2,3$ for $\cO_{\bP^1}$, $\cO_{\bP^1}(-4)$, and $\cO_{\bP^1}(-6)$. Then $Y\to \bP^1$ is a $\bP(1,2,3)$-bundle containing $S$ as an anti-canonical divisor. Let $E$ be the only effective divisor in $|\cO_Y(1)|$. Then it is clear that $\Sigma=E|_S$. Let $\tY$ be the $(2,1)$-weighted blow up along the divisors $(S, E)$. Note that $\tY$ does not depend on the choice of $S$. Denote by $\tS$ and $\tE$ the strict transforms of $S$ and $E$ in $\tY$. It is clear that $\tS\cong S$. \footnote{YL: change notation later.}

\begin{prop}\label{prop:X_u-construct}
With the above notation, $\tPE$ is a Gorenstein canonical weak Fano threefold, i.e. $-K_{\tPE}$ is nef and big. We call its anti-canonical model $X_u$. Then $X_u$ is a Gorenstein canonical Fano threefold. Moreover, the birational morphism $\psi:\tPE\to X_u$ contracts $\tH_{\cE}$ to an isolated singularity $o\in X_u$, and is isomorphic elsewhere.
\end{prop}

\begin{proof}
From the geometry of the weighted blow-up $h$ we know that $\tPE$ has quotient singularities along three disjoint smooth rational curves in $\tH_{\cE}$, where two curves are of type $\frac{1}{2}(1,1,0)$, and the rest is of type $\frac{1}{3}(1,2,0)$. Thus $\tPE$ is Gorenstein canonical. 

Let $S_{\cE}$ be an anti-canonical divisor on $\PE$ defined by $(z^2=y^3+f_8x^4y+f_{12} x^6)$ where $f_8$ and $f_{12}$ are general degree $8$ and $12$ binary forms respectively. By \cite[Chapter 11 \S 2]{huybrechts} %\footnote{is this just since the binary forms are general, the discriminant does not vanish?},
we have that $S_{\cE}\to \bP^1$ is a smooth elliptic K3 surface with a section $C_{\cE}$, i.e. a smooth unigonal K3 surface, as the discriminant does not vanish for general binary forms $f_8$ and $f_{12}$. Denote by $F$ a general elliptic fiber on $S_{\cE}$. Then since $(C_{\cE}^2) = -2$ we know that $mF+C_{\cE}$ is ample on $S_{\cE}$ whenever $m\geq 3$. It is easy to see that $S_{\cE}$ and $T_{\cE}$ are tangent along $C_{\cE}$ which implies that $\ord_{\tE_{\cE}}(S_{\cE})=2$. Denote by $\tS_{\cE}$ the strict transform of $S_{\cE}$ in $\tPE$. Thus  $\tS_{\cE}=h^*S_{\cE}-2\tE_{\cE}\sim h^*(-K_{\PE})-2\tE_{\cE}\sim -K_{\tPE}$ because $A_{\PE}(\tE_{\cE})=3$. By \eqref{eq:-K_PE}, we know that $-K_{\PE}|_{S_{\cE}}\sim 12F + 6 C_{\cE}$. Thus  
\[
-K_{\tPE}|_{\tS_{\cE}}\sim h^*(12F+6C_{\cE})-2\tE_{\cE}|_{\tS_{\cE}}\sim 4h^*(3F+C_{\cE}).
\]
Since $h:\tS_{\cE}\xrightarrow{\cong} S_{\cE}$, we know that $-K_{\tPE}|_{\tS_{\cE}}$ is ample. Because $-K_{\tPE} \sim \tS_{\cE}$, it has non-negative intersection with any curve not contained in $\tS_{\cE}$, so this implies that $-K_{\tPE}$ is nef.  Furthermore, this implies $(-K_{\tPE})^3 = (-K_{\tPE})^2\cdot \tS_{\cE} = (-K_{\tPE}|_{\tS_{\cE}})^2 > 0$, so by \cite[Theorem 2.2.16]{Positivity1}, $-K_{\tPE}$ is big. By the Kawamata-Shokurov basepoint free theorem \cite[Theorem 3.3]{KM98}, $-K_{\tPE}$ is semiample with ample model $X_u$.

Moreover, $\tS_{\cE}$ is disjoint from $\tH_{\cE}$ which implies that $-K_{\tPE}|_{\tH_{\cE}}\sim 0$. Thus $\tH_{\cE}$ is contracted under $\psi$ to an single point. Since $h\circ p$ realizes $\tH_{\cE}$ as a $\bP^1$-bundle over $\bP^1$, we know that the curve classes $\cO(1,0)$ and $\cO(0,1)$ on $\tH_{\cE}$ are not proportional in $N_1(\tPE)_{\bR}$. Since $\tPE$ has Picard rank $3$, if we only contract $\tH_{\cE}$ from $\tPE$ then the resulting variety has Picard rank $1$ which has to be $X_u$. This implies that $\psi$ is isomorphic away from $\tH_{\cE}$. 

Denote by $o = \psi(\tH_{\cE})$ the unique singular point of $X_u$.
Since $\tS_{\cE}\sim -K_{\tPE}$ is disjoint from $\tH_{\cE}=\mathrm{Exc}(\psi)$, we know that $\psi_* \tS_{\cE}\sim -K_{X_u}$ does not pass through $o$. Thus $-K_{X_u}$ is Cartier. Moreover, we have $\psi^*(-K_{X_u}) = \psi^*\psi_*(-K_{\tPE}) = -K_{\tPE}$ as $-K_{\tPE}|_{\tH_{\cE}}\sim 0$. Thus $-K_{X_u}$ is ample on $X_u$, and $\psi$ is crepant birational. Thus $X_u$ is canonical as $\tPE$ is canonical. Therefore, $X_u$ is a  Gorenstein canonical Fano threefold.
\end{proof}

\begin{prop}\label{prop:tangent-Kps-replace}
With the above notation,  $E_0$ induces a K-polystable degeneration $(X_u, \frac{9}{13}T_0)$ of $(\bP^3, \frac{9}{13}T)$.%\footnote{YL: any idea for a better notation of this mysterious Fano $3$-fold? Maybe call it $X_u$.}
\end{prop}

\begin{proof}
We perform the following birational transformations,
\[
\begin{tikzcd}
 & \cX\arrow{ld}{\mu}\arrow{rd}{g}\arrow[rr,dashed,"f"]& & \cX^+\arrow{ld}{h}\arrow{rd}{\psi}\\
 \bP^3\times\bA^1& & \cY & & \cZ
 \end{tikzcd}
\]
where in the central fiber we have
\[
 \begin{tikzcd}
 & Y_0\cup \PE\arrow{ld}{\mu}\arrow{rd}{g}\arrow[rr,dashed,"f"]& & Y_0'\cup \tPE \arrow{ld}{h}\arrow{rd}{\psi}\\
 \bP^3 & & Y_0'\cup \PE & & X_u
 \end{tikzcd}
\]

Recall that $C_0 \subset T$ denotes the twisted cubic curve.  Using notation from Theorem \ref{thm:tangent-kst} and Proposition \ref{prop:X_u-construct}, the maps are as follows: 
    \begin{enumerate}
        \item $\mu: \calX \to \bP^3 \times \bA^1$ is the $(2,3,1)$-weighted blow up along $C_0$ in the central fiber, where $Y_0$ is the strict transform of $\bP^3 \times \{0\}$ and $\bP\calE$ is the exceptional divisor.
        \item $f: \calX \dashrightarrow \calX^+$ is the flip of $\widetilde{T}$ induced by the contraction $g$ of the rulings $R$ of $\widetilde{T}$ in the central fiber.  The strict transform of $Y_0$ is $Y'_0$ and the strict transform of $\bP\calE$ is $\widetilde{\bP\calE}$.
        \item $\psi: \calX^+ \to \calZ$ is the divisorial contraction of $Y'_0$. 
    \end{enumerate}
    
We elaborate on these maps below. 

The first morphism $\mu: \calX \to \bP^3 \times \bA^1_t$ is the $(2,3,1)$-weighted blow up of $\bP^3\times\bA^1$ along $C_0$, in local coordinates $(x,y,t)$, where $x,y$ are as in Theorem \ref{thm:tangent-kst}, whose exceptional divisor $\PE$ is a $\bP(1,2,3)$ bundle over $C_0$.  In the central fiber, this is the  $(2,3)$-weighted blow up of $\bP^3$ along $C_0 \subset T$.  

Let $R$ and $C$ denote the two pencils of rulings ($\mu_*R$ is a tangent line of $C_0$, and $\mu_*C$ is a conic curve).  With this notation, $g$ is the flipping contraction obtained by contracting the rulings $R$ of $\widetilde{T}$, the strict transform of $T$ in the central fiber, and the flip $f$ flips $\widetilde{T} \subset \calX$.  Indeed, the rulings $R$ are contractible, and the computations in Theorem \ref{thm:tangent-kst} imply that they generate a $K_{\calX} + \calT$-negative extremal ray in $\calX$, where $\calT$ is the family of surfaces.  Using the local structure of $\calX$ as a fibration over $C_0$, we can obtain an explicit description of the flip.  Along $C_0$, using the local coordinates $(x,y)$ in the central fiber, $T$ has equation $y^2 = x^3$ and $\mu$ normalizes the cusp.  Consider $\calX$ as a fibration of surfaces in a neighborhood of $p\in \calC_0$.  In each fiber $\calX_p = {Y_0}_p\cup \bP(1,2,3)$, the normal bundle of $R$ is $\calN_{R/\calX_p} = \calO_R(-2) \oplus \calO_R(-1)$ by computation, so the contraction $g$ creates a $\frac{1}{2}(1,1)$ singularity at $q\in {Y'_0}_p$ and the map $h$ is the $(2,1)$ weighted blow up of $q$ in $\bP(1,2,3)$.  The structure of the central fiber of $\calX^+$ is therefore exactly as described in Proposition \ref{prop:X_u-construct}, so in the central fiber, the strict transform of $\cT$ is $\tT_\cE$ in $\tPE$.  Recall from Proposition \ref{prop:X_u-construct} that $T_\cE$ is the anti-canonical divisor on $\PE$ defined by $(z^2 - y^3 = 0)$, and by construction, the strict transform of $\cT$ under the first morphism $\mu$ in $\PE$ is an anti-canonical divisor with this equation. 

The last morphism $\psi$ is the divisorial contraction of $Y'_0$, and $X_u$  is the image of $\tPE$ under $\psi$.  Indeed, an easy computation shows that curves in $Y'_0$ are $K_{\calX^+}$-negative, and Proposition \ref{prop:X_u-construct} shows that $Y'_0 \cap \tPE$ is contractible in $\tPE$, but $Y_0'$ has Picard rank one, hence is contractible in $\calX^+$ to $\calZ$.  Let $T_0$ be the central fiber of the strict transform of $\calT$ in $\calZ$, which is the image of $\tT_\cE \subset \tPE$ under $\psi$.

Finally, we show that $(X_u, \frac{9}{13}T_0)$ is K-polystable.  By Proposition \ref{thm:tangent-kst}, the log Fano pair $(\bP^3, \frac{9}{13}T)$ is K-semistable but not K-polystable. Let $G:=\PGL(2,\bC)$ as in the proof of Theorem \ref{thm:tangent-kst}. Then Proposition \ref{prop:G-kps} implies that there exists a $G$-invariant special divisor $F$ over $(\bP^3, \frac{9}{13}T)$ which induces a K-polystable degeneration and $\beta_{(\bP^3, \frac{9}{13}T)}(F)=0$. Since $F$ is $G$-invariant of plt type by Lemma \ref{lem:specialplt}, from the proof of Theorem \ref{thm:tangent-kst} there are only two such divisors: $T$ or $E_0$, and the $\beta$-invariant of $E_0$ (resp. of $T$) is zero (resp. non-zero). Hence $F=E_0$ and the proof is finished.
\end{proof}

In particular, a consequence of Proposition \ref{prop:tangent-Kps-replace} is that $X_u$ admits a $\bQ$-Gorenstein smoothing to $\bP^3$. Let $L$ be the $\bQ$-Cartier Weil divisor on $X_u$ obtained as the limit of $\cO_{\bP^3}(1)$ according to Lemma \ref{lem:L-construct}. The next result determines the Cartier index of $L$.

\begin{lem}\label{lem:L-index}
With the above notation, the Cartier index of $L$ on $X_u$ is $4$.
\end{lem}

\begin{proof}
By Lemma \ref{lem:L-construct} and Proposition \ref{prop:X_u-construct}, we know that $L$ is a $\bQ$-Cartier Weil divisor on $X_u$ such that $4L\sim -K_{X_u}$ is Cartier. Thus the Cartier index of $L$ divides $4$. Recall that $h:\tPE\to \PE$ is a $(2,1)$-weighted blow up along $(T_{\cE}, H_{\cE})$ where the center $C_{\cE}$ is a section of $\PE\to \bP^1$. Hence on $\tPE$ the divisor $\tT_{\cE}:=h_*^{-1} T_{\cE}$ and the $h$-exceptional divisor $\tE_{\cE}$ have  transverse intersection along a section of $\tPE\to \bP^1$. Denote by $R_{\cE}$ a fiber of $E_{\cE}\to \bP^1$. Then the above argument implies that $(\tT_{\cE}\cdot R_{\cE})=1$. By the proof of Proposition \ref{prop:X_u-construct}, we know that $\tT_{\cE}\sim -K_{\tPE}=\psi^*(-K_{X_u})\sim 4\psi^*L$. Thus we have $(L\cdot \psi_* R_{\cE}) = \frac{1}{4}(\tT_{\cE}\cdot R_{\cE})=\frac{1}{4}$ which implies that $L$ has Cartier index $4$.
\end{proof}

\subsection{GIT for Weierstrass models of unigonal K3 surfaces}

By construction of $X_u$, the surface $T_0$ has the structure of a Weierstrass elliptic fibration over the twisted cubic curve.  Repeating these birational transforms for any family $(\bP^3 \times \bA^1, \calT)$ whose central fiber $\calT_0$ is the tangent developable surface, the new central fiber also has this structure.  These elliptic surfaces naturally live inside (birational models of) $\bP\calE$, a $\bP(1,2,3)$ bundle over $C_0$.  Next, we relate this construction to GIT for Weierstrass elliptic surfaces. 

 We recall some notation from Section \ref{sec:construction}, namely the construction of Weierstrass elliptic surfaces inside a $\bP(1,2,3)$-bundle. As above, let $\calE = \calO_{\bP^1} \oplus \calO_{\bP^1}(-4) \oplus \calO_{\bP^1}(-6)$ and let $\calE_i$ denote the $i$th direct summand line bundle where $\calE_i$ has degree $i$. Then we obtain the $\bP(1,2,3)$ bundle $p: \bP\calE \to \bP^1$, where $\bP\calE := \Proj_{\bP^1} \Sym \calE$. Consider the affine space $\bfA:=H^0(\bP^1, \cO_{\bP^1}(8))\oplus H^0(\bP^1,\cO_{\bP^1}(12))$ of dimension $22$ parametrizing pairs $(A,B)$ of degree $8$ and $12$ binary forms respectively. In this way, for each pair $(A,B)\in \bfA$ one can associate a \emph{Weierstrass elliptic surface} $S_{(A,B)}\in |-K_{\PE}|$  whose equation is given by 
\[ S_{(A,B)}=(z^2 = y^3 + Ax^4y + Bx^6),\]
where $x, y, z$ have degrees one, two, and three respectively. Note that $S_{(0,0)}=T_{\cE}$ by definition.

%Let $\pi: S \to \bP^1$ be an elliptic K3 surface with section $\Sigma$. It is standard (see \cite[Section 2.2]{huybrechts}) that 
%\[ \pi_* \calO_S(\Sigma) \cong \calO_{\bP^1}, \quad \pi_* \calO_S(2\Sigma) \cong \calO_{\bP^1}(-4) \oplus \calO_{\bP^1}, \textrm{ and } \quad \pi_* \calO_S(3\Sigma) \cong \calO_{\bP^1}(-6) \oplus \calO_{\bP^1}(-4) \oplus \calO_{\bP^1}.  \]

%If $F = \pi_* \calO_S(3\Sigma)$, then the image of the morphism $S \to \bP(F^*)$ is the Weierstrass model $\overline{S}$ of $S$ (i.e. the unique surface obtained by contracting all fibral components not meeting the section). 
%Alternatively, the surface can be viewed as the closed subscheme of $\bP(F^*)$ defined by 
%\[ y^2z = x^3 + Axz^2 + Bz^3,    \] where $A \in H^0(\bP^1, \calO_{\bP^1}(8))$ and $B \in H^0(\bP^1, \calO_{\bP^1}(12))$.
The discriminant divisor $4A^3 + 27B^2$ vanishes at the singular fibers, and for every point $p \in \bP^1$, one has $\ord_p(A) \leq 3$ or $\ord_p(B) \leq 5$. The latter condition is equivalent to the surface having only ADE singularities. Furthermore, one can show that a pair $(A,B)$ satisfying these conditions
yields a unigonal K3 surface with ADE singularities, and vice versa. 

Next, we introduce the GIT set-up for Weierstrass elliptic surfaces.  Clearly the affine space $\bfA$ admits an $(\SL(2,\bC)\times\bG_m)$-action given by 
\[
(g,t)\cdot (A,B):= (t^4 A\circ g, t^6 B\circ g).
\]
Let $\cP:=[(\bfA\setminus\{0\})/\bG_m]$ be the weighted projective stack where $\bG_m$ acts as above. Let $\bfP$ be the weighted projective space as the coarse moduli space of $\cP$. Then we have $\bfP\cong \bP(2^9, 3^{13})$. The $\SL(2,\bC)$-action on $\bfA$ descends to an $\SL(2,\bC)$-action on $(\bfP, \cO_{\bfP}(1))$. We say that a non-zero pair $(A,B)\in \bfA\setminus\{0\}$ is GIT (poly/semi)stable if $[A,B]\in \bfP$ is GIT (poly/semi)stable with respect to the $\SL(2,\bC)$-action on $(\bfP,\cO_{\bfP}(1))$. Denote by $\bfA^{\rm ss}\subset \bfA\setminus\{0\}$ the GIT semistable open locus.

In \cite{Mir81}, Miranda constructs a compact moduli space parametrizing Weierstrass fibrations over $\bP^1$ using GIT following the above discussion. 
In particular, 
Miranda gives (see \cite[Proposition 5.1]{Mir81}) a criterion for when a pair $(A,B)$ is GIT (semi)stable based on the valuations.

\begin{rem}
Before stating the following theorem, we point out that Miranda's setup (and the discussion that follows in \cite{OO18} and \cite{ABK3}) considers Weierstrass elliptic surfaces embedded in a $\bP^2$ bundle, instead of a $\bP(1,2,3)$ bundle, which is more convenient for our work. The same results hold in our setup, as the underlying surfaces are the same. 
\end{rem}

\begin{thm}\label{thm:unigonal-GIT-slc}
If $(A,B)\in \bfA\setminus \{0\}$ is GIT semistable, then the surface $S_{(A,B)}$ is slc.
\end{thm}

\begin{proof}
See \cite[Proposition 7.4]{OO18} for an explicit description of the GIT compactification of Weierstrass elliptic surfaces (see also \cite[Section 10]{ABK3}).
\end{proof}

\begin{lem}\label{lem:X_u-anti-can}
Let $S\in |-K_{X_u}|$ be an anti-canonical divisor. 
If $S$ does not pass through the singular point $o\in X_u$, then there exists an automorphism $\varphi\in \Aut(X_u)$ such that $(X_u,\varphi^* S)$ is the birational transform of $(\PE, S_{(A,B)})$ for some $(A,B)\in \bfA$. Moreover, $S$ is isomorphic to the Weierstrass elliptic surface $S_{(A,B)}$.
\end{lem}

\begin{proof}
Recall from Proposition \ref{prop:X_u-construct} that there is a  diagram $\PE \xleftarrow{h} \tPE \xrightarrow{\psi} X_u$ where $h$ is a weighted blow-up and $\psi$ is taking the anti-canonical model. Let $S_{\cE}:=h_* \psi^* S$. Then we know that $\ord_{\tE_{\cE}}(S_{\cE})=2$ which implies that $S_{\cE}$ is tangent to $T_{\cE}$ along the curve $C_{\cE}$. Thus by \eqref{eq:PE-anti-can} we know that $S_{\cE}$ has the equation $(a(z^2-y^3)+f_6x^3z+f_4x^2y^2+f_8x^4y+f_{12}x^6=0)$. Since $o\not\in S$, we know that $S_{\cE}$ does not contain $H_{\cE}=(x=0)$ which implies $a\neq 0$. Hence we may assume that $a=1$.  Let $\varphi_{\cE}\in \Aut(\PE)$ be $\varphi_{\cE}(x,y,z):=(x, y+\frac{f_4}{3}x^2, z-\frac{f_6}{2}x^3)$. Then clearly $\varphi_{\cE}^*S_{\cE}$ is defined by the equation $(z^2=y^3+Ax^4y+Bx^6)$ which is a Weierstrass elliptic surface, i.e. $\varphi_{\cE}^*S_{\cE}=S_{(A,B)}$. Hence $\varphi_{\cE}$ induces the desired $\varphi\in\Aut(X_u)$. Since $S_{\cE}$ is tangent to  $T_{\cE}$ along $C_{\cE}$, we know that $h_*^{-1} S_{\cE}$ is isomorphic to $S_{\cE}$ and disjoint with $\tH_{\cE}$, which implies $S=\psi_* (h_*^{-1} S_{\cE})\cong h_*^{-1} S_{\cE}\cong S_{\cE}$.
\end{proof}

By Lemma \ref{lem:X_u-anti-can}, every unigonal K3 surface $S_{(A,B)}$ with ADE singularities is isomorphic  to an anti-canonical divisor $S\in |-K_{X_u}|$ not passing through $o$. Since the ample $\bQ$-Cartier divisor $L$ on $X_u$ is Cartier away from $o$, and $(L^2\cdot S)=\frac{1}{16}(-K_{X_u})^3=4$, we know that every polarized unigonal K3 surface of degree $4$ is isomorphic to some $(S,L|_S)$.

\begin{thm}\label{thm:GIT=K-unigonal}
Let $S,S'\in |-K_{X_u}|$ be two divisors where $o\not\in S$ and $o\in S'$. Let $(A,B)\in \bfA$ be a pair such that $\varphi^* S$ is the birational transform of $S_{(A,B)}$ for some $\varphi\in \Aut(X_u)$. 
\begin{enumerate}
    \item $(X_u, \frac{9}{13}S)$ is always K-semistable, and it is  K-polystable if and only if $\varphi^*S=T_0$;
    \item if $c\in (\frac{9}{13},1)\cap \bQ$, then $(X_u, cS)$ is K-(poly/semi)stable if and only if $(A,B)\in \bfA\setminus \{0\}$ is GIT (poly/semi)stable;
    \item if $c\in  (0, \frac{9}{13})\cap \bQ$, then $(X_u, cS)$ is K-unstable;
    \item $(X_u, cS')$ is K-unstable for any rational $c\in (0,1)$.
\end{enumerate}
\end{thm}

\begin{proof}
For simplicity, we always assume that $\varphi$ is the identity as guaranteed by Lemma \ref{lem:X_u-anti-can}, i.e. $S=\psi_*(h_*^{-1}S_{(A,B)})$ for some $(A,B)\in \bfA$.

(1) Since $\lim_{t\to 0}(t^4A, t^6 B)=(0,0)$, the $\bG_m$-action on $\bfA$ induces a special degeneration of $(X_u,\frac{9}{13}S)$  to $(X_u, \frac{9}{13}T_0)$ which is K-polystable by Proposition \ref{prop:tangent-Kps-replace}. Hence it follows from openness of K-semistability \cite{BLX19, Xu19} and the fact that $(X_u,S)\not\cong (X_u, T_0)$ whenever $(A,B)\neq (0,0)$.

(2) To start with, notice that $(\bP^3, cT)$ admits a special degeneration to $(X_u,cT_0)$ for $c\leq \frac{9}{13}$. Since $(\bP^3, cT)$ is K-polystable for $c<\frac{9}{13}$ by Theorem \ref{thm:tangent-kst}, we know that $(X_u,cT_0)$ is K-unstable for $c<\frac{9}{13}$. Since $(X_u, \frac{9}{13}T_0)$ is K-polystable by Proposition \ref{prop:tangent-Kps-replace},  we know that $(X_u, cT_0)$ is K-unstable for any $c\neq \frac{9}{13}$ by \cite[Proposition 3.18]{ADL19}. Thus we may assume $(A,B)\neq 0$.

We first show the ``only if'' part, that is, K implying GIT. This follows from the Paul-Tian criterion \cite[Theorem 2.22]{ADL19}.
There is a universal family of Weierstrass elliptic surfaces $(\PE\times\bfA, \cS_{\cE,\bfA})\to \bfA$ where the fiber over $(A,B)\in \bfA$ is $(\PE, (z^2=y^3+Ax^4 y+ Bx^6))$. By performing the weighted blow-up $h$ and anti-canonical morphism $\psi$ in families, we obtain a universal family $(X_u\times\bfA, \cS_{\bfA})\to \bfA$. Clearly the $\bG_m$-action on $\bfA$ lifts naturally to $\PE\times\bfA$ (where $t\cdot ([x,y,z], (A,B)):= ([tx,y,z], (t^4A, t^6B))$) and $X_u\times\bfA$ such that both $\cS_{\cE,\bfA}$ and $\cS_{\bfA}$ are $\bG_m$-invariant. Hence taking the stacky quotient of $\bG_m$ over $\bfA\setminus \{0\}$, we obtain a $\bQ$-Gorenstein log Fano family $\pi_{\cP}:(\cX_{\cP}, \cS_{\cP})\to \cP$. We show that the assumptions of \cite[Theorem 2.22]{ADL19} are satisfied for $\pi_{\cP}:(\cX_{\cP},c\cS_{\cP})\to \cP$ when $c\in (\frac{9}{13},1)$. Assumptions (a) and (b) are straightforward. For assumption (c), it suffices to show that the CM line bundle $\lambda_{\CM, \pi_{\cP}, c\cS_{\cP}}$ descends to an ample $\bQ$-line bundle $\Lambda_{\bfP,c}$ on $\bfP$ for $c\in (\frac{9}{13},1)$. For a general choice of $(A,B)$, the surface $S$ is klt which implies that $(X_u, S)$ is plt. Hence Theorem \ref{thm:uKs-almost-logCY} implies that $(X,(1-\epsilon)S)$ uniformly K-stable for $0<\epsilon\ll 1$. Thus interpolation \cite[Proposition 2.13]{ADL19} implies that $(X_u,cS)$ is uniformly K-stable for any $c\in (\frac{9}{13},1)$. Thus \cite{CP18, Pos19, XZ19} implies that $\Lambda_{\bfP,c}$ is big, hence it is ample since $\bfP$ has Picard rank $1$. Thus the proof of the ``only if'' part is finished by \cite[Theorem 2.22]{ADL19}.

For the ``if'' part, notice that by Theorem \ref{thm:unigonal-GIT-slc} any GIT semistable $S$ is slc. Thus $(X_u, cS)$ is K-semistable for any $c\in (\frac{9}{13},1)$ by part (1) and interpolation for K-stability \cite[Proposition 2.13]{ADL19}. For a general choice of $(A,B)$, the above discussion shows that $(X_u,cS)$ is K-stable for any $c\in (\frac{9}{13},1)$. If $S$ is GIT polystable, then it suffices to show that $(X_u, (\frac{9}{13}+\epsilon)S)$ is K-polystable, since then $(X_u, cS)$ is K-polystable by interpolation \cite[Proposition 2.13]{ADL19}. Let $(X_0, (\frac{9}{13}+\epsilon)S_0)$ be the K-polystable degeneration of $(X_u, (\frac{9}{13}+\epsilon)S)$. Then we know that $(X_0, \frac{9}{13}S_0)$ is K-semistable. Since $(X_u, \frac{9}{13}T_0)$ is the K-polystable degeneration of $(X_u, \frac{9}{13}S)$, by \cite{LWX18} we have a sequence of special degenerations $(X_u, \frac{9}{13}S)\rightsquigarrow (X, \frac{9}{13}S_0)\rightsquigarrow (X_u, \frac{9}{13}T_0)$. This implies that $X\cong X_u$ and $S_0$ does not pass through the singular point on $X$. We may assume that $X=X_u$ for simplicity. By Lemma \ref{lem:X_u-anti-can}, we can find a $1$-parameter family of automorphisms $(\varphi_t)\in \Aut(X_u)$ over a pointed curve $(B,0)$ such that $\varphi_t^* S~(t\neq 0)$ and $\varphi_0^*S_0$ are all in Weierstrass form, and $\lim_{t\to 0}\varphi_t^* S=\varphi_0^*S_0$. Since $(X_u, (\frac{9}{13}+\epsilon)\varphi_0^* S_0)$ is K-polystable, we know that $\varphi_0^*S_0\neq T_0$, i.e. $\varphi_0^*S_0$ corresponds to a non-zero pair $(A_0,B_0)\in \bfA$. By the ``only if'' part, $(A_0,B_0)$ is GIT polystable. Hence the separatedness of GIT quotient implies that $[A,B]$ and $[A_0,B_0]$ belong to the same $\SL(2,\bC)$-orbit in $\bfP$. In particular, $(X_u,  (\frac{9}{13}+\epsilon)S)\cong (X_u,  (\frac{9}{13}+\epsilon)S_0)$ are K-polystable. The proof is finished.

(3) Assume to the contrary that $(X_u, cS)$ is K-semistable for some $c< \frac{9}{13}$. By openness of K-semistability \cite{BLX19, Xu19}, we may choose $S$ such that $(A,B)$ is a general pair of binary forms. Thus $(A,B)$ is GIT stable which implies that $(X, (\frac{9}{13}+\epsilon)S)$ is K-stable by part (2). By interpolation  \cite[Proposition 2.13]{ADL19} we know that $(X, \frac{9}{13}S)$ is K-stable, contradicting part (1). Thus part (3) is proved.

(4) By Proposition \ref{prop:X_u-construct},  the divisor $S'$ is Cartier as $X_u$ is Gorenstein. Besides, the birational morphism $\psi:\tPE\to X_u$ contracts the prime divisor $\tH_{\cE}$ to $o\in X_u$, where $\tH_{\cE}\cong\bP^1\times\bP^1$ and $K_{\tPE}|_{\tH_{\cE}}\sim 0$. Hence by adjunction we have $\tH_{\cE}|_{\tH_{\cE}}=K_{\tH_{\cE}}+\Gamma$ where $\Gamma$ is the different divisor. Denote the two families of fibers of $\bP^1\times\bP^1$ by $F_a$ and $F_b'$ for $a,b\in\bP^1$. Then by construction we have that $\Gamma=\frac{1}{2}F_0+\frac{1}{2} F_1+\frac{2}{3}F_\infty$ under a suitable choice of coordinates. Thus $A_{X_u}(\tH_{\cE})=1$ and $\vol_{X_u, o}(\tH_{\cE})=(-K_{\tH_{\cE}}-\Gamma)^2=\frac{4}{3}$.
Here we refer to the survey article \cite{LLX18} for the background of normalized volumes. Since $S$ is Cartier and passes through $o=\psi(\tH_{\cE})$, we have $\ord_{\tH_{\cE}}(S)\geq 1$. Hence we have
\[
\hvol_{(X_u,cS),o}(\tH_{\cE})=( A_{X_u}(\tH_{\cE})-c\ord_{\tH_{\cE}}(S))^3 \vol_{X_u,o}(\tH_{\cE})\leq \frac{4}{3}(1-c)^3.
\]
On the other hand, we have $\frac{27}{64}(-K_{X_u}-cS)^3= 27(1-c)^3> \hvol_{(X_u,cS),o}(\tH_{\cE})$,
which implies that $(X_u,cS)$ is K-unstable by \cite{LL16}.
\end{proof}

\begin{comment}
\subsection{Smoothness of the stack}
\YL{Temporary name, using K3 to show smoothness of the stack containing $\bP^3$ and $X_u$.}
\KD{I think the original section is still okay; I added some details to the local-to-global obstruction calculation so that it is actually correct and can add more if we decide to keep it as is.  If we want to shorten the paper, though, I am happy to remove the deformation theory of $X_u$ section!}

In this subsection, we aim to prove the following result.

\begin{thm}
    There exists a smooth open substack $\sU$ of $\ocM_{3,64}^{\rm sm}$ parametrizing $\bP^3$ and $X_u$. 
\end{thm}

\begin{proof}
We first show that $\sU$ is an open substack.
It is clear that the locus parametrizing $\bP^3$ and $X_u$ is constructible. Thus it suffices to show that every $1$-parameter family  $\cX\to B$ with special fiber being $\bP^3$ or $X_u$ satisfies that a general fiber is also $\bP^3$ or $X_u$.
\end{proof}
\YL{Halted for now, try to work out deformation theory instead.}
\end{comment}

\subsection{Deformation theory of $X_u$}

In this subsection, we show that the allowed deformations of $X_u$ are unobstructed, and any such small deformation of $X_u$ is isomorphic to itself or $\bP^3$. 

In the moduli stack $\ocM_{3,64}^{{\rm sm},\delta\geq \epsilon_0}$, we only consider deformations induced by the index one cover of the $\bQ$-Cartier Weil divisor $L$ that is the limit of $\calO_{\PP^3}(1)$, see Lemma \ref{lem:L-construct}.  In the case of $X_u$, the divisor $L$ is 4-Cartier by Lemma \ref{lem:L-index}.

\begin{lemma}\label{lem:X_u-indexonecover}
The unique singular point $o\in X_u$ admits a special degeneration to an orbifold cone singularity $x_0\in X_{u,0}$. Under this special degeneration, the divisor $L$ degenerates to a $\bQ$-Cartier Weil divisor $L_0$ on $X_{u,0}$, whose Cartier index is $4$. Moreover, the index one cover $Y\to X_{u,0}$ of $L_0$ is a toric Gorenstein canonical singularity of the form $(xy - zw=0) \subset \bA^4/\bmu_3$, where the action of $\bmu_3$ is given by $\zeta_3 \cdot (x,y,z,w) = (\zeta_3 x, \zeta_3^2 y, \zeta_3^2 z, \zeta_3 w )$ with $\zeta_3=e^{\frac{2\pi i}{3}}$.%\footnote{add proof/computation}
\end{lemma}

\begin{proof}
By construction from Section \ref{sec:construction} and the proof of Theorem \ref{thm:GIT=K-unigonal}(4), we know that $(\tH_{\cE},\Gamma)\cong(\bP^1\times\bP^1,\frac{1}{2}F_0+\frac{1}{2} F_1+\frac{2}{3}F_\infty)$ and $\tH_{\cE}|_{\tH_{\cE}}=K_{\tH_{\cE}}+\Gamma$ where $\Gamma$ is the different divisor. Thus $\tH_{\cE}$ induces a special degeneration of $(o\in X_u)$ to the orbifold cone $X_{u,0}:=C_{a}(\bP^1\times\bP^1, -K_{\bP^1\times\bP^1}-\Gamma)$ (as defined in \cite[Section 2.4]{LX16} and \cite[Section 2.4]{LZ19}).  The notion of special degeneration in the local setting comes from Koll\'ar components over the singular point (c.f. \cite[Definition 4.24]{Xu20}). Since $X_{u,0}$ is an anti-canonical cone over a klt log Fano pair with standard coefficients, it is Gorenstein canonical (\cite[Lemma 3.1, Proposition 3.14]{Kol13}).

Let $L_0$ be the degeneration of $L$ to $X_{u,0}$. Then clearly $4L_0\sim -K_{X_{u,0}}$ is Cartier, and $2L_0$ is not Cartier because $2L$ is not Cartier. Thus we can take an index $1$ cover $Y\to X_{u,0}$ of $L_0$ which has degree $4$. In fact, there is a local universal cover of $o\in X_{u,0}$ as follows:
\[
C_a(\bP^1\times\bP^1, \cO(1,1))\to C_a(\bP^1\times\bP^1,\cO(2,2))\to C_{a}(\bP^1\times\bP^1, -K_{\bP^1\times\bP^1}-\Gamma).
\]
Here the first map is raising the polarization to the second tensor power, and the second map comes from the quotient map $\bP^1\to (\bP^1,\frac{1}{2}[0]+\frac{1}{2}[1]+\frac{2}{3}[\infty])$ of degree $6$, induced by the $\fS_3$-action on $\bP^1$ generated by $[u_0,u_1]\mapsto [u_0,\zeta_3 u_1]$ and $[u_0,u_1]\mapsto [u_1,u_0]$. Thus we have a group of order $12$ (isomorphic to a binary dihedral group) acting on the singularity $C_a(\bP^1\times\bP^1,\cO(1,1))$ whose quasi-\'etale quotient is precisely $x_0\in X_{u,0}$.

Since $L_0$ has Cartier index $4$,  the index $1$ cover of $L_0$ on $X_{u,0}$ is a $\bmu_3$-quotient of the singularity $C_a(\bP^1\times\bP^1,\cO(1,1))$. Let $([u_0,u_1],[v_0,v_1])$ be the projective coordinate on $\bP^1\times\bP^1$, then we may identify $C_a(\bP^1\times\bP^1,\cO(1,1))$ with $\tY:=(xy-zw=0)\subset \bA^3$, where
\[
(x,y,z,w)=(u_0 v_0, u_1 v_1, u_1 v_0, u_0 v_1).
\]
The group $\bmu_3$ acts on $\bP^1$ as generated by $[u_0,u_1]\mapsto [u_0, \zeta_3 u_1]$. After lifting to the anti-canonical cone $C_a(\bP^1\times\bP^1, \cO(2,2))$, the $\bmu_3$-action becomes $(u_0^2, u_0u_1, u_1^2)\mapsto (\zeta_3^2 u_0^2, u_0u_1 , \zeta_3 u_1^2)$ where $v_0^2$, $v_0v_1$, and $v_1^2$ are $\bmu_3$-invariant. Thus the only lifting of $\bmu_3$-action on $\tY$ is given by $(x,y,z,w)\mapsto (\zeta_3 x, \zeta_3^2 y, \zeta_3^2 z, \zeta_3 w)$.
\end{proof}

Therefore, we must understand the deformation theory of these singularities.  Because they are toric Gorenstein threefold singularities, Altmann's method described in \cite{Altmann} applies to compute the miniversal deformation space.  

\begin{lemma}
The singularity  $0\in Y = (xy = zw) \subset \bA^4/\bmu_3$ with above action of $\bmu_3$ is isomorphic to the affine toric Gorenstein threefold singularity $V$ defined by the cone in $\mathbb{R}^3$ with two dimensional polytope $Q$ as the hyperplane section $t = 1$ pictured below.

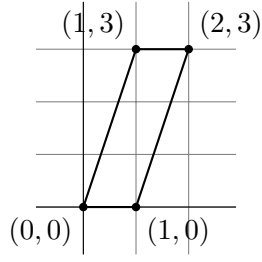
\begin{center}
    \begin{tikzpicture}[scale=.7]
    \draw[thin, gray] (-.9,-.9) grid (2.9,3.9);
    \draw (-.9,0) -- (2.9,0);
    \draw (0,-.9) -- (0,3.9);
    \draw[thick] (0,0) -- (1,0); 
    \draw[thick] (1,0) -- (2,3); 
    \draw[thick] (2,3) -- (1,3);
    \draw[thick] (1,3) -- (0,0);
    \filldraw[black] (0,0) circle (2pt) node[below left] {$(0,0)$};
    \filldraw[black] (1,0) circle (2pt) node[below right] {$(1,0)$};
    \filldraw[black] (1,3) circle (2pt) node[above left] {$(1,3)$};
    \filldraw[black] (2,3) circle (2pt) node[above right] {$(2,3)$};
    \end{tikzpicture}
\end{center}

\end{lemma}

\begin{proof}%\footnote{YL: I took away the longer proof, hope it is ok.}
Let $M'=\bZ^3$ be the standard lattice.
Let $\omega\subset M'_{\bR}=\bR^3$ be the cone generated by $(0,0,1)$, $(1,1,1)$, $(1,0,1)$, and $(0,1,1)$ corresponding to the variables $x$, $y$, $z$, and $w$. Then the toric variety $\tY:=\bC[x,y,z,w]/(xy-zw)$ is isomorphic to $\Spec \bC[\omega\cap M']$. Let $N':=\Hom(M',\bZ)=\bZ^3$ and $\sigma=\omega^\vee\subset N'_{\bR}=\bR^3$, then computations shows that $\sigma$ is generated by $(1,0,0)$, $(0,1,0)$, $(0,-1,1)$, and $(-1,0,1)$. The $\bmu_3$-action on $\tY$ induces an index $3$ sublattice $M\subset M'$ corresponding to $\bmu_3$-invariant monomials. It is clear that $xz$, $\frac{w}{x}$, and $x^3$ are $\bmu_3$-invariant. Thus $M$ is generated by $(1,0,2)$, $(0,1,0)$, $(0,0,3)$ which correspond to the above three $\bmu_3$-invariant monomials. Let $N=\Hom(M,\bZ)$ be the dual lattice. Then computing dual basis shows that $N$ is generated by $(1,0,0)$, $(0,1,0)$, and $(-\frac{2}{3}, 0, \frac{1}{3})$. We pick a new basis $\vec{v}_1=(1,-1,0)$, $\vec{v}_2=(-\frac{2}{3}, 0, \frac{1}{3})$, and $\vec{v}_3=(0,1,0)$ of $N$. Then under the basis $\vec{v}_i$, the cone $\sigma$ is generated by 
\[
\vec{v}_1+\vec{v}_3=(1,0,0),~ \vec{v_3}=(0,1,0),~ 2\vec{v}_1+3\vec{v}_2+\vec{v}_3=(0,-1,1),~\textrm{and } \vec{v}_1+3\vec{v}_2+\vec{v}_3=(-1,0,1).
\]
This shows that $Y=\tY/\bmu_3$ corresponds to the polytope $Q$ from the picture as 
\[Q=\{(a,b)\in\bR^2\mid a\vec{v}_1+b\vec{v}_2+\vec{v}_3\in \sigma\}.
\] 
The proof is finished.
\end{proof}

\begin{proposition}
The base of the miniversal deformation space for the singularity $0\in Y$ is a smooth curve.  In particular, the singularity $0\in Y$ has unobstructed deformations.
\end{proposition}

\begin{proof}
A more general result is proved in \cite[Proposition 4.4 (ii)]{Pet21}. We provide a proof here for readers' convenience.
The miniversal base space of a toric Gorenstein threefold singularity is determined by the corresponding two-dimensional polytope.  Indeed, in \cite[Theorem 5.1]{Altmann}, Altmann constructs a flat deformation over a base space $\widetilde{M}$ from the polytope, and proves it is the miniversal deformation space in \cite[Corollary 7.2]{Altmann}.  For more information on the construction of $\widetilde{M}$, see \cite[Definition 2.2, Theorem 2.4]{Altmann}.

To obtain equations for the space $\widetilde{M}$, label the edges of the polytope $Q$ in a counterclockwise fashion starting at the origin: $d^1 = (1,0)$, $d^2 = (1,3)$, $d^3 = (-1,0)$, $d^4 = (-1,-3)$.  Define the vector valued polynomial $g_k(t) = \sum_{i = 1}^4 t_i^kd^{i}$.  The inner products $\langle g_k(t) , (1,0) \rangle$ and $\langle g_k(t), (0,1) \rangle$ define two polynomials $g_{k,x}(t)$ and $g_{k,y}(t)$.  We define the ideal  \[ \mathcal{J} = \{ g_{k,x}(t), g_{k,y}(t) \mid k \ge 1 \} \subset \bC[t_1,t_2,t_3,t_4].\]   Let $M = \Spec \bC[t_1,t_2,t_3,t_4]/\mathcal{J} \subset \bA^4$.  By \cite[Theorem 7.4]{Altmann}, $M$ is defined by equations in $\bC[t_i - t_j]$, and the miniversal base space $\widetilde{M}$ is defined by $\mathcal{J} \cap \bC[t_i - t_j]$, or the pre-image of $M$ under the canonical projection $\bC^4 \to \bC^4 / \bC \cdot (1,1,1,1)$. 

Plugging in the values of $d^i$, we see that $\mathcal{J}$ is defined by the equations $t_1^k +t_2^k - t_3^k - t_4^k$ and $3t_2^k - 3t_4^k$, for $k \ge 1$ which reduces to $t_1 - t_3$ and $t_2 - t_4$.  Hence, $M = \Spec \bC[t_1,t_2]$, and $\widetilde M = \Spec \bC[t_1-t_2] \cong \bA^1$.  Therefore, the miniversal base space of the singularity $Y$ is a smooth curve, so $Y$ has unobstructed deformations. 
\end{proof}

The miniversal base space of the singular point of $X_{u,0}$ where $L_0$ deforms in a $\bQ$-Gorenstein family is given by the $\bZ/(4)$-invariant part of the miniversal base space of $Y$, which is a proper subspace of $\widetilde{M}$.  However, by construction  $x_0\in X_{u,0}$ deforms to $o\in X_u$ which admits a smoothing to $\bP^3$, so the miniversal base space must be at least one-dimensional.  Therefore, it must be all of $\widetilde{M}$ and we obtain the following corollary. 

\begin{corollary}\label{cor:local-deform}
The singularities $x_0\in X_{u,0}$ and $o\in X_u$ are formally isomorphic. As a result, the singular point $o\in X_u$ has unobstructed deformations where $L$ deforms in a $\bQ$-Gorenstein family.  
\end{corollary}

To finish the study of the deformation theory of $X_u$, we must show that there are no local-to-global obstructions in extending the local deformation to a global deformation of $X_u$, which is done by the following lemma.  

\begin{lemma}\label{lem:local-global-deform}
There are no local-to-global obstructions for deformations of $X_u$.
\end{lemma}

%\KD{If we decide to keep this section unmodified, I've added a sketch of a complete proof.  I can add more details if we decide to keep it, but didn't want to fill it in if we are removing this section.}

\begin{proof}
By \cite[Proposition 2.3]{Petracci}, it suffices to show that $H^2(X_u, T_{X_u}) = 0$.  

For simplicity, write $X = X_u$.  By construction of $X$, there is a small $\bQ$-factorialization $\pi: Z \to X$ such that $Z$ has $D_5$ singularities along an irreducible curve, contracted to the singular point $o \in X$.  The map $\psi: \tPE \to X$ factors through $\pi$ and the fibration structure of $\tPE \to \bP^1$ descends to a fibration $f:Z \to \bP^1$.  Furthermore, $f: Z \to \bP^1$ is an isotrivial fibration of $\bQ$-Fano surfaces $S$ each with an isolated $D_5$ singularity. In a neighborhood of any point $p$ in the singular locus of $Z$, the threefold looks like $S \times T$, where $T$ is a smooth curve.

From the five-term exact sequence from the Leray spectral sequence \[H^p(X,R^q \pi_*T_Z) \Rightarrow H^{p+q}(Z, T_Z)\] and the fact that $\pi_*T_Z = T_X$, there is a sequence 
\[ 0 \to H^1(X, T_X) \to H^1(Z,T_Z) \to H^0(X,R^1\pi_*T_Z) \to H^2(X, T_X) \to H^2(Z,T_Z). \]
In fact, this can be extended to 
\[ 0 \to H^1(X, T_X) \to H^1(Z,T_Z) \to H^0(X,R^1\pi_*T_Z) \to H^2(X, T_X) \to H^2(Z,T_Z) \to H^1(X,R^1\pi_*T_Z) \]
as in general the second-to-last term is $\ker H^2(Z,T_Z) \to H^0(X,R^2\pi_*T_Z)$, but the latter is $0$ as the fibers of $\pi$ have dimension at most 1.  
In the following Lemma, we will show that $R^1\pi_*T_Z = 0$, so we find that $H^1(X, T_X) = H^1(Z,T_Z)$ and $H^2(X,T_X) = H^2(Z,T_Z)$.  

Finally, we show that $H^2(Z,T_Z) = 0$. First note that $T_Z$ and its dual $\Omega^{[1]}_Z$ are Cohen-Macaulay: their restriction to any (Cartier) fiber $S$ of $f: Z \to \bP^1$ is $S_2$ and $\dim S = 2$, so the restriction to the Cartier fiber is Cohen Macaulay.  (To see that the restriction is $S_2$, one may use the sequence $0 \to \calO_S(-S) \to \Omega^{[1]}_Z|_S \to \Omega^{[1]}_S \to 0$; shown to be exact using the description of $Z$ as a fibration over $\bP^1$.)  By Serre Duality, then $H^2(Z,T_Z) = H^1(Z,\Omega^{[1]}_Z(K_Z))^\vee$.  

The proof that $H^1(Z,\Omega^{[1]}_Z(K_Z)) = 0$ then follows by the same logic used in \cite[Proposition 4]{Namikawa}.  The input \cite[Proposition 2]{Namikawa} holds (and the proof holds verbatim), replacing $Z$ by $X$ and $Y$ by $X_u$ in the author's notation, i.e. if $D$ is an anticanonical section of $X_u$ and $D'$ its strict transform on $Z$, we have $\Pic(Z) \to \Pic(D')$ is injective. 
 Therefore, from the standard exact sequence 
\[ 0 \to \mathbb{Z} \to \calO \to \calO^\times \to 0 \]
and vanishing of $H^i(Z, \calO_Z)$ for $i > 0$, we have an injection $H^2(Z, \mathbb{Z}) \to H^2(D', \mathbb{Z})$.  Because $Z$ has only quotient singularities, it is a $V$-manifold as in \cite[Theorem pg. 4]{Steenbrink}, and the Hodge structure is pure and $H^{1,1}(Z)$ is identified with $H^{1}(Z, \Omega_Z^{[1]})$.  By Hodge theory, the injection on cohomology then induces an injection on the parts of the Hodge decomposition $H^{p,q}(Z) \to H^{p,q}(D')$ with $p+q = 2$.   Therefore, we have an injection $H^1(Z, \Omega_Z^{[1]}) \to H^1(D', \Omega_{D'}^1)$.  

Noting that the exact sequences in the proof of \cite[Proposition 4]{Namikawa} are exact replacing $\Omega_Z^1$ with $\Omega_Z^{[1]}$ because $D'$ does not intersect $Z$, the argument shows that $H^2(Z,T_Z) = 0$.
%The same argument as that above for $T^1_Z$ shows that $R^1\pi_*(\Omega^{[1]}_Z(K_Z)) = 0$\footnote{\YL{not quite}} so the five-term exact sequence from the Leray spectral sequence gives $H^1(X,\Omega^{[1]}(K_X)) = H^1(Z,\Omega^{[1]}(K_Z))$.  Because $K_{X_u}$ is an anti-ample line bundle, the former is $0$ by the singular Kodaira-Akizuki-Nakano vanishing theorem in \cite[Proposition 4.3.2]{GKP}.  Therefore, $H^1(Z,\Omega^{[1]}_Z(K_Z)) = 0$, and hence $H^2(Z,T_Z) = 0$, and we conclude $H^2(X,T_X) = 0$.
\end{proof}

\begin{lem}\label{lem:R^1-vanishing}
In the notation as in the proof of Lemma \ref{lem:local-global-deform}, we have $R^1\pi_* T_Z = 0$.
\end{lem}

\begin{proof}
Recall that $\psi: \tPE\to X_u$ is taking the anti-canonical model. From the construction of $Z$, we know that $\psi$ factors as $\tPE\xrightarrow{\theta} Z\xrightarrow{\pi} X_u$ where both $\theta$ and $\pi$ are crepant birational. Since $\tPE$ and $Z$ are both Gorenstein canonical with quotient singularities, applying  \cite[Lemma 1.11]{Ste77} (see also \cite[Theorem 1.4]{GKKP11}) to a common log resolution of them yields that $\theta_* \Omega_{\tPE}^{[2]} \cong \Omega_Z^{[2]}$. Thus we have
\[
\theta_* T_{\tPE} = \theta_*(\Omega_{\tPE}^{[2]}\otimes \omega_{\tPE}^\vee) = \theta_*(\Omega_{\tPE}^{[2]}\otimes \theta^*\omega_{Z}^\vee) \cong (\theta_* \Omega_{\tPE}^{[2]} )\otimes \omega_{Z}^\vee \cong \Omega_Z^{[2]}\otimes \omega_{Z}^\vee = T_Z.
\]
From the first two terms of Leray spectral sequence, we have that $R^1\pi_*T_Z$ injects into $R^1\psi_* T_{\tPE}$. Thus it suffices to show $R^1 \psi_* T_{\tPE} = 0$.

Recall that $X_{u,0}=C_a(\bP^1\times\bP^1, -K_{\bP^1\times\bP^1} - \Gamma)$ is the orbifold cone where $\Gamma = \frac{1}{2}F_0 + \frac{1}{2}F_1 +\frac{2}{3}F_\infty$. Let $W$ be the total space of the orbifold line bundle $K_{\bP^1\times\bP^1} + \Gamma$, i.e.
\[
W := \Spec_{\bP^1\times\bP^1}\bigoplus_{m=0}^\infty \cO_{\bP^1\times\bP^1}(\lfloor m(-K_{\bP^1\times\bP^1} - \Gamma)\rfloor).
\]
Let $\Sigma$ be the zero section of $W$. Denote by $\Delta_\Sigma$ the different of $(W, \Sigma)$ on $\Sigma$. 
Thus we know that $\psi_0: W\to X_{u,0}$ provides the Koll\'ar component $(\Sigma, \Delta_\Sigma)\cong (\bP^1\times\bP^1, \Gamma)$.
From the proof of Lemma \ref{lem:X_u-indexonecover} and Corollary \ref{cor:local-deform} we know that the formal isomorphism between $X_u$ and $X_{u,0}$ lifts to a formal isomorphism between  $\tPE$ and $W$ along $\tH_{\cE}$ and $\Sigma$. Thus it suffices to prove $R^1\psi_{0,*} T_W =0$ which is equivalent to $H^1(W, T_W)=0$ as $X_{u,0}$ is affine. 

Recall from the proof of  Lemma \ref{lem:X_u-indexonecover} that there is a finite Galois morphism $\tau_{\tX}:\tX \to X_{u,0}$ where $\tX:= C_a(\bP^1\times\bP^1, \cO(2,2))$ and the Galois group is $\fS_3$. Let $\tW$ be the total space of the line bundle $\cO_{\bP^1\times\bP^1}(-2,-2)$ with zero section $\tSigma$. Then $\tpsi: \tW\to \tX$ is the blow-up of the cone vertex providing the Koll\'ar component $\tSigma \cong \bP^1\times\bP^1$. Then the quotient map $\tau_{\tX}$ lifts to a quotient map $\tau: \tW\to W$ by the action of $\fS_3$ which gives the following commutative diagram:
\[
\begin{tikzcd}
\tW \arrow[r,"\tau"]\arrow[d,"\tpsi"] & W\arrow[d,"\psi_0"]\\
\tX\arrow[r,"\tau_{\tX}"] & X_{u,0}
\end{tikzcd}
\]
From the proof of Lemma \ref{lem:X_u-indexonecover}, 
we know that $\tau$ is quasi-\'etale, and  $\tau|_{\tSigma}:\tSigma \to \Sigma$ is  the quotient map of the effective $\fS_3$-action on $\tSigma\cong \bP^1\times\bP^1$. Thus $\tau$ is also quasi-\'etale. By \cite[Theorem 3]{Kni73}, we know that $\Omega_{W}^{[1]}\cong (\tau_* \Omega_{\tW}^{[1]})^{\fS_3}$. Since $\tau$ is quasi-\'etale and $\Omega_{\tW}^{[1]}$ is reflexive, we know that $\tau_*\Omega_{\tW}^{[1]}$ is also reflexive. This implies that $((\tau_* \Omega_{\tW}^{[1]})^{\fS_3})^\vee \cong ((\tau_* \Omega_{\tW}^{[1]})^\vee)^{\fS_3}$. Since $\psi: \tPE \to X_{u,0}$ is crepant between Gorenstein normal varieties, so is $\psi_0: W\to X_{u,0}$. In particular, $W$ is normal and Gorenstein. By construction, $\tW$ is smooth hence also normal and Gorenstein.
Applying Lemma \ref{lem:Gduality} to the morphism $\tau$ and sheaf $\Omega_{\tW}^{[1]}$, we have that
\[
T_W = (\Omega_{W}^{[1]})^\vee \cong ((\tau_* \Omega_{\tW}^{[1]})^{\fS_3})^\vee \cong ((\tau_* \Omega_{\tW}^{[1]})^\vee)^{\fS_3}\cong (\tau_* ((\Omega_{\tW}^{[1]})^\vee))^{\fS_3} = (\tau_* T_{\tW})^{\fS_3}.
\]
Since  $T_W = (\tau_* T_{\tW})^{\fS_3}$ is a direct summand of $\tau_* T_{\tW}$, we have that $H^1(W, T_W)$ is a direct summand of $H^1(W, \tau_* T_{\tW})\cong H^1(\tW, T_{\tW})$. Thus it suffices to show $H^1(\tW, T_{\tW})=0$.

Denote by $f_{\tW}:\tW\to \bP^1\times\bP^1$ the $\bA^1$-bundle structure. Then we have a short exact sequence 
\[
0\to T_{\tW/\bP^1\times\bP^1} \to T_{\tW} \to f_{\tW}^* T_{\bP^1\times\bP^1}\to 0.
\]
By the long exact sequence of cohomology, it suffices to show the vanishing of both $H^1(\tW, T_{\tW/\bP^1\times\bP^1})$ and $H^1(\tW, f_{\tW}^* T_{\bP^1\times\bP^1})$.

Denote by $L_{\tW}:= \cO_{\bP^1\times\bP^1}(2,2)$. Since $\tW$ is the total space of the line bundle $L_{\tW}^\vee$, we know that $f_{\tW,*}\cO_{\tW}\cong \oplus_{m=0}^\infty L_{\tW}^{\otimes m}$, and $T_{\tW/\bP^1\times\bP^1}=f_{\tW}^* L_{\tW}^\vee$. Thus we have
\[
H^1(\tW, T_{\tW/\bP^1\times\bP^1}) = H^1(\tW, f_{\tW}^* L_{\tW}^\vee) \cong H^1(\bP^1\times\bP^1, f_{\tW,*}f_{\tW}^* L_{\tW}^\vee)\cong  H^1(\bP^1\times\bP^1, \bigoplus_{m=-1}^\infty L_{\tW}^{\otimes m}).
\]
By Kodaira vanishing  we know that $H^1(\bP^1\times\bP^1, L_{\tW}^{\otimes m}) = H^1(\bP^1\times\bP^1, \cO(2m, 2m)) = 0$ for every $m\in \bZ$. Thus we get the vanishing of $H^1(\tW, T_{\tW/\bP^1\times\bP^1})$. On the other hand, we have
\[
H^1(\tW, f_{\tW}^* T_{\bP^1\times\bP^1}) \cong H^1(\bP^1\times\bP^1, f_{\tW,*}f_{\tW}^* T_{\bP^1\times\bP^1}) \cong H^1(\bP^1\times\bP^1, \bigoplus_{m=0}^\infty L_{\tW}^{\otimes m}\otimes T_{\bP^1\times\bP^1}).
\]
Since $T_{\bP^1\times\bP^1}\cong T_{\bP^1}\boxtimes T_{\bP^1}\cong \cO(2,0)\oplus \cO(0,2)$, by Kodaira vanishing we have 
\[
H^1(\bP^1\times\bP^1, L_{\tW}^{\otimes m}\otimes T_{\bP^1\times\bP^1}) \cong H^1(\bP^1\times\bP^1, \cO(2m+2,2m)\oplus\cO(2m,2m+2)) = 0
\]for every $m\geq 0$. Thus the vanishing of $H^1(\tW, f_{\tW}^* T_{\bP^1\times\bP^1})$ is proved. The proof is finished.
\end{proof}

\begin{lem}\label{lem:Gduality}
Let $f: X\to Y$ be a quasi-\'etale finite morphism between normal Gorenstein varieties. Let $\cG$ be a  coherent sheaf on $X$. Then we have
$(f_* \cG)^\vee = f_* (\cG^\vee) $. 
\end{lem}

\begin{proof}
Since $f$ is quasi-\'etale and both $X$ and $Y$ are Gorenstein, we know that $\omega_{X} = f^* \omega_Y = f^{!}\omega_Y$. By Grothendieck duality for finite morphisms (see e.g. \cite[\href{https://stacks.math.columbia.edu/tag/0AU3}{Tag 0AU3}]{stacks-project}), we have
\begin{equation}\label{eq:Gduality}
f_*\cHom_{\cO_{X}}(\cG\otimes \omega_X, \omega_X) = \cHom_{\cO_Y}(f_*(\cG\otimes \omega_X),\omega_Y).
\end{equation}
Since $\omega_X$ is invertible, the left-hand side of \eqref{eq:Gduality} is
\[
f_*\cHom_{\cO_{X}}(\cG\otimes \omega_X, \omega_X) = f_*\cHom_{\cO_{X}}(\cG, \cO_X) = f_*(\cG^\vee).
\]
For the right-hand side of \eqref{eq:Gduality}, we get
\begin{align*}
    \cHom_{\cO_Y}(f_*(\cG\otimes \omega_X),\omega_Y) & = \cHom_{\cO_Y}(f_*(\cG\otimes f^*\omega_Y),\omega_Y)\\
    & = \cHom_{\cO_Y}((f_*\cG)\otimes \omega_Y,\omega_Y) \\
    & = \cHom_{\cO_Y}(f_*\cG,\cO_Y) \\& = (f_* \cG)^\vee.
\end{align*}
Here we use projection formula and the fact that $\omega_Y$ is invertible.
The proof is finished by combining the above equalities.
\end{proof}

\begin{corollary}\label{cor:X_u-global-deform}
The $\bQ$-Fano threefold $X_u$ has unobstructed deformations where $L$ deforms in a $\bQ$-Gorenstein family, and the miniversal base space is a smooth curve. Moreover, any small deformation of $X_u$ in $\ocM^{{\rm sm},\delta\geq \epsilon_0}_{3,64}$ is isomorphic to $\bP^3$ or $X_u$.
\end{corollary}

\begin{proof}
This follows from Corollary \ref{cor:local-deform} and Lemma \ref{lem:local-global-deform}.
\end{proof}

\begin{defn-prop}\label{d-p:H_u,c}
Let $c\in (\frac{9}{13},1)$ be a rational number.
Let $\sH_{u,c}$ be the locally closed substack of $\osM_{c}^{\K}$ with reduced structure parametrizing K-semistable pairs $(X,cS)$ where $X\cong X_u$. Then $\sH_{u,c}$ is a closed substack of $\osM_c^{\K}$. 

Denote by $H_{u,c}$ the good moduli space of $\sH_{u,c}$. Then $H_{u,c}$ is a closed subscheme of $\ofM_c^{\K}$ that is isomorphic to the GIT quotient $\bfP\sslash \SL(2,\bC)$. 
\end{defn-prop}

\begin{proof}
We first show that $\sH_{u,c}$ is closed which follows from the existence part of valuative criterion for properness of the map $\sH_{u,c}\hookrightarrow \osM_c^{\K}$. Let $(\cX,c\cS)\to C$ be a $\bQ$-Gorenstein family of K-semistable log Fano pairs over a smooth pointed curve $0\in C$ such that $K_{\cX/C}+\cS\sim_{C} 0$ and $\cX_t\cong X_u$ for any $t\in C\setminus \{0\}$. It suffices to show that $\cX_0\cong X_u$ as well. 

Denote by $C^\circ:=C\setminus \{0\}$, $\cX^\circ:=\cX\times_C C^\circ$, and $\cS^\circ:=\cS|_{\cX^\circ}$. After replacing $(0\in C)$ by a quasi-finite cover if necessary, we may assume that $\cX^\circ\cong X_u\times C^\circ$. Recall from the proof of Theorem \ref{thm:GIT=K-unigonal}(2) that there is a universal family $(X_u\times \bfA,\cS_{\bfA})\to \bfA$ parametrizing $(X_u, S_{(A,B)})$ for $(A,B)\in \bfA$.
Hence by a family version of Lemma \ref{lem:X_u-anti-can}, we can find a map $\gamma^\circ: C^\circ\to \bfA$ such that $(\cX^\circ, \cS^\circ)\cong (X_u\times \bfA,\cS_{\bfA})\times_{\gamma^\circ} C^\circ $. Since $(\cX_t, c\cS_t)$ is K-semistable for $t\in C^\circ$, Theorem \ref{thm:GIT=K-unigonal}(2) implies that $\cS_t$ is the birational transform of $S_{(A,B)}$ where $(A,B)\in \bfA\setminus \{0\}$ is GIT semistable. Hence $\gamma^\circ$ factors as $\gamma^\circ: C^\circ\to \bfA^{\rm ss}\hookrightarrow \bfA$. Since $\bfA^{\rm ss}\sslash (\SL(2,\bC)\times\bG_m)\cong \bfP\sslash \bG_m$ is proper, after replacing $(0\in C)$ by a further quasi-finite base change we can find $g: C^\circ\to \SL(2,\bC)\times \bG_m$ and $\gamma': C\to \bfA^{\rm ss}$ such that $\gamma'|_{C^\circ}=  g \cdot \gamma^\circ$. In particular, we have a K-semistable log Fano family $(X_u\times C,c\cS'):=(X_u\times\bfA,c\cS_{\bfA})\times_{\gamma'} C$ over $C$ such that $(X_u\times C,c\cS')\times_C C^\circ\cong(\cX^\circ, c\cS^\circ)$. Thus $(X_u,c\cS_0')$ and $(\cX_0, c\cS_0)$ are $\mathrm{S}$-equivalent K-semistable log Fano pairs, and by \cite{BX18} they admit a common K-polystable  degeneration $(X'', c S'')$. By Theorem \ref{thm:GIT=K-unigonal}(2), we have that $X''\cong X_u$ and $S''$ is the GIT polystable degeneration of $\cS_0'$. Hence $\cX_0$ is isomorphic to $X_u$ as it not only specially degenerates to $X_u$ but also comes from an isotrivial degeneration of $X_u$. Thus $\sH_{u,c}$ is a closed substack of $\osM_c^{\K}$.

Finally, we show that $H_{u,c}$ is isomorphic to $\bfP\sslash \SL(2,\bC)$. In fact, the universal family $(X_u\times\bfA, \cS_{\bfA})\times_{\bfA} \bfA^{\rm ss}$ parameterizes $c$-K-semistable log Fano pairs by Theorem \ref{thm:GIT=K-unigonal}(2). After taking quotient of $\SL(2,\bC)\times \bG_m$, we get a stack morphism $[\bfA^{\rm ss}/(\SL(2,\bC)\times\bG_m)]\to \sH_{u,c}$. Thus taking good moduli spaces yields a morphism $\bfP\sslash \SL(2,\bC)\to H_{u,c}$ which is bijective by Theorem \ref{thm:GIT=K-unigonal}. By Corollary \ref{cor:X_u-global-deform} we know that $\sH_{u,c}$ is smooth, so $H_{u,c}$ is normal. Therefore, $\bfP\sslash \SL(2,\bC)\to H_{u,c}$ is an isomorphism.
\end{proof}

\begin{thm}\label{thm:T_0-replace}
The K-moduli spaces $\ofM_{c}^{\K}$ has a wall at $c=\frac{9}{13}$. Moreover, we have 
\begin{enumerate}
    \item The wall crossing morphism $\phi^-:\ofM_{\frac{9}{13}-\epsilon}^{\K}\to \ofM_{\frac{9}{13}}^{\K}$ replaces $[(\bP^3, T)]$ by $[(X_u, T_0)]$, and is isomorphic near $[(X_u, T_0)]$.
    \item The wall crossing morphism $\phi^+:\ofM_{\frac{9}{13}+\epsilon}^{\K}\to \ofM_{\frac{9}{13}}^{\K}$ replaces $[(X_u, T_0)]$ by the divisor $H_{u, \frac{9}{13}+\epsilon}$.
    \item For any $c\in (\frac{9}{13},1)$, the birational map $\ofM_{c}^{\K}\dashrightarrow\ofM_{\frac{9}{13}+\epsilon}^{\K}$ is an isomorphism over a neighborhood of $H_{u, \frac{9}{13}+\epsilon}$.
\end{enumerate}
%In particular, the birational map $(\phi^-)^{-1}\circ \phi^+:\ofM_{\frac{9}{13}+\epsilon}^{\K}\dashrightarrow\ofM_{\frac{9}{13}-\epsilon}^{\K}$ is a weighted blow-up at $[(\bP^3, T)]$ near this point.

\end{thm}

\begin{proof}
(1) Let $U_T:=\ofM^{\GIT}\setminus W_8$ be an open neighborhood of $[T]$. By \eqref{eq:stratification}, we know that any $[S]\in U_T\setminus \{[T]\}$ is slc, thus $(\bP^3, cS)$ is K-stable for any $c\in (0,1)$. Since $\kst(\bP^3, T)=\frac{9}{13}$ by Theorem \ref{thm:tangent-kst}, there are open immersions $U_T\hookrightarrow \ofM_c^{\K}$ when $c\in (0,\frac{9}{13})$ and $U_T\setminus \{[T]\}\hookrightarrow \ofM_c^{\K}$ when $c\in [\frac{9}{13},1)$. Thus the map $\phi^-:\ofM_{\frac{9}{13}-\epsilon}^{\K}\to \ofM_{\frac{9}{13}}^{\K}$ is isomorphic on $U_T\setminus \{[T]\}$. On the other hand, we know that $\phi^-([(\bP^3, T)])=[(X_u, T_0)]$ by Proposition \ref{prop:tangent-Kps-replace}. By Corollary \ref{cor:X_u-global-deform}, we know that $\ocM_{3,64}^{{\rm sm},\delta\geq \epsilon_0}$ is smooth in a neighborhood of $[X_u]$. Thus $\ofM_{\frac{9}{13}}^{\K}$ is normal near $[(X_u, T_0)]$ by Lemma \ref{lem:forgetful-smooth}. By Zariski's main theorem, we know that $(\phi^-)^{-1}([(X_u, T_0)])$ is  connected, thus it has to be the singleton $\{[(\bP^3, T)]\}$. Hence $\phi^-$ is an isomorphism near $[(X_u, T_0)]$.

(2) By Theorem \ref{thm:GIT=K-unigonal}(1) we know that $\phi^+$ contracts $H_{u, \frac{9}{13}+\epsilon}$ to the point $[(X_u, T_0)]$. It suffices to show that any $[(X,S)]\in \ofM_{\frac{9}{13}+\epsilon}^{\K}$ whose $\frac{9}{13}$-K-polystable replacement is $[(X_u, T_0)]$  satisfies $X\cong X_u$. 
By Corollary \ref{cor:X_u-global-deform}, we know that $X$ is isomorphic to $\bP^3$ or $X_u$. If $X\cong \bP^3$, then by interpolation \cite[Proposition 2.13]{ADL19} we know that $(X,\frac{9}{13}S)$ is also K-polystable, a contradiction to the uniqueness of K-polystable degenerations \cite{LWX18}. Hence we have $X\cong X_u$.

(3) Let $U_u:= (\phi^+)^{-1}(\phi^-(U_T))$. By parts (1) and (2), we have $U_u = (U_T\setminus \{[T]\})\sqcup H_{u, \frac{9}{13}+\epsilon}$ as sets. Since every $[S]\in U_T\setminus \{[T\}$ is slc and GIT polystable, Proposition \ref{prop:k-moduli-irred} implies that $(\bP^3, cS)$ is always K-polystable for $c\in (0,1)$. Moreover, Theorem \ref{thm:GIT=K-unigonal}(2) implies that any pair $(X_u, S)$ in $H_{u, \frac{9}{13}+\epsilon}$ is $c$-K-polystable for any $c\in (\frac{9}{13}, 1)$. Thus there are open immersions $U_u \hookrightarrow \ofM_{c}^{\K}$ for any $c\in (\frac{9}{13},1)$, which implies that the birational map $\ofM_c^{\K}\dashrightarrow\ofM_{\frac{9}{13}+\epsilon}^{\K}$ is an isomorphism over $U_u$. The proof is finished.
%Since $T_0$ does not contain the unique singular point $o$ of $X_u$, we know that $o\not\in S$. Hence Lemma \ref{lem:X_u-anti-can} implies that after an automorphism of $X_u$, the surface $S$ is a Weierstrass elliptic surface corresponding to a pair $(A,B)\in \bfA$. Since $(X_u, (\frac{9}{13}+\epsilon)T_0)$ is K-unstable, we know that  $(A,B)\neq 0$. Hence $(A,B)$ is GIT polystable by Theorem \ref{thm:GIT=K-unigonal}. The proof is finished.
\end{proof}

\section{Hyperelliptic K3 surfaces}\label{sec:hyperelliptic}

In this section, we will use the results from \cite{ADL20} to study K-polystable replacements of the locus $W_i$ in $\ofM^{\GIT}$ for $i\in \{0,1,2,3,4,6,7,8\}$ (see Section  \ref{sec:GIT-quartic} for the definition). We will show that the first K-moduli wall crossing extracts the divisor $H_h$ birationally over the point $W_0=\{[2Q]\}$, and subsequential wall crossings precisely replace $W_i~(i\geq 1)$  by  $Z^{i+1}\subset \sF$ inside the hyperelliptic divisor $H_h$ as introduced in Section \ref{sec:laza-ogrady}. 

%In the rest of this article, we denote by $X_h:= C_p(\bP^1\times\bP^1, \cO_{\bP^1\times\bP^1}(2,2))$ as the projective anti-canoncial cone over $\bP^1\times\bP^1$. 

\subsection{A cone construction for hyperelliptic K3 surfaces}\label{sec:cone}

In this subsection, we provide a cone construction to produce K-polystable threefold pairs from K-polystable surface pairs. This is very useful in constructing the K-polystable replacements of the locus $W_i$ in the GIT moduli space based on the replacements from \cite{ADL20}.

\begin{defn}\label{def:cone}
Let $V$ be a Gorenstein log Del Pezzo surface, that is, a $\bQ$-Fano variety of dimension $2$ with $K_V$ Cartier.\footnote{In later discussions, we will very often assume that $V$ is isomorphic to either $\bP^1\times\bP^1$ or $\bP(1,1,2)$.} Let $C\in |-2K_V|$ be an effective Cartier divisor defined by a section $s_C \in H^0(V,-2K_V)$ on $V$. Let $X:=C_p(V, -K_V) = \Proj \left(\oplus_{m \ge 0} \oplus_{r = 0}^m H^0(V, -rK_V)t^{m-r}\right)$ be the projective cone. Let $S$ be the double cover of $V$ branched along $C$; i.e. $S = (t^2 = s_C)$. Then $S$ is naturally embedded into $X$ as an anti-canonical divisor. We denote this construction by $\sC(V, cC):=(X,\frac{4c+1}{3}S)$ where $c\in [0,\frac{1}{2}]$ is a rational number.
\end{defn}

\begin{thm}\label{thm:induce-3-fold}
With the above notation, let $c\in [0, \frac{1}{2})$ be a rational number. Then 
$(V,cC)$ is K-semistable (resp. K-polystable) if and only if $(X,\frac{4c+1}{3}S)=\sC(V,cC)$ is K-semistable (resp. K-polystable).
\end{thm}

\begin{proof}
We first treat the ``only if'' part. 
Assume that $(V, cC)$ is K-semistable. Let $V_X\subset X$ be the section at infinity. Let $Y:=C_p(V,-2K_V)$ be a new projective cone with $V_Y$ the section at infinity. Then there exists a finite morphism $\pi: X\to Y$ as a double cover branched along $V_Y$ and $\pi^*V_Y=2V_X$. Denote by $\tau:X\to X$ the involution induced by $\pi$. Then it is clear that $S$ is $\tau$-invariant. Denote by $D:=S/\tau$ as a divisor in $Y$. Clearly $D$ corresponds to a section of $Y$ such that $D|_{V_Y}=C$. The finite morphism $\pi$ is crepant between $(X, \frac{4c+1}{3})$ and $(Y, \frac{1}{2}V_Y+\frac{4c+1}{3}D)$. Hence by \cite{LZ20, Zhu20} it suffices to show that $(Y, \frac{1}{2}V_Y +\frac{4c+1}{3} D)$ is K-semistable. The  natural $\bG_m$-action on $Y$ degenerates $D$ to $D_0$ as the cone over $C$. Let $r:=\frac{1}{2}-c$ be a positive rational number, hence $-2K_V\sim_{\bQ}r^{-1}(-K_V-cC)$. By \cite[Proposition 5.3]{LX16}, we know that $(Y, (1-\frac{r}{3})V_Y + cD_0)$  is K-semistable. Since $1-\frac{r}{3}=\frac{5}{6}+\frac{c}{3}$, we know that $(Y, (\frac{5}{6}+\frac{c}{3})V_Y + cD)$ is K-semistable since it admits a K-semistable special degeneration. Similarly, since $D$ is also section of $Y$, the roles of $V_Y$ and $D$ are interchangeable and we could alternatively degenerate $V_Y$ to ${V_Y}_0$, we know that
$(Y, cV_Y + (\frac{5}{6}+\frac{c}{3}) D)$ is also K-semistable. We know that $(\frac{1}{2}, \frac{4c+1}{3})$ is a convex linear combination of $((\frac{5}{6}+\frac{c}{3}), c)$ and $(c,(\frac{5}{6}+\frac{c}{3}))$ since the sum of two components are the same and $c<\frac{1}{2}<\frac{5}{6}+\frac{c}{3}$. Hence by interpolation (\cite[e.g. Proposition 2.13]{ADL19}), we conclude that $(Y, \frac{1}{2}V_Y+\frac{4c+1}{3}D)$ is K-semistable. Hence $(X, \frac{4c+1}{3}S)$ is K-semistable.

Next we assume that $(V,cC)$ is K-polystable. Since $\pi$ is a Galois morphism, by \cite{LZ20, Zhu20} it suffices to show that $(Y, \frac{1}{2}V_Y+\frac{4c+1}{3}D)$ is K-polystable. By \cite{LWX18} we can choose a special test configuration $(\cY, \frac{1}{2}\cV+\frac{4c+1}{3}\cD)$ of $(Y, \frac{1}{2}V_Y+\frac{4c+1}{3}D)$ whose central fiber $(Y',\frac{1}{2}V'+\frac{4c+1}{3}D')$ is K-polystable. In particular, $\Fut(\cY, \frac{1}{2}\cV+\frac{4c+1}{3}\cD)=0$. Denote by $b:=\frac{5}{6}+\frac{c}{3}$.
By linearity of the generalized Futaki invariant in coefficients and K-semistability of $(Y, bV_Y+cD)$ and $(Y, cV_Y+bD)$, we know that 
\begin{equation}\label{eq:Fut-linear}
\Fut(\cY, b\cV+c \cD)=
\Fut(\cY, c \cV+b\cD)=0.
\end{equation}
By \cite[Theorem 1.4]{LWX18}, we know that the analogous statement of \cite[Proposition 5.3]{LX16} for K-polystability is true (see also \cite[Proposition 2.11]{LZ19}). Hence $(Y, bV_Y+cD_0)$ is the K-polystable since $(V,cC)$ is K-polystable. In particular, $(Y, bV_Y+cD_0)$ is the K-polystable special degeneration of  $(Y, bV_Y+cD)$. By \cite{LX14} and \cite[Lemma 3.1]{LWX18}, \eqref{eq:Fut-linear} implies that $(Y', bV'+cD')$ is a K-semistable special degeneration of $(Y, bV_Y+cD)$. Thus \cite[Theorem 1.3]{LWX18} implies that $(Y, bV_Y+cD_0)$ is isomorphic to the K-polystable special degeneration of $(Y', bV'+cD')$. Thus we have a sequence of special degenerations 
\begin{equation}\label{eq:specialdeg-seq}
(Y,bV_Y+cD)\rightsquigarrow (Y', bV'+cD')\rightsquigarrow (Y, bV_Y+cD_0).
\end{equation}
By forgetting $D$, $D'$, and $D_0$, we obtain $(Y,bV_Y)\rightsquigarrow (Y',bV')\rightsquigarrow (Y,bV_Y)$. This implies that $(Y',V')\cong (Y,V_Y)$. Similarly, since $V_Y$ and $D$ are symmetric, using the second equality in \eqref{eq:Fut-linear} we have that $(Y',D')\cong (Y,V_Y)$. Thus $Y'\cong Y=C_p(V,-2K_V)$ where both $V'$ and $D'$ are sections in $Y'$. Moreover, since $D|_{V_Y}=D_0|_{V_Y}$, after restricting \eqref{eq:specialdeg-seq} to $V_Y$ and $V'$ we see that $(V,C)\cong (V_Y, D|_{V_Y})\cong (V',D'|_{V'})$. Hence $(Y', V'+D')\cong (Y,V_Y+D)$ 
% Hence by forgetting $V'$ and $D'$ respectively, we obtain that 
% \[
% (Y',V')\cong(Y,V_Y)\cong (Y,D)\cong (Y',D').
% \]
% Hence $(Y', V'+D')\cong (Y, V_Y +D_1)$ where $D_1$ is a section of $Y$. By restricting $(\cY, \cV+\cD)$ to $V_Y$ we see that $(V_Y,D_1|_{V_Y})\cong (V,C)$. Hence
which implies that $(Y, \frac{1}{2}V_Y+\frac{4c+1}{3}D) $ is K-polystable.

Next we treat the ``if'' part for K-semistability. Assume that $(X,\frac{4c+1}{3}S)$ is K-semistable, which implies the K-semistability of $(Y, \frac{1}{2}V_Y + \frac{4c+1}{3}D)$ from the above discussion. It suffices to show that $(V,cC)$ is K-semistable.  Assume to the contrary that $(V,cC)$ is K-unstable. By Theorem \ref{thm:valuative}, there exists a prime divisor $E$ over $V$ such that $\beta_{(V,cC)}(E)<0$. Let $v_t$ be the quasi-monomial valuation on $Y$ obtained by taking the $(1,t)$-linear combination of $\ord_{V_Y}$ and $\ord_{E_\infty}$ where $E_\infty$ is a prime divisor over $Y$ by taking cone over $E$. For simplicity, denote by $\Delta:= \frac{1}{2}V_Y+ \frac{4c+1}{3}D$. Then a simple computation shows that
$A_{Y}(v_t)= 1+tA_V(E)$, $v_t(V_Y)=1$, and $v_t(D)= \min \{1, t\ord_E(C)\}$.
Hence we have
\begin{equation}\label{eq:A-cone}
A_{(Y,\Delta)}(v_t)=\frac{1}{2}+tA_{(V,cC)}(E)+ct\ord_E(C)-\frac{4c+1}{3}\min\{1,t\ord_E(C)\}.    
\end{equation}
Next we compute $S_{(Y, \Delta)}(v_t)$. Let $L_Y:=\cO_Y(V_Y)$. Then for $m\in\bZ_{>0}$ it is clear that 
\[
H^0(Y, mL_Y)\cong \oplus_{i=0}^m H^0(V, -2iK_V)\cdot s_0^{m-i},
\]
where $(s_0=0)$ represents the divisor $V_Y$, and $s\in H^0(V, -2iK_V)$ corresponds a section in $H^0(Y, iL_Y)$ by taking the cone. We have that $v_t(s\cdot s_0^{m-i})= t\ord_E(s)+ (m-i)$ for each non-zero section $s\in H^0(V, -2iK_V)$. Hence we have
\begin{equation}\label{eq:S_m-cone}
S_{L_Y, m}(v_t)=\frac{1}{mh^0(Y, mL_Y)}\sum_{i=0}^{m} h^0(V,-2iK_V) ( 2i t S_{-K_V,m}(\ord_E) +(m-i)).    
\end{equation}
Here we refer to \cite{BJ17} for the definition and properties of $S_m$-invariants. It is clear that $h^0(Y,mL_Y)\sim \frac{1}{6}\vol_V(-2K_V)m^3$ and $h^0(V,-2iK_V)\sim \frac{1}{2}\vol_V(-2K_V)i^2$ as $m,i\to \infty$. By taking limit of $\eqref{eq:S_m-cone}$ as $m\to \infty$, we have that 
\begin{equation}\label{eq:S-cone-1}
S_{L_Y}(v_t)=\frac{3}{2}tS_{-K_V}(E)+\frac{1}{4}.
\end{equation}
Since $-K_Y-\Delta\sim_{\bQ}\frac{2-4c}{3}L_Y$ and $-K_V-cC\sim_{\bQ} (1-2c)(-K_V)$, \eqref{eq:S-cone-1} implies that 
\begin{equation}\label{eq:S-cone-2}
S_{(Y,\Delta)}(v_t)=tS_{(V,cC)}(E)+\frac{1-2c}{6}. 
\end{equation}
Combining \eqref{eq:A-cone} and \eqref{eq:S-cone-2}, we have that 
\begin{equation}\label{eq:beta-cone}
\beta_{(Y,\Delta)}(v_t)=t\beta_{(V,cC)}(\ord_E)  +ct\ord_E(C)+\frac{c+1}{3} - \frac{4c+1}{3}\min\{1,t\ord_E(C)\}. 
\end{equation}   
Recall that $\beta_{(V,cC)}(\ord_E)<0$ by our assumption.
If $\ord_E(C)=0$, then we see that $\beta_{(Y,\Delta)}(v_t)<0$ for $t\gg 0$ which implies that $(Y,\Delta)$ is K-unstable by Theorem \ref{thm:valuative}, a contradiction. If $\ord_E(C)\neq 0$, we choose $t=\frac{1}{\ord_E(C)}$. Then \eqref{eq:beta-cone} implies that $\beta_{(Y,\Delta)}(v_t)=t\beta_{(V,cC)}(\ord_E)<0$, again a contradiction. Thus the ``if'' part for K-semistability is proven.

Finally, we treat the ``if'' part for K-polystability. Assume that $(X,\frac{4c+3}{3}S)$ is K-polystable, which implies the K-polystability of $(Y, \frac{1}{2}V_Y + \frac{4c+1}{3}D)$ from the above discussion. Assume to the contrary that $(V,cC)$ is not K-polystable. By the ``if'' part for K-semistability, we know that $(V,cC)$ is K-semistable. Let $(V', cC')$ be a K-polystable degeneration of $(V,cC)$. By the ``only if'' part, we may use the cone construction over $(V',cC')$ to obtain a K-polystable log Fano pair  $(Y', \frac{1}{2}V'_{Y'} +\frac{4c+1}{3}D')$. Then we may take the cone of $(V,cC)\rightsquigarrow (V',cC')$ as in \cite[Proof of Proposition 2.11]{LZ19} to produce a K-polystable degeneration $(Y, \frac{1}{2}V_Y+\frac{4c+1}{3}D)\rightsquigarrow (Y', \frac{1}{2}V'_{Y'} +\frac{4c+1}{3}D')$ which has to be a product test configuration since $(Y, \frac{1}{2}V_Y+\frac{4c+1}{3}D)$ is K-polystable. By restricting to $V_Y\rightsquigarrow V'_{Y'}$ we have that $(V,cC)\cong (V',cC')$ is K-polystable. 
The proof is finished.
\end{proof}

We will apply the cone construction to hyperelliptic degree 4 K3 surfaces, when $V$ is either $\bP^1 \times \bP^1$ or $\bP(1,1,2)$.  In this case, the cone constructed in Definition \ref{def:cone} is either the cone over the anticanonical embedding of the smooth quadric, which we denote by $X_h = C_p(\bP^1 \times \bP^1, \cO(2,2))$, or the cone over the anticanonical embedding of the singular quadric, which is the weighted projective space $\bP(1,1,2,4) = C_p(\bP(1,1,2), \cO(4))$.  

\begin{prop}\label{prop:cone-vertex-unstable}
Let $X$ be a $\bQ$-Fano threefold that is isomorphic to either $X_h$ or $\bP(1,1,2,4)$.  Let $S\sim -K_X$ be an effective Cartier divisor on $X$. If $S$ passes through the cone vertex $o$ of $X$, then $(X,cS)$ is K-unstable for any $c\in [0,1)$. %\footnote{YL: Determine later whether this prop is used or not.}
\end{prop}

\begin{proof}
Let $\tX\to X$ be the partial resolution by blowing up the cone vertex. Denote by $E\subset \tX$ the exceptional divisor. Since $X\cong C_p(V, -K_V)$ where  $V$ is $\bP^1\times\bP^1$ or $\bP(1,1,2)$, we know that $A_X(E)=1$ and $\vol_{X,o}(E)= (-K_V)^2=8$.  Since $S$ is Cartier and passes through $o=c_X(E)$, we know that $\ord_E(S)\geq 1$. Hence we have 
\[
\hvol_{(X,cS),o}(E)=( A_{X}(E)-c\ord_E(S))^3 \vol_{X,o}(E)\leq 8(1-c)^3.
\]
On the other hand, we have $\frac{27}{64}(-K_X-cS)^3= 27(1-c)^3>\hvol_{(X,cS),o}(E)$,
which implies that $(X,cS)$ is K-unstable by \cite{LL16}.%\footnote{YL: Probably we can avoid normalized volume, just to save space.}
\end{proof}

% \begin{remark}
% Right now it is reasonable to believe that the K-polystable replacements away from the tangent develop surface have ambient threefolds isomorphic to either $C_p(\bP^1\times\bP^1, \cO(2,2))$ or $\bP(1,1,2,4)$. Indeed, all exceptional loci $E_i^+$ corresponding to walls induced from hyperelliptic ones should be contained in the closure of strict transform of $H_h$. In other words, the $18$-dimensional K-moduli wall crossing can be embedded (up to finite inertia) into the $19$-dimensional K-moduli wall crossing as a divisor under the transformation rule $c\mapsto \frac{4c+1}{3}$. At this point, we already verified that the predicted walls in Prop. \ref{prop:walls-predict} indeed do occur.\footnote{YL: Consider moving this remark somewhere else}
% \end{remark}

\subsection{Deformation theory of cones}

In this subsection, we show that the allowed deformations of $X_h$ or $\bP(1,1,2,4)$ are unobstructed, and any such small deformation of $X_h$ (resp. $\bP(1,1,2,4)$) is isomorphic to itself or $\bP^3$ (resp. itself, $X_h$, or $\bP^3$). Recall that $X_h$ is the projective anti-canonical cone over $\bP^1\times\bP^1$.

\begin{lem}\label{lem:Xh-unobstructed}
The $\bQ$-Gorenstein deformations of $X_h$ or $\bP(1,1,2,4)$ in the moduli stack $\ocM_{3,64}^{{\rm sm},\delta\geq \epsilon_0}$ of $\bQ$-Gorenstein smoothable $\bQ$-Fano varieties are unobstructed. 
\end{lem}

\begin{proof}
We consider the deformations only in the smoothable locus, in particular, we only consider deformations induced by the index 1 cover of $L$, the limit of $\calO_{\bP^3}(1)$.  In this setting, the deformation theory in \cite[Section 3]{Hac01} applies and the obstructions are contained in $T^2_{QG,X}$, defined as follows.  Let $\pi: Z \to X$ be local index 1 cover of $L$ near $x \in X$, with group $G$, and let $p: \calZ \to X$ be the index 1 cover stack.  Then, define $\calT^i_Z = \sExt^i(\Omega^1_Z, \calO_Z)$ and $T^i_Z = \Ext^i(\Omega^1_Z, \calO_Z)$.  The $\bQ$-Gorenstein smoothable deformations of $X$ are controlled by $\calT^i_{QG,X} = \pi_*(\calT^i_Z)^G$ (locally) and $T^i_{QG,X} = \Ext^i(\bL_{\calZ}, \calO_X)$, where $\bL_{\calZ}$ is the cotangent complex of the stack.  

By definition, $\calT^0_{QG,X} = T_X$, the sheaf $\calT^1_Z$ is supported on the singular locus of $Z$ and, by \cite[Corollary 3.1.13(ii)]{Sernesi} the sheaf $\calT^2_Z$ is supported on the non-lci locus of $Z$.  Furthermore, there is a local-to-global spectral sequence $H^p(\calT^q_{QG,X}) \Rightarrow T^{p+q}_{QG,X}$, so it suffices to show that $H^p(\calT^q_{QG,X}) = 0$ for $p+q =2$.  

First, consider $X = X_h$.  The divisor $L$ is $2$-Cartier and passes through the vertex of the cone.  A computation shows $H^2(\calT^0_{QG, X_h}) = H^2(T_{X_h}) = 0$, and $\calT^1_{QG, X_h}$ is supported on the singular locus of $X_h$, a single point, hence $H^1(\calT^1_{X_h}) = 0$.  Finally, the index 1 cover of $L$ on $X_h$ has only hypersurface singularities, hence $\calT^2_Z = 0$, so $H^0(\calT^2_{QG, X_h}) = 0$.  Therefore, $T^2_{QG,X_h} = 0$ and the deformations are unobstructed. 

Now, let $X = \bP(1,1,2,4)$.  The divisor $L = \calO_X(2)$ is $2$-Cartier.  From the Euler sequence and cohomology of weighted projective space, we also obtain $H^2(\calT^0_{QG, X}) = H^2(T_{X}) = 0$.  Let $o \in X$ be the $\frac{1}{4}(1,1,2)$ singularity.  Away from $o$, $L$ is Cartier, and near $o$, we may compute the index one covering of $L$: if the coordinates on $X$ are $[x:y:z:w]$, near $o$, the section $(w = 0)$ is a non-vanishing section of $L^{[2]}$.  So, one can compute the index one cover is given by the map $p: \PP(1,1,2,2) \to X$, where, if the coordinates on $Z = \PP(1,1,2,2)$ are $[x:y:u:v]$, the map is $[x:y:u:v] \mapsto [x:y:u:v^2]$.  Noting that $Z$ is a (singular) quadric threefold in $\bP^4$, we have an exact sequence 
\[ 0 \to \calO_{\bP^4}(-2)|_Z \to \Omega^1_{\bP^4}|_Z \to \Omega^1_Z \to 0. \]
Dualizing, we obtain 
\[ 0 \to T_Z \to T_{\bP^4}|_Z \to \calO_{\bP^4}(2)|_Z \to \calT^1_Z \to 0.\]
Because $\calT^1_Z$ is mapped onto by the line bundle $\calO_{\bP^4}(2)|_Z=\calO_Z(4)$ in the weighted coordinates on $Z$, and is supported on the singular locus $x = y = 0$ of $Z$, $\bP^1_{[u:v]}$, we obtain that $\calT^1_Z = \calO_{\bP^1}(2)$.  Let us assume the branch locus of $p$ is given by $w = 0$.  By definition and because the canonical covering stack is uniquely determined in the \'etale topology, away from the branch locus of $p: Z \to X$, we have $\calT^1_{QG,X} = p_*(T^1_Z)^G$, where $G$ is the action $v \to -v$ on the singular locus $\bP^1_{[u:v]}$.  As $T^1_{QG,X}$ is supported on the singular locus $x = y = 0$ of $X$, which is $\bP(2,4)_{[z:w]} \cong \bP^1_{[z^2:w]}$, we can explicitly compute $p_*(T^1_Z)^G$, where $p := p|_{\bP^1_{[u:v]}} : \bP^1_{[u:v]} \to \bP^1_{[z^2:w]}$.  By computation, since $T^1_Z = \calO_{\bP^1}(2)$, $p_*(T^1_Z) = \calO(1) \oplus \calO$.  As $p$ is given by the map $[u:v] \mapsto [u:v^2]$, we can compute the local charts and transition functions for $p_*(T^1_Z)$.  On the local chart where $u \ne 0$, computation shows that the module $\calO_{\bP^1}(2) = \mathbb{C}[v/u]$ can be viewed as the $\calO_{\bP^1_{[z^2:w]}}$-module $\mathbb{C}[z^2/w] \oplus v/u \mathbb{C}[z^2/w]$, where $z^2/w = v^2/u^2$.   Similarly, on the local chart where $v \ne 0$, computation shows that the module $\calO_{\bP^1}(2) = \mathbb{C}[u/v]$ can be viewed as the $\calO_{\bP^1_{[z^2:w]}}$-module $\mathbb{C}[z/z^2] \oplus u/v \mathbb{C}[w/z^2]$, where $w/z^2 = u^2/v^2$. Furthermore, the transition function on $\calO_{\bP^1}(2)_{u\ne 0} = \mathbb{C}[v/u]$ to $\calO_{\bP^1}(2)_{v\ne 0} = \mathbb{C}[u/v]$ is given by multiplication by $u^2/v^2$ (and $v^2/u^2$ in the other direction), gluing together to give three global sections, $u^2, uv,$ and $v^2$.  With the $\calO_{\bP^1_{[z^2:w]}}$-module structure, we see that the transition functions (from the $w \ne 0$ chart to the $z^2 \ne 0$ chart, and vice versa) become multiplication by $z^2/w = v^2/u^2$ and multiplication by $w/z^2 = u^2/v^2$, so we see that first summand of $p_*(T^1_Z)$ is $G$-invariant and the second is not.  Therefore, $p_*(T^1_Z)^G = \calO(1)$.  

However, this only computes $T^1_{QG,X}$ on the chart $w \ne 0$.  In order to compute $T^1_{QG,X}$ on the entire singular locus, we can compute the canonical covering with a different section (branched over a different point), and use the explicitly computed transition functions to glue the them together. 

Indeed, consider index-one covering using the section $w - z^2$, giving a map $p': Z \to X$ such that the branch locus is $w = z^2$.  This computes $T^1_{QG,X}$ on the chart $w \ne z^2$.  By the same computation as above, $p'_*(T^1_Z)^G = \calO(1)$, given on charts $w-z^2 \ne 0$ by $\mathbb{C}[z^2/(w-z^2)]$ and $z^2 \ne 0$ by $\mathbb{C}[(w-z^2)/z^2]$ with transition function multiplication by $(w-z^2)/z^2$ from $w-z^2 \ne 0$ to $z^2 \ne 0$ and  $z^2/(w-z^2)$ in the opposite direction.  

From the computation of $T^1_{QG,X}$ on both charts, now it is a matter of gluing the charts together.  Note that the singular locus is covered completely by the charts $w \ne 0$ (accurately computing $T^1_{QG,X}$ at all points) and $w-z^2 \ne 0$ (accurately computing $T^1_{QG,X}$ at all points).  Noting that these coincide on their common intersection, we use the previous descriptions to determine the gluing and transition functions.  We then see that $T^1_{QG,X}$ is given by $\mathbb{C}[z^2/w]$ on the $w \ne 0$ chart and $\mathbb{C}[z^2/(w-z^2)]$ on the $w - z^2 \ne 0$ chart, and the transition function is given by multiplication by $w/z^2 \cdot z^2/w-z^2 = w/w-z^2$ from $w \ne 0$ to $w-z^2 \ne 0$, and by $w-z^2/w$ in the other direction.  Then, it is easy to see that there are two global sections $w$ and $z^2$, so in particular, $T^1_{QG,X} = \calO(1)$ and $H^1(T^1_{QG,X}) = 0$.

Finally, the index 1 cover of $L$ on $\bP(1,1,2,4)$ has only hypersurface singularities, so $\calT^2_Z = 0$ and $H^0(\calT^2_{QG, X}) = 0$.  Hence, $T^2_{QG,X} = 0$ and the deformations are unobstructed. 
\end{proof}

\begin{lem}\label{lem:classify-deform-cone}
Let $\pi:\cX\to B$ be a $\bQ$-Gorenstein smoothable $\bQ$-Fano family over a smooth pointed curve $0\in B$.
\begin{enumerate}
    \item If $\cX_0\cong X_h$, then a general fiber $\cX_b$ is isomorphic to $\bP^3$ or $X_h$.
    \item If $\cX_0\cong \bP(1,1,2,4)$, then a general fiber $\cX_b$ is isomorphic to $\bP^3$, $X_h$, or $\bP(1,1,2,4)$. 
\end{enumerate}
\end{lem}

\begin{proof}
Since both $X_h$ and $\bP(1,1,2,4)$ belong to $\ocM_{3,64}^{{\rm sm},\delta\geq \epsilon_0}$, we know that $\pi$ is obtained by pulling back the universal family over $\ocM_{3,64}^{{\rm sm},\delta\geq \epsilon_0}$ under some morphism $B\to \ocM_{3,64}^{{\rm sm},\delta\geq \epsilon_0}$.
Let $\eta$ be the generic point of $B$. Since the geometric geometric fiber $[\cX_{\bar{\eta}}]\in \ocM^{{\rm sm},\delta\geq \epsilon_0}_{3,64}$, by Lemma \ref{lem:L-construct} it admits a $\bQ$-Cartier Weil divisor $\cL_{\bar{\eta}}$ such that $4\cL_{\bar{\eta}}\sim -K_{\cX_{\bar{\eta}}}$. After replacing $0\in B$ by a quasi-finite cover, we may assume that $\cL_{\bar{\eta}}$ is
the base change of a $\bQ$-Cartier Weil divisor $\cL_{\eta}$ on the generic fiber $\cX_\eta$. Then we can take the Zariski closure $\cL$ of $\cL_\eta$ as a Weil divisor on $\cX$. By similar arguments to the proof of Lemma \ref{lem:L-construct},  we have $\cL^{[-4]}\cong_{B} \omega_{\cX/B}$, and the sheaves $\pi_*\cL$ and $\pi_*\cL^{[2]}$ are locally free over $B$ of rank $4$ and $10$ respectively. 

Consider the sheaf $\cF:= \mathrm{coker} (\Sym^2 \pi_* \cL \to \pi_*\cL^{[2]})$. Denote by $\cF_b:=\cF\otimes \cO_B/\fm_b$. If $(\cX_0, \cL_0)$ is isomorphic to $(\bP(1,1,2,4), \cO_{\bP(1,1,2,4)}(2))$, we know that $\dim \cF_0= 1$. Since $b\mapsto \dim \cF_b$ is upper semi-continuous, we know that there exists an open neighborhood $0\in B'\subset B$ such that $\dim \cF_b\leq 1$ for any $b\in B'$.  In particular, because $\bP(1,1,2,4)$ can be partially smoothed to $X_h$ in a $\bQ$-Gorenstein smoothable family, we may assume that $\dim \cF_0 \le 1$ if either $(\cX_0, \cL_0)$ is isomorphic to $(\bP(1,1,2,4), \cO_{\bP(1,1,2,4)}(2))$ or isomorphic to $(X_h, \cO_{X_h}(2))$. Moreover, since $\cL_0^{[2]}$ is base point free and $\pi_*\cL^{[2]}$ is locally free, we may assume that $\cL_b^{[2]}$ is base point free (in particular Cartier) for any $b\in B'$. If $\dim \cF_b=0$, then $\cL_b$ is base point free hence Cartier, which implies that $\cX_b\cong \bP^3$ because $\cX_b$ is a canonical Gorenstein threefold with Fano index 4 \cite[Theorem 3.9]{Shin89}. If $\dim \cF_b=1$, then we may pick a basis $s_0, \cdots, s_3$ of $H^0(\cX_b, \cL_b)$ and $s_4\in H^0(\cX_b,\cL_b^{[2]})$ whose image in $\cF_b$ is non-zero. Since $\cL_b^{[2]}$ is base point free, we know that $[s_0, \cdots, s_3, s_4]$ defines a finite morphism $\cX_b\rightarrow \bP(1^4,2)$. Since this map is a closed embedding when $b=0$, it is of degree $1$ for general $b$. Thus its image is a weighted hypersurface of degree $2$ that is not isomorphic to projective space. Then a simple analysis shows that $\cX_b\cong X_h$ or $\bP(1,1,2,4)$: a degree 2 weighted hypersurface $(g = 0)$ in $\bP(1^4,2)$ with coordinates $[x_0:...:x_3:y]$ is not isomorphic to $\bP^3$ if and only if $g(0,0,0,0,y) = 0$, so has equation $g(x_0,\dots, x_3)$ consisting of monomials of degree 2 in the $x_i$s.  As $\bP(1^4,2) = C_p(\bP^3, \cO(2))$, such an equation defines the anticanonical cone over a (possibly singular) quadric in $\bP^3$, so $\cX_b$ is either $X_h$ or $\bP(1,1,2,4)$.  The proof is complete by observing that, if $\cX_0 \cong X_h$, then $\cX_b$ cannot be $\bP(1,1,2,4)$ as the singular locus cannot have larger dimension on the generic fiber.
\end{proof}

The next result shows that there is a smooth open substack of $\osM_c^{\K}$ parametrizing $\bP^3$, $X_h$, $\bP(1,1,2,4)$, or $X_u$. Later on we will see that this open substack is indeed the entire stack  $\osM_c^{\K}$ (see Proposition \ref{prop:no-other-3-fold}).

\begin{cor}\label{cor:K-moduli-smooth}
The subset of $\ocM^{{\rm sm},\delta\geq \epsilon_0}_{3,64}$ that parametrizes $\bP^3$, $X_h$, $\bP(1,1,2,4)$, and $X_u$ is a smooth open substack. In particular, there exists a smooth open substack of $\osM_c^{\K}$ parametrizing $(X, S)$  where $X$ is isomorphic to $\bP^3$, $X_h$, $\bP(1,1,2,4)$, or $X_u$.
\end{cor}

\begin{proof}
This follows from Corollary \ref{cor:X_u-global-deform}, Lemma \ref{lem:Xh-unobstructed}, and Lemma \ref{lem:classify-deform-cone}.
\end{proof}

\begin{definition}\label{def:Hilb}
 Consider the reduced locally closed substack $\sT$ of $\ocM_{3,64}^{{\rm sm},\delta\geq \epsilon_0}$ parametrizing $X_h$ or $\bP(1,1,2,4)$. Let $\bfH$ be the locally closed subscheme of $\Hilb(\bP^9)$ parametrizing $X_h$ and $\bP(1,1,2,4)$ embedded by $2L$. Then $\sT\cong [\bfH/\PGL(10,\bC)]$ (see e.g. \cite[Section 3.6]{ADL19} or \cite[Proof of Proposition 5.9]{ADL20}). Let $\pi_c: \osM_c^{\K}\to \ocM_{3,64}^{{\rm sm},\delta\geq \epsilon_0}$ be the forgetful map where $\pi_c([(X,S)])=[X]$. We define $\sH_{h,c}$ to be the locally closed substack $\sH_{h,c}:=\pi_c^{-1}(\sT)$ of $\osM_c^{\K}$. 
\end{definition}

\begin{lemma}\label{lem:hilbsmooth}
Notation as in Definition \ref{def:Hilb}. The stack $\sH_{h,c}$ is smooth.
\end{lemma}

\begin{proof}
The strategy of showing smoothness of $\sT$  is similar to \cite[Section 3.6]{ADL19} and \cite[Proof of Proposition 5.9]{ADL20}.  Let $\cX_{\bfH}\to \bfH$ be the universal family. Since the embedded $X_h$ and $\bP(1,1,2,4)$ are projective cones, there exists a section $\sigma: \bfH\to \cX_{\bfH}$ taking fiberwise cone vertices. Therefore we have a morphism $g:\bfH \to \bP^{9}$ as the composition  $\bfH\xrightarrow{\sigma}\cX_{\bfH}\hookrightarrow\bP^9\times\bfH\to  \bP^9$. It is clear that $g$ is an isotrivial fibration with fiber isomorphic to the Hilbert scheme $\bfH'$ of anti-canonically embedded $\bP^1\times\bP^1$ and $\bP(1,1,2)$ in $\bP^8$. By \cite[Proof of Proposition 5.9]{ADL20} we know that $\bfH'$ is smooth, hence $\bfH$ is also smooth. Thus we obtain the smoothness of $\sT$, thereby obtaining the smoothness of $\sH_{h,c}$.
\end{proof}

\begin{thm}\label{thm:H_h-embed}
Let $\ocK_c$ be the K-moduli stack constructed from \cite{ADL20}. Let $\tcK_c\to \ocK_c$ be the $\bm{\mu}_2$-gerbe obtained by taking fiberwise double covers. Then $\tcK_{\frac{3c-1}{4}}\cong \sH_{h,c}$ for any $c\in (\frac{1}{3},1)$. 
\end{thm}

\begin{proof}Let $f_c: \sH_{h,c} \to \tcK_{\frac{3c-1}{4}}$ be the forgetful functor $f_c([X,D]) = [D]$.  We will show that $f_c$ is separated, stabilizer preserving on $\mathbb{C}$-points, and an isomorphism on $\mathbb{C}$-points. Using this, by Lemma \ref{lem:hilbsmooth} and \cite[Theorem A.5]{AI} (a version of Zariski's main theorem for Artin stacks), we will conclude that $f_c$ is an isomorphism of stacks.  

By the valuative criterion (see e.g. \cite[Chapter 7]{LMB00}), we consider a diagram 

\[
\begin{tikzcd}
  U \arrow{d} \arrow{r} & \sH_{h,c} \arrow[d,"f_c"]\\
 T \arrow{r} \arrow[ru,dashed] & \tcK_{\frac{3c-1}{4}} 
 \end{tikzcd}
\]

\noindent where $T = \Spec R$ is a DVR and $U$ is the complement of the closed point $0 \in T$.  We must show there is at most one dashed arrow completing the diagram.  Suppose for contradiction there are two, i.e. there exist two families $(\calX, \calD) \to T$ and $(\calX', \calD') \to T$.  Because the diagram commutes, the maps agree on $\calD$, so we may assume $\calD = \calD'$, and $\calX$ and $\calX'$ are isomorphic away from the central fiber.  If they are not isomorphic in the central fiber, then by considering the graph of the rational map $\calX \dashrightarrow \calX'$, the image of $\calX_0$ in $\calX'$ is a proper subvariety of $\calX'_0$ containing $\calD_0$, as the map is an isomorphism on $\calD$.  Consider a generic ruling $R$ of the cone $\calX_0$, and let $\mathbb{A}^1$ be the ruling $R$ minus the cone point.  Because the image of $\mathbb{A}^1$ in $\calX'_0$ must be 0-dimensional and the map was an isomorphism on $\calD_0$, the image must be the two points of intersection of $\mathbb{A}^1$ and $\calD_0$.  However, this means the image of the connected variety $\mathbb{A}^1$ is disconnected, a contradiction.  Therefore, $\calX \cong \calX'$ and the map is separated.  %\footnote{\YL{There may be a small gap in the proof: the rational map $\cX\dashrightarrow \cX'$ may not preserve the $\bA^1$-fibration structure for the central fiber.}}

Next, we show the forgetful functor is stabilizer preserving on $\mathbb{C}$-points.  Suppose $\sigma \in \Aut(X,D)$ is an automorphism that is the identity on $D$.  Then, we claim that $\sigma$ is the identity.  Note that $\sigma$ must take rulings of the cone $X$ to rulings of $X$: in the universal family of the threefolds in $\bP^9$, the rulings are lines and hence curves of lowest degree, and intersection numbers are preserved in the automorphism.  Furthermore, any line on $X$ must be a ruling.  Fix any ruling $R \cong \bP^1$ of $X$.  By definition, $\sigma$ fixes the cone point and the two points of intersection with $D$, and $\sigma$ takes rulings to rulings, so $\sigma|_{\bP^1}$ fixes three points on $\bP^1$, and hence must be the identity.  Therefore, $\sigma$ is the identity. 

Now we show that $f_c$ is an isomorphism on $\mathbb{C}$-points. By \cite[Theorem 4.8]{ADL20}, we know that for any point $[(V,C)]\in \tcK_{\frac{3c-1}{4}}$, the underlying surface $V$ is isomorphic to either $\bP^1\times\bP^1$ or $\bP(1,1,2)$. Hence by taking the cone construction $\sC(V,\frac{3c-1}{4}C)=(X,cD)$ (c.f. Definition \ref{def:cone}), we know that $X$ is isomorphic to either $X_h$ or $\bP(1,1,2,4)$.  On the other hand, if $(X,D)\in \sH_{h,c}$, then Proposition \ref{prop:cone-vertex-unstable} implies that $D$ does not pass through the cone vertex of $X$. Hence after an automorphism of $X$ we have $(X,cD)\cong \sC(V,\frac{3c-1}{4}C)$ where $V\cong \bP^1\times\bP^1$ or $\bP(1,1,2)$. Since $(X,cD)$ is K-semistable, Theorem \ref{thm:induce-3-fold} implies that $(V,\frac{3c-1}{4}C)$ is also K-semistable hence belongs to $\tcK_{\frac{3c-1}{4}}$.  As a result, the forgetful map $f_c: \sH_{h,c}(\mathbb{C})\to \tcK_{\frac{3c-1}{4}}(\mathbb{C})$ sending $[(X,D)]$ to $D$ has inverse given by the cone construction $\tcK_{\frac{3c-1}{4}}(\mathbb{C})\to \sH_{h,c}(\mathbb{C})$.

By Lemma \ref{lem:hilbsmooth}, the stack $\sH_{h,c}$ is smooth, and by \cite[Theorem 2.21]{ADL20}, the stack $\tcK_{\frac{3c-1}{4}}$ is smooth. Therefore, we conclude that $f_c$ is actually an isomorphism by \cite[Theorem A.5]{AI}, noting that being  isomorphic and stabilizer preserving on $\mathbb{C}$-points implies fully faithful and essentially surjective.\end{proof}

\begin{defn}\label{defn:H_h,c}
Let $c\in (\frac{1}{3},1)$ be a rational number. By Theorem \ref{thm:H_h-embed}, there exists a closed immersion $\tcK_{\frac{3c-1}{4}} \hookrightarrow \osM_{c}^{\K}$ of Artin stacks. By taking good moduli spaces, we obtain a closed immersion $\oK_{\frac{3c-1}{4}}\hookrightarrow \ofM_c^{\K}$. Let $H_{h,c}$ (resp. $\sH_{h,c}$) be the image of the closed embedding in the K-moduli space (resp. K-moduli stack). Equivalently, by Theorem \ref{thm:H_h-embed} $H_{h,c}$ (resp. $\sH_{h,c}$) is the locus parametrizing K-polystable (resp. K-semistable) pairs $(X,cD)$ where $X\cong X_h$ or $\bP(1,1,2,4)$.
\end{defn}

\subsection{K-polystable replacements}

In this section, we describe all K-polystable replacements of the locus $W_8$ inside $\ofM^{\GIT}$. In particular, we show that the replacements of $(\bP^3, S)$ where $[S]\in W_8$ is  either $(X_h, S')$ or $(\bP(1,1,2,4), S')$. Here $S'$ is a double cover of $\bP^1\times\bP^1$ or $\bP(1,1,2)$.

By Theorem \ref{thm:shah}, every $[S]\in W_8$ is defined by $S=(q^2+g=0)$ where $q$ is a degree $2$ polynomial and $g$ is a degree $4$ polynomial. 
There is a natural closed embedding of such a log pair $(\bP^3, S)$ with into the weighted projective space $\bP(1^4, 2)$ with coordinates $[x_0,x_1,x_2,x_3, z]$ such that the image is given by $(V(z-q), V(z-q, z^2+g))$. 

We start with the K-polystable replacement of $[2Q]$.

\begin{prop}\label{prop:2Q-replace}
Let $S=2Q$ where $Q\subset\bP^3$ is a smooth quadric surface. Let $c\in (0,1)$ be a rational number. Then $(\bP^3,cS)$ is K-semistable (resp. K-polystable) if and only if $c\leq \frac{1}{3}$ (resp. $<\frac{1}{3}$).  Moreover, the K-polystable degeneration of $(\bP^3,\frac{1}{3}S)$ is isomorphic to $(X_h, \frac{1}{3}S_0)$ where $S_0=2 Q_\infty$ and $Q_\infty$ is the section of $X_h=C_p(Q, -K_Q)$ at infinity.
\end{prop}

\begin{proof}
We first show that if $(\bP^3,cS)$ is K-semistable, then $c\leq \frac{1}{3}$. Computation shows that
\[
A_{(\bP^3, cS)}(\ord_Q)=1-2c, \quad S_{(\bP^3, cS)}(\ord_Q)=(4-4c)\int_{0}^{1/2} (1-2t)^3 dt=\frac{1-c}{2}.
\]
By Theorem \ref{thm:valuative}, we have 
$A_{(\bP^3, cS)}(\ord_Q)\geq S_{(\bP^3, cS)}(\ord_Q)$ which implies that $c\leq \frac{1}{3}$.

Next, we show that $(\bP^3, \frac{1}{3}S)$ special degenerates to $(X_h, \frac{1}{3}S_0)$. We may embed $(\bP^3, S)$ into $\bP(1^4, 2)$ with image $(V(z-q), V(z-q, z^2))$ where $q=x_0x_1+x_2^2+x_3^2$. Consider a $1$-PS $\tilde{\sigma}$ in $\SL(4,\bC)\times\bG_m$ of weight $(0,0,0,0,-1)$ acting diagonally on $\bP(1^4,2)$. Then $\tilde{\sigma}$ specially degenerates $(\bP^3, \frac{1}{3} S)$ to $(V(q), \frac{1}{3}V(q,z^2))$ which is isomorphic to $(X_h, \frac{1}{3}S_0)$. Since $Q\cong \bP^1\times\bP^1$ is K-polystable and $X_h\cong C_p(Q, -K_Q)$, by \cite[Proposition 2.11(2)]{LZ19} we know that $(X_h, \frac{1}{3}S_0)$ is K-polystable. This implies that $(\bP^3, \frac{1}{3}S)$ is K-semistable by openness of K-semistability \cite{BLX19, Xu19}. Since $S=2Q$ is GIT polystable, Theorem \ref{thm:K=GIT-smallc} implies that $(\bP^3, \epsilon S)$ is K-polystable for $0<\epsilon \ll 1$. Hence interpolation for K-stability \cite[Proposition 2.13]{ADL19} implies that $(\bP^3, cS)$ is K-polystable for any $c\in (0, \frac{1}{3})$. The proof is finished.
\end{proof}

\begin{thm}\label{thm:Kpsreplace-W}
Let $[S]\in W_8\setminus \{[2Q]\}$ be a GIT polystable point with $S=(q^2+g=0)$ as in Theorem \ref{thm:shah}.  Let $V=(q=0)$ and $C=V(q, g)$ be a quadric surface and a $(2,4)$-complete intersection curve in $\bP^3$, respectively. Denote by $c:=\kst(\bP^3, S)$. Then we have the following. 
\begin{enumerate}
    \item The log Fano pair $(V, \frac{3c-1}{4}C)$ is K-semistable but not K-polystable. 
    \item There exists a $1$-PS $\sigma$ in $\SL(4,\bC)$ that induces a K-polystable degeneration $(V_0, \frac{3c-1}{4}C_0)$ of $(V, \frac{3c-1}{4}C)$ in $\bP^3$.
    \item There exists a $1$-PS $\tilde{\sigma}$ in $\SL(4,\bC)\times \bG_m$ acting diagonally on $\bP(1^4,2)$ such that the $\SL(4,\bC)$-component of $\tilde{\sigma}$ is a positive rescaling of $\sigma$, and $\tilde{\sigma}$ induces a K-polystable degeneration $\sC(V_0, \frac{3c-1}{4}C_0)$ of $(\bP^3, cS)$ in $\bP(1^4,2)$. 
\end{enumerate}
Moreover, the K-semistable thresholds and the destabilizing $1$-PS' in $\SL(4,\bC)\times\bG_m$ are given in Table \ref{table:singularities}, where $0<\alpha\ll 1$ and $\kst(W_i^\circ):=\kst(\bP^3,S)$ for any $[S]\in W_i^\circ$.
\end{thm}

\begin{table}[htbp!]\renewcommand{\arraystretch}{1.5}
\caption{K-polystable replacements of $W_i^{\circ}$}\label{table:singularities}
\begin{tabular}{|c|c|l|l|}
\hline 
$i$ & $\kst(W_i^\circ)$   & Sing. of $S$ in $W_i^{\circ}$ & Destabilizing $1$-PS\\ \hline \hline 
0 & $\frac{1}{3}$ & double quadric & $(0,0,0,0,-1)$\\
1 & $\frac{1}{2}$  & two quadrics tangent along a conic &   $(1,\alpha,2\alpha-1,-3\alpha,-6\alpha)$  \\ 
2 & $\frac{3}{5}$  & cuspidal along a conic  &  $(1,\alpha,2\alpha-1,-3\alpha,-4\alpha)$\\ 
3 & $\frac{2}{3}$  & $J_{4,\infty}$     &  $(7,3,-1,-9,-10)$\\ 
4 & $\frac{5}{7}$  & $J_{3,0}$, $J_{3,r}$, or $J_{3,\infty}$       &  $(3,1,-1,-3,-3)$ \\ 
6 & $\frac{3}{4}$ & $E_{14}$  &     $(17,5,-7,-15,-14)$\\ 
7 & $\frac{7}{9}$  & $E_{13}$      & $(11,3,-5,-9,-8) $  \\ 
8 & $\frac{9}{11}$ & $E_{12}$      &  $(8,2,-4,-6,-5)$   \\ \hline
\end{tabular}
\end{table}

\begin{proof}

%\textcolor{blue}{YL: will add more details. Follow \cite[6.7, 6.8, Table 2]{LO18a}.}

Let $S=(q^2+2g=0)$ be a GIT polystable quartic surface in $W_i^\circ$ for $i\in \{1,2,3,4,6,7,8\}$. As mentioned earlier, the pair $(\bP^3, S)$ admits an embedding into $\bP(1^4,2)$ given by the equations $(V(z-q), V(z-q, z^2+g))$. The pair $(V,C)=(V(q),V(q,g))\subset\bP^3$ provides a threefold pair $(X,S'):=\sC(V,C)$. It is clear that $(X,S')$ admits a embedding into $\bP(1^4,2)$ given by the equations $(V(q), V(q, z^2+g))$. According to \cite[Theorem 6.2, Proposition 6.6, and Table 2]{LO18a}, for each $i\in \{1,2,3,4,6,7,8\}$ and $[S]\in W_i^\circ$, there exists a 1-PS $\sigma$ of $\SL(4,\bC)$ acting diagonally on $\bP^3$ and a rational number $t_{i-1}$ depending on $i$ such that the following properties hold in the context of VGIT from \cite[Section 6]{LO18a}. 
\begin{itemize}
    \item $(V,C)$ is GIT polystable at slope $(t_{i-1}-\epsilon)$ for $0< \epsilon\ll 1$, and is GIT semistable but not polystable at slope $t_{i-1}$;
    \item $\sigma$ degenerates $(V,C)$ to a GIT polystable pair $(V_0, C_0)$ at slope $t_{i-1}$.
\end{itemize}
Here the correspondence between $i$ and $t_{i-1}$ are given in the following table. Indeed, the $\sigma$ from \cite[Table 2]{LO18a} matches the $\SL(4,\bC)$-component of $\tsigma$ from Table \ref{table:singularities} up to positive rescaling.

\begin{table}[htbp!]\renewcommand{\arraystretch}{1.5}
\caption{VGIT slopes and K-semistable thresholds}\label{table:VGIT-slope}
\begin{tabular}{|c||c|c|c|c|c|c|c|}
\hline 
$i$ & $1$ & $2$ & $3$ & $4$ & $6$ & $7$ & $8$\\ \hline 
%$k$ & $0$ & $1$ & $2$ & $3,4$ & $5$ & $6$ & $7$ \\\hline
VGIT slope: $t_{i-1}$ & $\frac{1}{6}$ & $\frac{1}{4}$ & $\frac{3}{10}$ & $\frac{1}{3}$ & $\frac{5}{14}$ & $\frac{3}{8}$ & $\frac{2}{5}$ \\  \hline
$\kst(W_i^\circ)$: $\frac{1+2t_{i-1}}{3-2t_{i-1}}$ & $\frac{1}{2}$ & $\frac{3}{5}$ & $\frac{2}{3}$ & $\frac{5}{7}$ & $\frac{3}{4}$ & $\frac{7}{9}$ & $\frac{9}{11}$ \\\hline
\end{tabular}
\end{table}

By \cite[Theorem 1.1(2)]{ADL20}, we know that $(V,\frac{2t_{i-1}}{3-2t_{i-1}}C)$ is K-semistable but not K-polystable, whose K-polystable degeneration is $(V_0, \frac{2t_{i-1}}{3-2t_{i-1}}C_0)$. Let $c:=\frac{1+2t_{i-1}}{3-2t_{i-1}}$ which depends on the choice of $i$. Hence Theorem \ref{thm:induce-3-fold} implies that $(X_0, cS_0):=\sC(V_0, \frac{2t_{i-1}}{3-2t_{i-1}}C_0)$ is K-polystable. Suppose $(V_0,C_0)=(V(q_0), V(q_0,g_0))\subset \bP^3$. Then clearly $(X_0, S_0)= (V(q_0), V(q_0, z^2+g_0))\subset \bP(1^4,2)$. We choose the $\bG_m$-component of $\tsigma$ so that the weight of $z^2$ is the same as the weight of $g_0$, while the weight of $z$ is smaller than the weight of $q_0$. It is straightforward to check that the column of $\tsigma$ from Table \ref{table:singularities} satisfies this property. Thus we know that $\tsigma$ degenerates both $(\bP^3, cS)$ and $(X,cS')$ to the K-polystable pair $(X_0, cS_0)$. By openness of K-semistability \cite{BLX19, Xu19}, this implies that $(\bP^3, cS)$ is K-semistable but not K-polystable, as $X_0$ is either $X_h$ or $\bP(1,1,2,4)$ so $\bP^3\not\cong X_0$. Thus $c=\kst(\bP^3, S)=\frac{1+2t_{i-1}}{3-2t_{i-1}}$ is listed in the last row of Table \ref{table:VGIT-slope} for every quartic surface $[S]\in W_i^\circ$. The proof is finished.
\end{proof}

\subsection{K-moduli wall crossings}\label{sec:wallcrossings}

Recall from Theorem \ref{thm:ADL19-wall} that K-moduli spaces $\ofM_c^{\K}$ display wall crossing phenomena as we vary $c$. The goal of this subsection is to completely describe all K-moduli walls $\{c_i\}_{1\leq i\leq k-1}$ and all wall crossing birational morphisms $\phi_i^{\pm}:\ofM_{c_i\pm \epsilon}^{\K}\to \ofM_{c_i}^{\K}$. We first give a definition of the exceptional loci of $\phi_{i}^{\pm}$.

\begin{defn}
For each K-moduli wall $c_i\in (0,1)$, the \emph{exceptional locus} $E_i^{\pm}$ is defined as
\[
E_i^{\pm}:=\{[(X,S)]\in \ofM_{c_i\pm \epsilon}^{\K}\mid (X, S) \textrm{ is $(c_i\pm \epsilon)$-K-polystable but not $c_i$-K-polystable}\}.
\]
By \cite[Theorem 1.2]{ADL19}, we know that $E_i^{\pm}$ is a Zariski closed subset of $\ofM_{c_i\pm \epsilon}^{\K}$.
\end{defn}

We start from describing the first wall crossing.

\begin{thm}[First wall crossing]\label{thm:firstwall}
The first wall of K-moduli stacks $\osM_c^{\K}$ is $c_1=\frac{1}{3}$.  The first wall crossing has the following description.
\begin{enumerate}
    \item For any $0<c<\frac{1}{3}$, we have $\osM_c^{\K}\cong \osM^{\GIT}$ and $\ofM_c^{\K}\cong \ofM^{\GIT}$.
    \item The wall crossing morphism $\phi_1^-:\ofM_{\frac{1}{3}-\epsilon}^{\K}\xrightarrow{\cong} \ofM_{\frac{1}{3}}^{\K}$ is an isomorphism which only replaces $(\bP^3, 2Q)$ by $(X_h, 2Q_\infty)$. 
    \item The wall crossing morphism $\rho: \ofM_{\frac{1}{3}+\epsilon}^{\K}\to \ofM_{\frac{1}{3}-\epsilon}^{\K}=\ofM^{\GIT}$ is a divisorial contraction with image $[2Q]$ such that $\phi_1^+=\phi_1^- \circ \rho: \ofM_{\frac{1}{3}+\epsilon}^{\K}\to \ofM_{\frac{1}{3}}^{\K}$. The exceptional locus $E_1^+$ of $\rho$ parametrizes $[(X_h, S)]$ where $S$ is a double cover of $\bP^1\times \bP^1$ branched along a GIT polystable $(4,4)$-curve $D$. 
    %at the point $[2Q]$ such that $\phi_1^+=\phi_1^- \circ \rho: \ofM_{\frac{1}{3}+\epsilon}^{\K}\to \ofM_{\frac{1}{3}}^{\K}$. Moreover, the exceptional divisor $E_1^+$ of $\rho$ parametrizes $[(X_h, S)]$ where $S$ is a double cover of $\bP^1\times \bP^1$ branched along a GIT polystable $(4,4)$-curve $D$.
\end{enumerate}
\end{thm}

%\begin{remark}
%By Remark \ref{rmk:weightedblowup} (a combination of \cite[Section 5.2]{LO18b} and the results of Section \ref{sec:proofs}, where we identify our K-moduli spaces with the spaces defined in Laza-O'Grady), we will see that the morphism $\rho$ in Theorem \ref{thm:firstwall}(3) is a weighted blowup of Kirwan type at the point $[2Q]$.\footnote{check}
%\end{remark}

\begin{proof}
We first prove part (1). It follows from Theorem \ref{thm:K=GIT-smallc} that we have $\osM_\epsilon^{\K}\cong \osM^{\GIT}$ for $0<\epsilon\ll 1$. By Proposition \ref{prop:k-moduli-irred}, the K-moduli stack $\osM_{c}^{\K}$ is irreducible. Hence it suffices to show that any GIT semistable quartic surface $S\subset\bP^3$ satisfies $\kst(\bP^3, S)\geq \frac{1}{3}$. Let $S_0$ be the GIT polystable degeneration of $S$. Then by openness of K-semistability \cite{BLX19, Xu19}, we know that $\kst(\bP^3, S_0)\leq \kst(\bP^3, S)$. If $S_0$ is slc, then Lemma \ref{lem:ADE-K} implies that $\kst(\bP^3, S_0)=1$. If $S_0$ is not slc, then either $[S_0]=[T]$ or $[S_0]\in W_8$ in $\ofM^{\GIT}$. Thus we have $\kst(\bP^3, S_0)\geq \frac{1}{3}$ by Theorems \ref{thm:tangent-kst} and  \ref{thm:Kpsreplace-W}. This finishes the proof of part (1).

For part (2), Theorem \ref{thm:ADL19-wall} implies that  $\phi_1^{\pm}$ are projective and birational, hence surjective. By Theorems \ref{thm:tangent-kst} and \ref{thm:Kpsreplace-W}, we know that the only GIT polystable quartic surface $S$ with $\kst(\bP^3, S)=\frac{1}{3}$ is the double quadric surface $[S]=[2Q]$. Hence $\phi_1^{-}$ only replaces $(\bP^3, 2Q)$ by $(X_h, 2Q_\infty)$ by Proposition \ref{prop:2Q-replace}. Since $\ofM_c^{\K}$ is irreducible, by \cite[Theorem 1.2]{ADL19} the morphism $\phi_1^-$ is  projective and bijective. Since any closed point $[(X,S)]$ of $\osM_\frac{1}{3}^{\K}$ satisfies that $X\cong \bP^3$ or $X_h$, Corollary \ref{cor:K-moduli-smooth} implies that $\osM_{\frac{1}{3}}^{\K}$ is smooth. Hence $\ofM_{\frac{1}{3}}^{\K}$ is normal, and $\phi_1^-$ is an isomorphism by Zariski's main theorem.

For part (3), we use the deformation theory of $X_h$.
%local VGIT presentation\footnote{YL: Here more details are needed.} of $\phi_1^\pm$ to obtain that $\phi_1^+$ is a weighted blow-up at $[(X_h, 2Q_\infty)]$ with a prime exceptional divisor $E_1^+$. 
By Theorem \ref{thm:H_h-embed} and Definition \ref{defn:H_h,c}, there is a closed embedding $\oK_{\epsilon/4} \hookrightarrow \ofM_{\frac{1}{3}+\epsilon}^{\K}$ whose image we denote by $H_{h,\frac{1}{3}+\epsilon}$. Note that  $H_{h,\frac{1}{3}+\epsilon}$ is a divisor since the locus of hyperelliptic K3s forms a divisor (see e.g. \cite{LO18a}).   By \cite[Theorem 1.1(1)]{ADL20}, we know that  $H_{h,\frac{1}{3}+\epsilon}$ parametrizes $(X_h, S)$ where $S$ is a double cover of $\bP^1\times \bP^1$ branched along a GIT polystable $(4,4)$-curve $D$. Since $(X_h,S)$ admits a special degeneration to $(X_h, 2Q_\infty)$, we know that $\phi_1^+(H_{h,\frac{1}{3}+\epsilon})=[(X_h, 2Q_\infty)]$. Thus $H_{h,\frac{1}{3}+\epsilon}$ is contained in $E_1^+=(\phi_1^+)^{-1}([(X_h, 2Q_\infty)])$.
On the other hand, if $(X,S)\in E_1^+$ i.e. $\phi_1^+([(X,S)])=[(X_h, 2Q_\infty)]$, then $X$ admits a special degeneration to $X_h$. Thus Lemma \ref{lem:classify-deform-cone} implies that $X\cong \bP^3$ or $X_h$. If $X\cong \bP^3$, then $(X,\frac{1}{3} S)$ is K-polystable by interpolation, a contradiction. Thus we have $X\cong X_h$ and hence $[(X,S)]\in H_{h,\frac{1}{3}+\epsilon}$ by Theorem \ref{thm:H_h-embed}. Thus we have $H_{h,\frac{1}{3}+\epsilon}=E_1^+$.
%\textcolor{blue}{YL: To prove our main result, it is possible that we don't need the full stack version here, although it looks better.}
\end{proof}

The following result shows that all but one K-moduli wall crossings of $\ofM_c^{\K}$ for $c\in (\frac{1}{3},1)$ are directly induced by those of $H_{h,c}$, i.e. the divisor  parametrizing hyperelliptic K3 surfaces. 

\begin{prop}\label{prop:no-other-3-fold}
    Let $c\in (\frac{1}{3}, 1)$ be a rational number. Then  any point in $\ofM_{c}^{\K}\setminus H_{h,c}$ is either $[(\bP^3, S)]$ where $S$ is a GIT polystable quartic surface, or $[(X_u, S')]$ where $S'$ is an anti-canonical divisor on $X_u$. %In particular, $\osM_c^{\K}$ is smooth, and $\ofM_c^{\K}$ is normal with $H_{h,c}$ a prime divisor.
\end{prop}

\begin{proof}
We do induction on the K-moduli walls. When $c=\frac{1}{3}+\epsilon$, by Theorem \ref{thm:firstwall} we know that $\ofM_{c}^{\K}= (\ofM^{\GIT}\setminus \{[2Q]\})\cup H_{h,c}$. Assume that we hit a K-moduli wall $c_i$ such that the statement is true for any $\frac{1}{3}<c<c_i$. Then we analyze the wall crossing morphisms 
\[
\ofM_{c_i-\epsilon}^{\K}\xrightarrow{\phi_i^-}\ofM_{c_i}^{\K}\xleftarrow[]{\phi_i^+}\ofM_{c_i+\epsilon}^{\K}.
\]
We first show that the statement is true for $c=c_i$. Indeed, all $c_i$-K-polystable replacements of $H_{h,c_i-\epsilon}$ belongs to $H_{h, c_i}$ by Theorem \ref{thm:H_h-embed}. For $[(\bP^3, S)]\in \ofM_{c_i-\epsilon}^{\K}\setminus H_{h,c_i-\epsilon}$ and $[S]\in W_8$, its K-polystable replacement belongs to $H_{h,c_i}$ by Theorem \ref{thm:Kpsreplace-W}. For $(\bP^3, T)$, its K-polystable replacement at $c_i=\frac{9}{13}$ is $(X_u, T_0)$. Since any Fano threefold appearing in $H_{h,c_i}$ is either $X_h$ or $\bP(1,1,2,4)$, Lemma \ref{lem:Xh-unobstructed} implies that $\osM_{c_i}^\K$ is a smooth stack. Thus the statement holds for $c=c_i$.

Next, we show that the statement is true for $c=c_i+\epsilon$. 
Since $\osM_{c_i+\epsilon}^{\K}$ is an open substack of $\osM_{c_i}^{\K}$, it is also a smooth stack. 
Assume that $[(X,S)]\in \ofM_{c_i+\epsilon}^{\K}\setminus H_{h,c_i+\epsilon}$. If $(X,S)$ is $c_i$-K-polystable, then it is also $(c_i-\epsilon)$-K-polystable by \cite[Proposition 3.18]{ADL19}. Thus $X\cong \bP^3$ or $X_u$. If $(X,S)$ is not $c_i$-K-polystable, then let $(X_0, S_0)$ be its $c_i$-K-polystable replacement. Then either $(X_0,S_0)\in H_{h,c_i}$ or $(X_0,S_0)\cong (X_u, T_0)$. In the first case, we know that $X_0\cong X_h$ or $\bP(1,1,2,4)$ which implies that $X\cong \bP^3$, $X_h$, or $\bP(1,1,2,4)$ by Lemma \ref{lem:classify-deform-cone}. In fact, $X$ cannot be isomorphic to $\bP^3$ since otherwise $(X, c_i S)$ is K-polystable by interpolation, a contradiction. Thus $X\cong X_h$ or $\bP(1,1,2,4)$ which implies that $[(X,S)]\in H_{h,c_i+\epsilon}$ by Theorem \ref{thm:H_h-embed}, again a contradiction. In the second case, by Corollary \ref{cor:X_u-global-deform} we have that $X\cong \bP^3$ or $X_u$. Hence the proof is finished.
%Indeed, this follows from the fact that fibers of VGIT wall crosssings have complementary dimensions (see \cite{thaddeus, dolgachevhu}). Hence by the local VGIT presentation from \cite[Theorem 1.2]{ADL19}, the exceptional locus of $\phi_i^+|_{H_{h,c_i+\epsilon}}$ has the same dimension as the exceptional locus of $\phi_i^+$. Then we know that they must have the same exceptional loci by smoothness of $\osM_{c_i}^{\K}$ (here we use irreducibility of exceptional loci for VGIT on a smooth variety). In other words, the K-moduli wall crossings always flip the $W_*$ into $H_{h, *}$.   Thus the proof is finished.
\end{proof}

The following theorem summarizes the results we have obtained, which provides a detailed description of wall crossings for K-moduli spaces $\ofM_{c}^{\K}$. 

\begin{thm}\label{thm:k-moduli-walls}
The K-moduli space $\ofM_c^{\K}$ (resp. K-moduli stack $\osM_c^{\K}$) is irreducible and normal (resp. smooth) for any $c\in (0,1)$. Moreover, the list of K-moduli walls of $\ofM_c^{\K}$ is given as follows.
\begin{equation}\label{eq:k-moduli-walls}
(c_1,c_2,\cdots, c_9)= \left( \frac{1}{3}, \frac{1}{2}, \frac{3}{5}, \frac{2}{3}, \frac{9}{13}, \frac{5}{7}, \frac{3}{4}, \frac{7}{9}, \frac{9}{11}  \right).
\end{equation}
In the below, we give precise description of the wall crossing morphisms.
\begin{enumerate}
    \item When $c=c_1=\frac{1}{3}$, the K-moduli wall crossing map decreasing from $c = \frac{1}{3}+\epsilon$ to $c = \frac{1}{3} - \epsilon$ is a divisorial contraction of the exceptional divisor $E_1^+$, which is the birational transform of $H_h$, to the point $[(\bP^3, 2Q)]$.
    %When $c=c_1=\frac{1}{3}$, the K-moduli wall crossing is a weighted blow-up at $[2Q]$ with exceptional divisor $E_1^+$ being the birational transform of $H_h$.
    \item When $c=c_5=\frac{9}{13}$, the K-moduli wall crossing map decreasing from $c = \frac{9}{13} + \epsilon$ to $\frac{9}{13}-\epsilon$ is a divisorial contraction of the exceptional divisor $E_5^+$, which is the birational transform of $H_u$, to the point $[(\bP^3, T)]$.
    %When $c=c_5=\frac{9}{13}$, the K-moduli wall crossing is a weighted blow-up at $[T]$ with exceptional divisor $E_5^+$ being the birational transform of $H_u$.
    \item When $c=c_i\in \{ \frac{1}{2}, \frac{3}{5}, \frac{2}{3}, \frac{5}{7}, \frac{3}{4}, \frac{7}{9}, \frac{9}{11} \}$, i.e. $2\leq i\leq 9$ and $i\neq 5$, K-moduli wall crossings are flips. Moreover, if $i\in \{2,3,4, 7,8,9\}$ (resp. if $i=6$) then the exceptional locus $E_i^{-}$ is the birational transform of $W_{i-1}$ (resp.  $W_{i-2}$), while $E_i^+$ is the birational transform of  $Z^{i}$ (resp.  $Z^{i-1}$). %for $i\in \{1,2,3,4,6,7,8\}$
   % \footnote{YL: here is Laza-O'Grady terminology} 
\end{enumerate}
\end{thm}

\begin{remark}\label{rmk:Kirwan1stpaper}
Using  techniques similar to \cite[Section 5.2]{ADL19}, one should be able to show that the wall crossing morphisms in Theorem \ref{thm:k-moduli-walls} (1) and (2) are weighted blow-ups of Kirwan type. Since this is not necessary for our main results, we omit the calculation.  By Remark \ref{rmk:weightedblowup} (a combination of \cite[Sections 5.1 and 5.2]{LO18b} and the results of Section \ref{sec:proofs}, where we identify our K-moduli spaces with the spaces defined in Laza-O'Grady), we will see that the wall crossing morphisms in Theorem \ref{thm:k-moduli-walls} (1) and (2) are indeed weighted blowups of Kirwan type at the point $[2Q]$ and $[T]$ respectively.
\end{remark}

\begin{proof}
By Theorem \ref{thm:firstwall} and Proposition \ref{prop:no-other-3-fold}, we know that the only possible Fano threefolds appearing in $\ofM_c^{\K}$ are $\bP^3$, $X_h$, $\bP(1,1,2,4)$, or $X_u$ for any $c\in (0,1)$. Thus the smoothness of $\osM_c^{\K}$ and normality of $\ofM_c^{\K}$ follow from Corollary \ref{cor:K-moduli-smooth}, while irreducibility is proven in Proposition \ref{prop:k-moduli-irred}.

Next, we turn to the list of K-moduli walls. By Proposition \ref{prop:no-other-3-fold}, a K-moduli wall $c_i$ of $\ofM_c^{\K}$ either satisifies $c_i=\kst(\bP^3, S)$ for some $[S]\in W_8\cup \{[T]\}$, or it is a wall of the K-moduli spaces $\oK_{\frac{3c-1}{4}}\cong H_{h,c}$ from \cite{ADL20}, as there are no wall crossings on $H_{u,c}$ when $c\in (\frac{9}{13},1)$ by  Theorem \ref{thm:T_0-replace}. In the former case, we precisely obtain the right-hand-side of \eqref{eq:k-moduli-walls}. In the latter case, by \cite[Remark 5.13]{ADL20} the collection of walls is $\{c\mid \frac{3c-1}{4}\in \{\frac{1}{8}, \frac{1}{5}, \frac{1}{4}, \frac{2}{7}, \frac{5}{16}, \frac{1}{3}, \frac{4}{11}\}\}$, which equals $\{\frac{1}{2}, \frac{3}{5}, \frac{2}{3},  \frac{5}{7}, \frac{3}{4}, \frac{7}{9}, \frac{9}{11}\}$ as a subset of the right-hand-side of \eqref{eq:k-moduli-walls}. Thus we have verified the list of all K-moduli walls \eqref{eq:k-moduli-walls}.

Next, we characterize the K-moduli wall crossing morphisms.
Part (1) follows from Theorem \ref{thm:firstwall}. Part (2) follows from Theorems \ref{thm:GIT=K-unigonal} and \ref{thm:T_0-replace}. We focus on part (3). Let $j:= i-1$ (resp. $j:=i-2$) when $i\in \{2,3,4,7,8,9\}$ (resp. when $i=6$). By Theorem \ref{thm:Kpsreplace-W} and Proposition \ref{prop:no-other-3-fold}, we know that $E_i^-=W_j^\circ \sqcup (E_i^-\cap H_{h, c_i-\epsilon})$ as sets, and $E_i^+\subset H_{h,c_i+\epsilon}$. Since $H_{h,c}\cong \oK_{\frac{3c-1}{4}}$ by Theorem \ref{thm:H_h-embed}, we know that $E_i^\pm\cap H_{h, c_i\pm\epsilon}$ is isomorphic (via the operation $\sC(\cdot,\cdot)$) to the exceptional locus $E_i'^\pm$ of $\oK_{c_i'\pm\epsilon}\to \oK_{c_i'}$ where $c_i':=\frac{3c_i-1}{4}$. By \cite[Theorem 1.1]{ADL20}, the locus $E_i'^\pm$ is the same as the VGIT exceptional locus of slope $t_i:=\frac{3c_i'}{2c_i'+2}$ for $(2,4)$ complete intersections in $\bP^3$. By \cite[Theorem 1.1]{LO18a}, we know that $E_i'^-$ is the strict transform of $\rho_*^{-1}W_j\cap H_{h, \frac{1}{3}+\epsilon}$, while $E_i'^+$ is the strict transform of $Z^{j+1}$. Thus we see that $E_i^-$ (resp. $E_i^+$) is the strict transform of $W_j$ (resp. $Z^{j+1}$). The fact that these morphims are flips will follow from the calculations in Theorem \ref{thm:CM-proportional}, as the morphisms are shown to be MMP with scaling with respect to the CM line bundle.
\end{proof}

\section{Proof of main theorems}\label{sec:proofs}

%\textcolor{blue}{This section's previous title is ``Prediction based on Laza-O'Grady''. Here we will try to prove the main theorems, including CM line bundle computations and use Xu-Zhuang to obtain a Hassett-Keel type result.}

In this section we present proofs of main theorems. 

\subsection{CM line bundles on K-moduli spaces}

In this subsection, we compute the log CM $\bQ$-line bundles on the K-moduli spaces $\ofM_c^{\K}$, and prove Theorem \ref{mthm:Kmod}.

%Here we give a computation of the variation of CM line bundles on $\ofM_c^{\K}$ based on Laza-O'Grady. Let $\sF$ be the arithmetic quotient of the period domain of polarized quartic K3 surfaces with Du Val singularities. Global Torelli tells us that $\sF$ is isomorphic to the coarse moduli space of such K3 surfaces. Denote by $H_h$ and $H_u$ the Heegner divisors in $\sF$ of types hyperelliptic and unigonal, respectively.

The following result describes the locus of K3 surfaces with Du Val singularities inside K-moduli stacks and spaces.

\begin{prop}\label{prop:loci-ADE-K3}
Let $c\in (0,1)$ be a rational number.
There exists a saturated open substack $\sM_{c}^{\K, \circ}$ of $\osM_{c}^{\K}$ consisting of $c$-K-stable log pairs $[(X,S)]$ where $S$ has Du Val singularities. Moreover,  $\sM_{c}^{\K, \circ}$ is a smooth Deligne-Mumford stack with coarse moduli space $\fM_{c}^{\K, \circ}$. The birational period map $\fp_c:\ofM_c^{\K}\dashrightarrow \sF$ for boundary divisors induces an open immersion $\fp_c^\circ:\fM_{c}^{\K, \circ}\hookrightarrow \sF$ satisfying the following properties.
\begin{enumerate}
    \item  $\fM_{c}^{\K, \circ}$ is a big open subset of $\ofM_{c}^{\K}$ for any $c\in (0,1)$.
    \item The divisorial components of $\sF\setminus \fp_c^\circ(\fM_{c}^{\K, \circ})$ are
    $\begin{cases}
         \textrm{none} & \textrm{if }c\in (\frac{9}{13},1);\\
        H_u & \textrm{if }c\in (\frac{1}{3},\frac{9}{13}];\\
        H_h\cup H_u & \textrm{if }c\in (0,\frac{1}{3}]. 
    \end{cases}$
    \item $\fp_{c}^\circ$ is an isomorphism if and only if $c\in (\frac{9}{11},1)$. 
\end{enumerate}
\end{prop}

\begin{proof}
We first show that $\sM_{c}^{\K, \circ}$ is a saturated open substack of $\osM_c^{\K}$. By the openness of K-stability \cite{BLX19} and lower semi-continuity of $\lct$, there exists an open substack $\sM_{c}^{\rm lc}$ of $\osM_c^{\K}$ parametrizing $c$-K-stable log pairs $(X,S)$ that are log canonical. By applying adjunction, we obtain a $\bQ$-Gorenstein universal family $\sS\to \sM_{c}^{\rm lc}$ with fibers being semi-log-canonical surfaces with trivial canonical divisor. Since klt is an open condition in $\bQ$-Gorenstein families, and Gorenstein klt is the same as having ADE singularities for surfaces, we know that $\sM_c^{\K,\circ}$ is an open substack of $\sM_{c}^{\rm lc}$ and hence an open substack of $\osM_c^{\K}$. Since $\sM_c^{\K,\circ}$ consists of K-stable log pairs, every point is closed with finite stabilizers. Hence we know that $\sM_c^{\K,\circ}$ is a saturated open Deligne-Mumford substack of $\osM_c^{\K}$. By taking period map for boundary divisors, we obtain a morphism $\sM_c^{\K, \circ}\to \sF$ which descends to a morphism $\fp_c^{\circ}: \fM_c^{\K, \circ}\to \sF$. 

Since $\fp_c$ is birational, so is $\fp_c^{\circ}$. Next we show that $\fp_c^{\circ}$ is injective on closed points. By global Torelli theorem it suffices to show that if $(X,S)$ and $(X',S')$ both belong to $\fM_c^{\K,\circ}$ and $S\cong S'$ as polarized K3 surfaces, then $(X,S)\cong (X',S')$. This clearly holds when $S$ and $S'$ are quartic surfaces in $\bP^3$. If $S$ and $S'$ are hyperelliptic, then their quotients  $(V,C)$ and $(V',C')$ by the hyperelliptic involutions are isomorphic. By Theorem \ref{thm:H_h-embed} and Proposition \ref{prop:no-other-3-fold}, we know that $(X,cS)\cong \sC(V,\frac{3c-1}{4}C)\cong \sC(V',\frac{3c-1}{4}C')\cong (X',cS')$. If $S$ and $S'$ are unigonal, then Proposition \ref{prop:no-other-3-fold} implies that $X\cong X'\cong X_u$, and hence $(X,S)\cong (X',S')$ by Lemma \ref{lem:X_u-anti-can}. Thus $\fp_c^\circ$ is injective on closed points. By Zariski's main theorem and normality of K-moduli spaces (Theorem \ref{thm:k-moduli-walls}), we conclude that 
$\fp_c^{\circ}$ is an open immersion.

Next we turn to parts (1) -- (3). For part (1), if $c\in (0,\frac{1}{3}]$, then by Lemma \ref{lem:ADE-K} and Theorem \ref{thm:firstwall} we know that $\fM_{c}^{\K,\circ}\cong \fM^\circ$ under the isomorphism $\ofM_c^{\K}\cong \ofM^{\GIT}$, which implies that $\fM_{c}^{\K,\circ}$ is a big open subset of $\ofM_c^{\K}$. If $c>\frac{1}{3}$, then a general hyperelliptic K3 surface $S\subset X_h$ satisfies that $(X_h, cS)$ is K-stable by Theorem \ref{thm:H_h-embed} and \cite[Theorem 1.1]{ADL20}. If $c>\frac{9}{13}$, then a general unigonal K3 surface $S'\subset X_u$ satisfies that $(X_u,cS')$ is K-stable by Theorem \ref{thm:T_0-replace}. Thus we know that $\fM_c^{\K,\circ}$ has non-empty intersection with $H_{h,c}$ (resp. $H_{u,c}$) if $c>\frac{1}{3}$ (resp. if $c>\frac{9}{13}$). Since there is an open immersion $\fM^\circ\hookrightarrow \fM_c^{\K,\circ}$ for any $c$ by Lemma \ref{lem:ADE-K}, and all but two wall crossings are flips by Theorem \ref{thm:k-moduli-walls}, we know that $\fM_c^{\K,\circ}$ is a big open subset of $\ofM_c^{\K}$ for any $c\in (0,1)$. Part (2) follows for similar reasons to part (1) as $\fM^\circ \cong \sF\setminus(H_h\cup H_u)$ under $\fp$. For part (3), since $\fp_c^\circ$ is an open immersion, it suffices to show that it is surjective if and only if $c\in (\frac{9}{11},1)$. We know that $\sF\setminus (H_h\cup H_u)\cong \fM^\circ \hookrightarrow \fM_c^{\K,\circ}$ for any $c$, and $H_u\subset \fM_c^{\K,\circ}$ if and only if $c> \frac{9}{13}$ by Theorem \ref{thm:T_0-replace}. Thus the surjectivity of $\fp_c^\circ$ is equivalent to the containment $H_h\subset  H_{h,c}$. Since $H_{h,c}\cong \oK_{\frac{3c-1}{4}}$ by Theorem \ref{thm:H_h-embed}, from the explicit wall crossings for K-moduli of hyperelliptic quartic K3 surfaces (see \cite[Remarks 5.13 and 5.14]{ADL20} and \cite[Section 6]{LO18a}) we know  that $H_h\subset  H_{h,c}$ if and only if $\frac{3c-1}{4}>\frac{4}{11}$, i.e. $c>\frac{9}{11}$. Thus the proof is finished.
\end{proof}

Next, we recall the definition of log CM line bundles on K-moduli stacks and spaces from \cite[Definition 3.34]{ADL19}. By the construction of K-moduli stacks from \cite[Section 3]{ADL19}, we may write $\osM_c^{\K}=[Z_{c}^{\red} /\PGL(N+1)]$ where $Z_c^{\red}$ is a reduced locally closed subscheme of certain relative Hilbert scheme in $\bP^N$ parametrizing $c$-K-semistable pairs in $\osM_c^{\K}$. Let $\pi_c: (\cX_c,\cD_c)\to Z_c^{\red}$ be the universal family. Let $c'\in (0,1)$ be a rational number. Then the log CM $\bQ$-line bundle $\lambda_{\CM,\pi_c, c'\cD_c}$ (resp. the Hodge $\bQ$-line bundle $\lambda_{\Hdg,\pi_c, \cD_c}$) on $Z_c^{\red}$ descends to a $\bQ$-line bundle $\lambda_{c,c'}$ (resp. $\lambda_{c,\Hdg}$) on $\osM_c^{\K}$. We simply denote $\lambda_c:=\lambda_{c,c}$. By \cite[Proposition 3.35]{ADL19}, we know that $\lambda_c$ descends to a $\bQ$-line bundle $\Lambda_c$ on $\ofM_{c}^{\K}$ for any $c\in (0,1)$. Moreover, if $c\in (0,1)$ is not a K-moduli wall listed in Theorem \ref{thm:k-moduli-walls}, then both $\lambda_{c,c'}$ and $\lambda_{c,\Hdg}$ descend to $\bQ$-line bundles $\Lambda_{c,c'}$ and $\Lambda_{c,\Hdg}$ on $\ofM_{c}^{\K}$. By \cite{XZ19}, we know that the CM $\bQ$-line bundle $\Lambda_c$ is ample on $\ofM_c^{\K}$ for any $c\in (0,1)$. Moreover, the Hodge $\bQ$-line bundle $\Lambda_{c,\Hdg}$ is nef on $\ofM_c^{\K}$ for any $c\in (\frac{9}{11},1)$ by Theorem \ref{thm:k-moduli-walls} and \cite[Proposition 3.35]{ADL19}.

%Denote by $\lambda$ the Hodge line bundle on $\sF$.  Denote by $\Delta:=\frac{1}{2}( H_h+H_u )$ and $\Delta^{\K}:= \frac{1}{4}H_h+ \frac{9}{8}H_u$.

\begin{thm}\label{thm:CM-proportional}
Let $\lambda$ be the Hodge line bundle on $\sF$.
Let $\Delta^{\K}:= \frac{1}{4}H_h +\frac{9}{8} H_u$. Then for any $c\in (0,1)\cap \bQ$ we have \[
\ofM_c^{\K}\cong \Proj ~ R(\sF, c\lambda + (1-c)\Delta^{\K}).
\]
Moreover, the CM $\bQ$-line bundle $\Lambda_{c}$ on $\ofM_c^{\K}$ is proportional to $(\fp_c^{-1})_*(c\lambda+ (1-c)\Delta^{\K})$  up to a positive constant.
\end{thm}

\begin{proof}
From Theorem \ref{thm:k-moduli-walls}, we know that the birational map $\ofM_{c'}^{\K}\dashrightarrow \ofM_c^{\K}$ is a birational contraction for any $0<c<c'<1$. Moreover, by \cite{XZ19} we know that $\Lambda_c$ is ample on $\ofM_c^{\K}$. Thus similar arguments to \cite[Theorem 9.4]{ADL19} imply that for any $c\in (0,1)\cap\bQ$ and $0<\epsilon\ll 1$ we have
\[
\ofM_c^{\K}\cong \Proj ~R(\ofM_{1-\epsilon}^{\K}, \Lambda_{1-\epsilon,c}),
\]
and $\Lambda_c$ is the same as the pushforward of $\Lambda_{1-\epsilon,c}$ under $\ofM_{1-\epsilon}^{\K}\dashrightarrow\ofM_c^{\K}$. By Proposition \ref{prop:loci-ADE-K3}(3), we know that $\fp_{1-\epsilon}^{-1}: \sF\hookrightarrow \ofM_{1-\epsilon}^{\K}$ is a regular open immersion whose image $\fM_{1-\epsilon}^{\K,\circ}$ is a big open subset of $\ofM_{1-\epsilon}^{\K}$. Since $\ofM_{1-\epsilon}^{\K}$ is normal by Theorem \ref{thm:k-moduli-walls},  to prove the theorem it suffices to show that $(\fp_{1-\epsilon}^{-1})^* \Lambda_{1-\epsilon,c}$ is proportional to $c\lambda+(1-c)\Delta^{\K}$ up to a positive constant.

% We follow the notation of \cite[Definition 3.34]{ADL19}. We have shown in ???? that 
% \[
% (\ofM_{1-\epsilon}^{\K}, ~\Lambda_{\Hdg, 1-\epsilon})\cong (\hsF,~ \widehat{\lambda}).
% \]
By \cite[Proposition 3.35]{ADL19} we know that 
\begin{equation}\label{eq:CM-Hodge}
(1-c)^{-3}\Lambda_{1-\epsilon, c}=(1-c)\Lambda_{1-\epsilon,0}+4^4 c\Lambda_{1-\epsilon,\Hdg}.
\end{equation}
By adjunction, we have $(\fp_{1-\epsilon}^{-1})^* \Lambda_{1-\epsilon,\Hdg}=\lambda$. Since $\fp_{1-\epsilon}^{-1}(\sF\setminus (H_h\cup H_u))$ parametrizes pairs $(\bP^3, S)$, we know that the pullback of $\Lambda_{1-\epsilon,0}$ to $\sF\setminus (H_h\cup H_u)$ is trivial as the underlying family of Fano threefolds is an isotrivial $\bP^3$-fibration. Thus the support of  $4^{-4}(\fp_{1-\epsilon}^{-1})^*\Lambda_{1-\epsilon,0}$ is contained in $H_h\cup H_u$, and we may write   
$4^{-4}(\fp_{1-\epsilon}^{-1})^*\Lambda_{1-\epsilon,0}= b_h H_h+b_u H_u$ for some $b_h,b_u\in\bQ$. Hence we have
\begin{equation}\label{eq:CM-variation}
4^{-4}(1-c)^{-3}(\fp_{1-\epsilon}^{-1})^*\Lambda_{1-\epsilon, c}=c\lambda + (1-c)(b_h H_h+b_u H_u).
\end{equation}
Therefore, the theorem reduces to showing $4^{-4} (\fp_{1-\epsilon}^{-1})^*\Lambda_{1-\epsilon,0}=\Delta^{\K}$, i.e. $b_h=\frac{1}{4}$ and $b_u=\frac{9}{8}$.
% By Theorem ??? we know that the birational map $\widehat{\psi}_{c}^{\K}:\hsF\dashrightarrow \ofM_c^{\K}$ is always a birational contraction. Moreover, the CM line bundle $\lambda_{\CM,c}$ is ample on $\ofM_c^{\K}$ by \cite{XZ19}. Hence similar arguments to \cite[Theorem 9.4]{ADL19} implies that 
% \[
% \ofM_c^{\K}\cong \Proj R(\hsF,c\widehat{\lambda}+(1-c)4^{-4}\Lambda_{1-\epsilon,0}),
% \]
% and 
% \[
% \lambda_{\CM, c}= 4^4(1-c)^3 (\widehat{\psi}_{c}^{\K})_*(c\widehat{\lambda}+(1-c)4^{-4}\Lambda_{1-\epsilon,0}).
% \]

Let $\Lambda^{\GIT}$ be the ample $\bQ$-line bundle on $\ofM^{\GIT}$ induced by the hyperplane line bundle on $|\cO_{\bP^3}(4)|$. Then by \cite[(4.1.2)]{LO19} we have
\begin{equation}\label{eq:GIT-lambda}
\fp^*\Lambda^{\GIT}=\lambda+\frac{H_h}{2}+\frac{H_u}{2},
\end{equation}
where $\fp:\ofM^{\GIT}\dashrightarrow \sF$ is the birational period map.

Next we compute $b_u$.
By Theorem \ref{thm:T_0-replace}, for every $c\in (\frac{9}{13},1)$ we have that $H_u\subset H_{u,c}$, and  $\ofM_c^{\K}\dashrightarrow \ofM^{\GIT}$ is regular near $H_{u,c}$ which contracts $H_{u,c}$ to the point $[T]$. Thus $(\fp^*\Lambda^{\GIT})|_{H_u}$ is $\bQ$-trivial. Since $H_h\cap H_u=\emptyset$ by \cite[Lemma 1.7.3]{LO19}, \eqref{eq:GIT-lambda} implies that 
\begin{equation}\label{eq:bu-1}
\left.\left(\lambda+\frac{H_u}{2}\right)\right|_{H_u}= 0.
\end{equation}
On the other hand, by \cite[Theorem 3.36]{ADL19} we know that 
$\Lambda_{\frac{9}{13}+\epsilon,\frac{9}{13}}=(\phi_5^+)^*\Lambda_{\frac{9}{13}}$ which is trivial along $H_u\subset H_{u, \frac{9}{13}+\epsilon}$. Since $\ofM_{1-\epsilon}^{\K}\dashrightarrow\ofM_{\frac{9}{13}+\epsilon}^{\K}$ is isomorphic near $H_u$, we have that 
$0 = \Lambda_{\frac{9}{13}+\epsilon,\frac{9}{13}}|_{H_u} = \Lambda_{1-\epsilon,\frac{9}{13}}|_{H_u}$. 
By \eqref{eq:CM-variation} we get 
\begin{equation}\label{eq:bu-2}
0= \left.\left(\frac{9}{13}\lambda+ \frac{4}{13} (b_h H_h+b_u H_u)\right)\right|_{H_u}=\frac{9}{13}\left.\left(\lambda+\frac{4}{9}b_u H_u\right)\right|_{H_u}
\end{equation}
Since $\lambda|_{H_u}$ is ample but not $\bQ$-trivial, equations \eqref{eq:bu-1} and \eqref{eq:bu-2} imply that $b_u=\frac{9}{8}$.

Finally, we compute $b_h$. Denote by $H_h^\circ:= H_h \setminus Z^2$ parametrizing hyperelliptic K3 surfaces $S$ with ADE singularities that are double covers of $\bP^1\times\bP^1$. By Theorem \ref{thm:firstwall}, we know that $(X_h, (\frac{1}{3}+\epsilon)S)$ is K-stable. Thus interpolation \cite[Proposition 2.13]{ADL19} implies that $(X_h, cS)$ is K-stable for any $c\in (\frac{1}{3}, 1)$, i.e. $H_h^\circ \subset H_{h,c}$. Similar to the $H_u$ case, we have $\Lambda_{\frac{1}{3}+\epsilon,\frac{1}{3}}=(\phi_1^+)^*\Lambda_{\frac{1}{3}}$ which is trivial along $H_h^\circ$, and $\ofM_{1-\epsilon}^{\K}\dashrightarrow\ofM_{\frac{1}{3}+\epsilon}^{\K}$ is isomorphic near $H_h^\circ$. Thus \eqref{eq:CM-variation} implies 
\begin{equation}\label{eq:bh-1}
0 =  \left.\left(\frac{1}{3}\lambda+ \frac{2}{3} (b_h H_h+b_u H_u)\right)\right|_{H_h^\circ}=\frac{1}{3}\left.(\lambda+2b_h H_h)\right|_{H_h^\circ}.
\end{equation}
Meanwhile, \eqref{eq:GIT-lambda} implies that 
\begin{equation}\label{eq:bh-2}
(\lambda+\frac{1}{2}H_h)|_{H_h^\circ}=0.
\end{equation}
It is clear that $H_h^\circ$ is a big open subset of $H_{h,\frac{1}{3}+\epsilon}$ which is isomorphic to the GIT moduli space of $(4,4)$-curves on $\bP^1\times\bP^1$. Thus $\lambda|_{H_h^\circ}$ is ample and not $\bQ$-trivial. 
%, and $H_{h,\frac{1}{3}+\epsilon}$ restricting to itself is anti-ample as it is the exceptional divisor of the weighted blow-up $\ofM_{\frac{1}{3}+\epsilon}^{\K}\to \ofM^{\GIT}$, thus $H_h|_{H_h^\circ}$ is anti-ample but not $\bQ$-trivial.
This combining with \eqref{eq:bh-1} and \eqref{eq:bh-2} implies that $b_h=\frac{1}{4}$. 
\end{proof}
% Since $\hsF\setminus (\widehat{H}_h\cup \widehat{H}_u)$ parametrizes log pairs $(\bP^3, S)$, we know that the support of $\Lambda_{1-\epsilon,0}$ is contained in $(\widehat{H}_h\cup \widehat{H}_u)$. Denote by $4^{-4}\Lambda_{1-\epsilon,0}= b_h \widehat{H}_h+b_u \widehat{H}_u$ on $\hsF$. 
% By Laza-O'Grady, we know that $H_h\cap H_u=\emptyset$, and $\lambda+\Delta$ is trivial after restricting to $H_h$ or $H_u$. By Theorem ???, we know that when $c=\frac{1}{3}$ (resp. $c=\frac{9}{13}$) the birational transform of the divisor $H_h$ (resp. $H_u$) is contracted under $\ofM_{c+\epsilon}^{\K}\to \ofM_c^{\K}$. Hence we have that 
% \[
% (\frac{1}{3}\lambda+\frac{2}{3}(b_h H_h + b_u H_u))|_{H_h}\sim_{\bQ} 0,\quad 
% (\frac{9}{13}\lambda+\frac{4}{13}(b_h H_h + b_u H_u))|_{H_u}\sim_{\bQ} 0.
% \]
% Thus we get $b_h=\frac{1}{4}$ and $b_u=\frac{9}{8}$. The proof is finished.
%\end{proof}

\begin{lem}\label{lem:small-variation}
For any $\epsilon\in (0,\frac{2}{11})\cap \bQ$, $a\in (0, \frac{1}{9})\cap\bQ$, and $b\in (0,\frac{1}{2})\cap\bQ$, we have 
\[
\ofM_{1-\epsilon}^{\K}\cong \Proj ~ R(\sF, \lambda + \tfrac{a}{2}H_h + \tfrac{b}{2} H_u)\cong \hsF.
\]
\end{lem}

\begin{proof}
We focus on the first isomorphism, as the second isomorphism is an easy consequence of \cite[Proposition 17]{LO18b} where it is shown that $\hsF \cong \Proj~ R(\sF, \lambda+\epsilon\Delta)$ for $0<\epsilon \ll 1$.
By Theorem \ref{thm:k-moduli-walls}, we know that $\ofM_{1-\epsilon}^{\K}$ is independent of the choice of $\epsilon \in (0,\frac{2}{11})\cap\bQ$. 
Since $\ofM_{1-\epsilon}^{\K}$ and $\sF$ are isomorphic in codimension $1$, it suffices to show that  $(\fp_{1-\epsilon}^{-1})_*(\lambda+\frac{1}{2} (aH_h + b H_u))$ is ample on $\ofM_{1-\epsilon}^{\K}$ for $\epsilon\in (0,\frac{2}{11})\cap \bQ$, $a\in (0, \frac{1}{9})\cap\bQ$, and $b\in (0,\frac{1}{2})\cap\bQ$. We split into two cases.

\textbf{Case 1: $b\leq \frac{9}{2}a$.} 
Since $a<\frac{1}{9}$, we may choose $\epsilon :=\frac{2a}{1+2a}<\frac{2}{11}$.  As $\Delta^{\K}=\frac{1}{4}H_h+\frac{9}{8}H_u$, we have 
\begin{equation}\label{eq:case1}
\lambda+\frac{1}{2} (aH_h + bH_u) = \frac{2b}{9a}(\lambda + 2a \Delta^{\K}) + (1-\frac{2b}{9a})(\lambda+\frac{a}{2}H_h).
\end{equation}
By Theorems \ref{thm:k-moduli-walls} and \ref{thm:CM-proportional},  $(\fp_{1-\epsilon}^{-1})_*(\lambda+2a \Delta^{\K})=(1+2a)(\fp_{1-\epsilon}^{-1})_*((1-\epsilon)\lambda+\epsilon \Delta^{\K})$ is a positive multiple of the CM $\bQ$-line bundle $\Lambda_{1-\epsilon}$ on $\ofM_{1-\epsilon}^{\K}$ hence is ample by \cite{XZ19}. As we mentioned earlier, 
$(\fp_{1-\epsilon}^{-1})_*\lambda=\Lambda_{1-\epsilon,\Hdg}$ is nef on $\ofM_{1-\epsilon}^{\K}$. Thus by \eqref{eq:case1} it suffices to show that $(\fp_{1-\epsilon}^{-1})_*(\lambda+ \frac{a}{2}H_h)=\Lambda_{1-\epsilon,\Hdg} + \frac{a}{2}H_{h,1-\epsilon}$ is nef on $\ofM_{1-\epsilon}^{\K}$. Since $\Lambda_{1-\epsilon,\Hdg}$ is nef and $H_{h,1-\epsilon}$ is effective $\bQ$-Cartier, it suffices to show that 
$(\Lambda_{1-\epsilon,\Hdg} + \frac{a}{2}H_{h,1-\epsilon})|_{H_{h,1-\epsilon}}$
is nef. By Theorem \ref{thm:k-moduli-walls}, we know that $H_{h,1-\epsilon}$ and $H_{u, 1-\epsilon}$ are disjoint. Thus 
\[
(\Lambda_{1-\epsilon,\Hdg} + \frac{a}{2}H_{h,1-\epsilon})|_{H_{h,1-\epsilon}}=(\Lambda_{1-\epsilon,\Hdg} + \frac{a}{2}H_{h,1-\epsilon}+\frac{9a}{4} H_{u, 1-\epsilon})|_{H_{h,1-\epsilon}}
\]
is a positive multiple of $\Lambda_{1-\epsilon}|_{H_{h,1-\epsilon}}$ which is ample. Thus Case 1 is proved.

\textbf{Case 2: $b > \frac{9}{2}a$.}
Since $b<\frac{1}{2}$, we may choose $\epsilon :=\frac{4b}{9+4b}<\frac{2}{11}$.  Then we have 
\[
\lambda+\frac{1}{2} (aH_h + bH_u) = \frac{9a}{2b}(\lambda + \frac{4b}{9} \Delta^{\K}) + (1-\frac{9a}{2b})(\lambda+\frac{b}{2}H_u).
\]
Similarly to Case 1, $(\fp_{1-\epsilon}^{-1})_*(\lambda+\frac{4b}{9} \Delta^{\K})=(1+\frac{4b}{9})(\fp_{1-\epsilon}^{-1})_*((1-\epsilon)\lambda+\epsilon \Delta^{\K})$ is ample as a positive multiple of the CM $\bQ$-line bundle $\Lambda_{1-\epsilon}$. Thus it suffices to show the nefness of 
$(\Lambda_{1-\epsilon,\Hdg} + \frac{b}{2}H_{u,1-\epsilon})|_{H_{u,1-\epsilon}}$ as $H_{u,1-\epsilon}$ is also effective $\bQ$-Cartier. This again follows from the disjointness of  $H_{h,1-\epsilon}$ and $H_{u, 1-\epsilon}$  and the ampleness of  $\Lambda_{1-\epsilon}|_{H_{u,1-\epsilon}}$. Thus Case 2 is proved.
\end{proof}

\begin{proof}[Proof of Theorem \ref{mthm:Kmod}]
Part (1) is a consequence of Theorem \ref{thm:firstwall}.  Part (2) is exactly Theorem \ref{thm:CM-proportional}. Part (4) follows from Theorem \ref{thm:k-moduli-walls}. Hence we only need to prove part (3).
By construction, $\sF$, $\hsF$, and $\sF^*$ are all isomorphic in codimension $1$. Hence Lemma \ref{lem:small-variation} implies $(\hsF,\hat{\lambda})\cong (\ofM_{1-\epsilon}^{\K},\Lambda_{1-\epsilon,\Hdg})$ where $\hat{\lambda}$ is the unique extension of $\lambda$ on $\hsF$. Clearly, $\lambda$ uniquely extends to an ample $\bQ$-line bundle $\lambda^*$ on $\sF^*$ whose pullback under the morphism $\hsF\to \sF^*$ is exactly $\hat{\lambda}$. Thus (3) is proved.
\end{proof}

\subsection{Proof of Laza-O'Grady's prediction}

% According to the above computation, we have 
% \[
% c\lambda+(1-c)\Delta^{\K}=c\lambda + \frac{1-c}{4}H_h +\frac{9(1-c)}{8}H_u = c(\lambda + \frac{1-c}{2c}\Delta)+\frac{7(1-c)}{8} H_u.
% \]
% Thus the K-moduli walls for $c$ (except $c=\frac{9}{13}$) relates to the Laza-O'Grady predicted walls as 
% \[
% \beta=\frac{1-c}{2c}, \quad c=\frac{1}{1+2
% \beta}.
% \]

%In \cite{LO19}, Laza and O'Grady predicted that the section ring  $R(\sF, \lambda+ \beta \Delta)$ is finitely generated for $\beta\in [0,1]\cap\bQ$ and $\Delta=\frac{1}{2}(H_h+H_u)$.

In this section, we prove Theorem \ref{mthm:fg} which implies  Laza-O'Grady's prediction. The key idea is to construct $\sF(a,b)$ from modifications of K-moduli spaces $\ofM_{c}^{\K}$, use positivity of the log CM line bundle \cite{CP18, Pos19, XZ19}, and follow the MMP with scaling from \cite{KKL16}. 

We recall some notation and consequences from the proof of Theorem \ref{thm:T_0-replace}. We define $U_T=\ofM^{\GIT}\setminus W_8$ and $U_u=(\phi_5^+)^{-1} (\phi_5^{-}(U_T))\subset \ofM^{\K}_{\frac{9}{13}+ \epsilon}$ where $\phi_5^{\pm}:\ofM^{\K}_{\frac{9}{13}\pm \epsilon}\to \ofM^{\K}_{\frac{9}{13}}$ are the K-moduli wall crossing morphisms. Moreover, there are canonical open immersions $U_T\hookrightarrow \ofM_c^{\K}$ and $U_u \hookrightarrow \ofM_{c'}^{\K}$ for any $0< c< \frac{9}{13}<c'<1$. Let $\rho_u: U_u\to U_T$ be the composition $\rho_u : = ((\phi_5^{-})^{-1}\circ \phi_5^+)|_{U_u}$. Then $\rho_u$ is a projective birational morphism that contracts the divisor $H_{u, \frac{9}{13}+\epsilon}$ to the point $[T]\in U_T$. In addition, the restriction of $\rho_u$ to $U_u \setminus H_{u, \frac{9}{13}+\epsilon}$, denoted by $\rho_u^\circ$, is  an isomorphim onto $U_T\setminus\{[T]\}$.

\begin{defn}\label{def:F(a,b)}%\footnote{needs to be fixed \YL{done}}
Let $a,b\in \bQ_{>0}$. Denote by $c=c(a):=\frac{1}{1+2a}$. We define the schemes $\sF(a,b)$ as follows.
\begin{enumerate}
%    \item $\sF(0):= \sF^*$, $\sF(1):=\ofM^{\GIT}$;
    \item For $a\in (0,\frac{2}{9})$ and $b\in (0,1)$, we define $\sF(a,b):=\ofM_c^{\K}$;
    \item For $a\in (\frac{2}{9},+\infty)$ and $b\in (0,1)$, we define $\sF(a,b)$ as the gluing of $\ofM_c^{\K}\setminus \{[(\bP^3, T)]\}$ and $U_u$ through the isomorphism of open subschemes $U_T\setminus\{[T]\}\xrightarrow[\cong]{(\rho_u^\circ)^{-1}} U_u\setminus H_{u, \frac{9}{13}+\epsilon}$, and $\sF(\frac{2}{9},b):=\sF(\frac{2}{9}+\epsilon,b)$ for $0<\epsilon\ll 1$;
    %we define $\sF(a,b)$ as the Kirwan weighted blow-up of $\ofM_c^{\K}$ at $[(\bP^3, T)]$ which extracts the birational transform of $H_u$, and $\sF(\frac{2}{9},b):=\sF(\frac{2}{9}+\epsilon,b)$ for $0<\epsilon\ll 1$;
   % \footnote{replace (2) and (3) by defining F(a,b) using gluing; make it by gluing together $U$ and $V$ where $U$ is a neighborhood of $[T]$ (defined in proof of 4.20) and $V$ is a neighborhood not including $[T]$; then after wall we glue $U'$ and $V$ instead where $U'$ is extracting the divisor over $[T]$...gluing should be okay (proper, projective, etc) because we are only gluing two sets and the only thing modified is in $U - V$ (so their intersection remains the same) \YL{done}}
    \item For $a\in (0,\frac{2}{9})$ and $b\in [1,+\infty)$, we define $\sF(a,b)$ as the gluing of $\ofM_c^{\K}\setminus H_{u,c}$ and $U_{T}$ through the isomorphism of open subschemes $U_u\setminus H_{u, \frac{9}{13}+\epsilon} \xrightarrow[\cong]{\rho_u^\circ} U_T\setminus\{[T]\} $;
    %the Kirwan weighted blow-down of $\ofM_c^{\K}$ that contracts $H_{u,c}$ to a point;
    \item For $a\in [\frac{2}{9},+\infty)$ and $b \in [1,+\infty)$ we define $\sF(a,b):=\ofM_c^{\K}$.
\end{enumerate}

\begin{prop}\label{prop:F(a,b)proper}
For every $a,b\in \bQ_{>0}$, the scheme $\sF(a,b)$ is an irreducible normal proper scheme. Denote by $c=c(a):=\frac{1}{1+2a}$. Moreover, if $a\in (\frac{2}{9},+\infty)$ and $b\in (0,1)$, then there is a birational morphism $\sigma_{a,b}:\sF(a,b) \to \ofM_c^{\K}$ which contracts $H_{u, \frac{9}{13}+\epsilon}$ to the point $[(\bP^3, T)]$, and is isomorphic elsewhere; if $a\in (0,\frac{2}{9})$ and $b\in [1,+\infty)$, then there is a birational morphism $\sigma_{a,b}':\ofM_c^{\K} \to \sF(a,b)$ which contracts $H_{u,c}$ to the point $[T]$, and is isomorphic elsewhere.
\end{prop}

\begin{proof}
We first construct the birational morphisms.

Suppose $a\in (\frac{2}{9}, +\infty)$ and $b\in (0,1)$ which implies $c\in (0, \frac{9}{13})$. Hence by Theorem \ref{thm:T_0-replace} we know that $\ofM_c^{\K}$ is the gluing of $\ofM_c^{\K}\setminus \{[(\bP^3, T)]\}$ and $U_T$ through the common open subschemes $U_T\setminus\{[T]\}$. Thus by Definition \ref{def:F(a,b)}(2) we know that there is a birational morphism $\sigma_{a,b}:\sF(a,b) \to \ofM_c^{\K}$ obtained by gluing    the identity map on  $\ofM_c^{\K}\setminus \{[(\bP^3, T)]\}$ and $\rho_u: U_u \to U_T$. Since $\rho_u$ is a projective morphism, so is $\sigma_{a,b}$. Thus $\sF(a,b)$ is a projective scheme. 

Suppose $a\in (0,\frac{2}{9})$ and $b\in [1,+\infty)$ which implies $c\in (\frac{9}{13},1)$. Hence by Theorem \ref{thm:T_0-replace} we know that $\ofM_c^{\K}$ is the gluing of two $\ofM_c^{\K}\setminus H_{u,c}$ and $U_u$ through the common open subschemes $U_u\setminus H_{u, \frac{9}{13}+\epsilon}$. Thus by Definition \ref{def:F(a,b)}(3) we know that there is a birational morphism $\sigma_{a,b}':\ofM_c^{\K}\to\sF(a,b)$ obtained by gluing   the identity map on  $\ofM_c^{\K}\setminus H_{u,c}$ and $\rho_u: U_u \to U_T$. Then it is clear that $\sigma_{a,b}$ is a  surjective proper morphism, which implies that $\sF(a,b)$ is a proper scheme by \cite[\href{https://stacks.math.columbia.edu/tag/03GN}{Tag 03GN} and \href{https://stacks.math.columbia.edu/tag/09MQ}{Tag 09MQ}]{stacks-project}.

Finally, the irreducibility and normality of $\sF(a,b)$ come from the corresponding properties of $\ofM_c^{\K}$ by Theorem \ref{thm:k-moduli-walls}.
\end{proof}

Denote by $\psi_{a,b}: \sF\dashrightarrow \sF(a,b)$ the inverse of the birational period map. Let $\lambda(a,b)$, $H_h(a,b)$, and $H_u(a,b)$ be the pushforward of $\lambda$, $H_h$, and $H_u$ under $\psi_{a,b}$, respectively. By Theorem \ref{thm:k-moduli-walls}, Definition \ref{def:F(a,b)} and Proposition \ref{prop:F(a,b)proper}, we know that such $\sF(a,b)$'s undergo wall crossings at $a=a_i$ or $b=1$ where
\[
(a_1, \cdots, a_8)=\left(\frac{1}{9},  \frac{1}{7}, \frac{1}{6}, \frac{1}{5}, \frac{1}{4}, \frac{1}{3}, \frac{1}{2},1\right)
\]
We denote the wall crossing morphisms for a fixed $b\in (0,1)$ by $\sF(a_i-\epsilon,b)\xrightarrow{\varphi_i^-} \sF(a_i,b)\xleftarrow{\varphi_i^+}\sF(a_i+\epsilon,b)$.
\end{defn}

\begin{lem}\label{lem:antiample}%\footnote{\YL{This is a new lemma.}}
Recall the birational morphisms $\rho: \ofM_{\frac{1}{3}+\epsilon}^{\K}\to \ofM^{\GIT}$ from Theorem \ref{thm:firstwall} and $\rho_u: U_u\to U_T$ from Definition \ref{def:F(a,b)}. Then $H_{h, \frac{1}{3}+\epsilon}$ is $\rho$-anti-ample, and $H_{u, \frac{9}{13}+\epsilon}$ is $\rho_u$-anti-ample.
\end{lem}

\begin{proof}
Since $\rho$ and $\rho_u$ come from wall-crossing morphisms of $\ofM_c^{\K}$ at $c=\frac{1}{3}$ and $c=\frac{9}{13}$ respectively, by Theorem \ref{thm:CM-proportional} we know that $(\fp_{c+\epsilon}^{-1})_*\Delta^{\K}$ on $\ofM_{c+\epsilon}^{\K}$ is relatively anti-ample over $\ofM_c^{\K}$ for $c\in \{\frac{1}{3}, \frac{9}{13}\}$. For $c=\frac{1}{3}$, we know that $(\fp_{\frac{1}{3}+\epsilon}^{-1})_*\Delta^{\K} = \frac{1}{4}H_{h, \frac{1}{3}+\epsilon}$ which implies the first statement as $\ofM_{\frac{1}{3}}^{\K}\cong \ofM^{\GIT}$. For $c=\frac{9}{13}$, we know that   $(\fp_{\frac{1}{3}+\epsilon}^{-1})_*\Delta^{\K} = \frac{1}{4}H_{h, \frac{9}{13}+\epsilon} + \frac{9}{8}H_{u, \frac{9}{13}+\epsilon} $. From the definition we know that $H_{h, \frac{9}{13}+\epsilon}$ is disjoint from $U_u$, hence $\frac{9}{8}H_{u, \frac{9}{13}+\epsilon} $ is relatively anti-ample over $U_T$ which implies the second statement. 
\end{proof}

\begin{lem}\label{lem:CM-perturb}
Let $a,b\in (0,1)\cap\bQ$. Then the birational map $\psi_{a,b}:\sF\dashrightarrow \sF(a,b)$ is an isomorphism in codimension $1$. Moreover, we have the following. 
\begin{enumerate}
    \item If $a\in (0, \frac{2}{9})$, then $\lambda(a,b) +\frac{a}{2}H_h(a,b)+ \frac{9a}{4} H_u(a,b)$ is ample on $\sF(a,b)$;
    \item If $a\in [\frac{2}{9},1)$, then $\lambda(a,b) +\frac{a}{2}H_h(a,b) + \frac{1-\epsilon}{2} H_u(a,b)$ is ample on $\sF(a,b)$ for $0<\epsilon\ll 1$. 
    \item $\lambda(a,b)+\frac{a}{2}H_h(a,b)$ is nef for any $a,b\in (0,1)\cap\bQ$.
\end{enumerate}
\end{lem}

\begin{proof}
By Theorem \ref{thm:CM-proportional} we know that  $(\fp_c^{-1})_*(\lambda+\frac{1-c}{c}\Delta^{\K})$ is ample on $\ofM_c^{\K}$.

(1) If $a\in (0,\frac{2}{9})$, we know that $\fp_c^{-1}=\psi_{a,b}$, hence $(\psi_{a,b})_* (\lambda+\frac{1-c}{c}\Delta^{\K})$ is ample on $\sF(a,b)$ for $c=\frac{1}{1+2a}$. It is clear that 
\[
\lambda+\frac{1-c}{c}\Delta^{\K}= \lambda + 2a (\frac{1}{4} H_h + \frac{9}{8}H_u)=\lambda + \frac{a}{2} H_h + \frac{9a}{4}H_u.
\]
Hence $\lambda(a,b)+\frac{a}{2}H_h(a,b)+ \frac{9a}{4} H_u(a,b)$ is ample.

(2) If $a\in [\frac{2}{9}, 1)$, by Theorem \ref{thm:k-moduli-walls} we know that $\fp_c^{-1}$ is a birational contraction which only contracts the divisor $H_u$ since $c=\frac{1}{1+2a}\in (\frac{1}{3},\frac{9}{13}]$. Hence we know 
\[
(\fp_c^{-1})_* (\lambda+\frac{a}{2}H_h)=(\fp_c^{-1})_* (\lambda+\frac{1-c}{c}\Delta^{\K})
\]
is ample on $\ofM_c^{\K}$. By Proposition \ref{prop:F(a,b)proper} and Lemma \ref{lem:antiample}, there is a birational morphism  $\sigma_{a,b}: \sF(a,b)\to \ofM_c^{\K}$ with exceptional divisor $H_u(a,b)$ which is anti-ample over $\ofM_c^{\K}$. Since $(\lambda+\frac{1}{2}H_u)|_{H_u}\sim_{\bQ} 0$ and $H_h\cap H_u=\emptyset$ by \eqref{eq:bu-1}, we know that 
\[
(\sigma_{a,b})^* (\fp_c^{-1})_* (\lambda+\frac{a}{2}H_h) = (\psi_{a,b})_* (\lambda +\frac{a}{2}H_h+ \frac{1}{2} H_u ). 
\]
Hence $(\psi_{a,b})_* (\lambda +\frac{a}{2}H_h+ \frac{1-\epsilon}{2} H_u )$ is ample as it is the pull back of an ample $\bQ$-divisor on $\ofM_c^{\K}$ twisted by a small multiple of a $\sigma_{a,b}$-ample divisor.

(3) Since the statement is independent of the choice of $b\in (0,1)$, we may assume $b\in (0, \frac{1}{2})$. We prove nefness by induction on the walls of $a$.  To start with, we assume $a\in (0,a_1)$ where $a_1=\frac{1}{9}$. Then we have $\sF(a,b)\cong \ofM_{c}^{\K}\cong \hsF$ by Lemma \ref{lem:small-variation} as $c=\frac{1}{1+2a}\in (\frac{9}{11},1)$. 
By Lemma \ref{lem:small-variation}, we know that $\lambda(a,b)+\frac{a}{2}H_h(a,b)+\frac{b}{2}H_u(a,b)$ is ample. Hence we get nefness of $\lambda(a,b)+\frac{a}{2}H_h(a,b)$ for $a\in (0, a_1)$ by letting $b\to 0$. Next, we divide the induction into two parts. Note that we always assume $0<\epsilon\ll 1$ in this proof.
%It is clear that $\widehat{\lambda}$ is nef on $\hsF$. Since $\widehat{H}_h$ is effective, it suffices to show that $((\widehat{\lambda}+\frac{a}{2}\widehat{H}_h)\cdot C)\geq 0$ for any curve $C\subset \widehat{H}_h$. This is true since $\widehat{H}_h\cap \widehat{H}_u=\emptyset$ and $\widehat{\lambda}+\frac{a}{2}\widehat{H}_h + b \widehat{H}_u$ is ample. 
%To simplify notation, we denote by $\lambda_a$, $H_h(a,b)$, and $H_u(a,b)$ the $\psi_{a,b}$-push forward of $\lambda$, $H_h$, and $H_u$, respectively. 

Assume that $\lambda (a_i-\epsilon,b)+\frac{a_i-\epsilon}{2}H_{h}(a_i-\epsilon,b)$ is nef. Since $\sF(a_i-\epsilon,b)$ is independent of the choice of $\epsilon$, by letting $\epsilon\to 0$ we have that $\lambda(a_i-\epsilon,b)+\frac{a_i}{2}H_{h}(a_i-\epsilon,b)$ is nef. As all $\sF(a,b)$'s with $a,b\in (0,1)\cap \bQ$ are isomorphic in codimension $1$, we have that 
\begin{equation}\label{eq:nef-trace-wall}
\lambda(a_i\pm\epsilon,b)+\frac{a_i}{2}H_{h}(a_i\pm\epsilon,b)=(\varphi_i^{\pm})^* (\lambda(a_i,b)+\frac{a_i}{2}H_{h}(a_i,b)).
\end{equation}
Hence we obtain that $\lambda (a_i,b)+\frac{a_i}{2}H_{h}(a_i,b)$ is also nef.

Assume that $\lambda(a_i,b)+\frac{a_i}{2}H_{h}(a_i,b)$ is nef and $a\in (a_i, a_{i+1})$. By \eqref{eq:nef-trace-wall} and $\sF(a,b)\cong \sF(a_i+\epsilon,b)$, we know that $\lambda(a,b)+\frac{a_i}{2}H_h(a,b)$ is the $\varphi_i^+$-pull-back of $\lambda(a_i,b)+\frac{a_i}{2}H_{h}(a_i,b)$ hence is nef. Since $a>a_i$, in order to show nefness of  $\lambda(a,b)+\frac{a}{2}H_h(a,b)$ it suffices to show that $(\lambda(a,b)+\frac{a}{2}H_h(a,b))|_{H_h(a,b)}$  is nef. By parts (1) and (2), there exists $b'=b'(a)\in (0,1)$ such that $\lambda(a,b)+\frac{a}{2}H_h(a,b)+\frac{b'}{2}H_u(a,b)$ is ample. Since  $H_h(a,b)\cap H_u(a,b)=\emptyset$, 
\[
(\lambda(a,b)+\frac{a}{2}H_h(a,b))|_{H_h(a,b)} = (\lambda(a,b)+\frac{a}{2}H_h(a,b)+\frac{b'}{2}H_u(a,b))|_{H_h(a,b)}
\]
is ample.  Thus $\lambda(a,b)+\frac{a}{2}H_h(a,b)$ is nef for any $a\in (a_i,a_{i+1})$. As a result, the induction steps are validated which yield the nefness of $\lambda(a,b)+\frac{a}{2}H_h(a,b)$ for any $a,b\in (0,1)\cap\bQ$.
\end{proof}

\begin{defn}\label{def:blow-up-GIT}%\footnote{needs to be fixed}
Let $\tfM^{\GIT}:=\sF(1-\epsilon, 1-\epsilon)$ for $0<\epsilon\ll 1$. Then 
%Let $\trho: \tfM^{\GIT}\to \ofM^{\GIT}$ be the composition of two Kirwan weighted blow-ups at $[2Q]$ and $[T]$ that extracts the union of birational transforms of $H_h$ and $H_u$. 
the birational map $\sF\dashrightarrow\tfM^{\GIT}$ is isomorphic in codimension $1$. Denote the pushforwards of $\lambda$, $H_h$, and $H_u$ under this map by $\tlambda$, $\tH_h$, and $\tH_u$, respectively. By Definition \ref{def:F(a,b)} and Proposition \ref{prop:F(a,b)proper}, there is a  birational morphism $\trho: \tfM^{\GIT}\to \ofM^{\GIT}$ induced by $\sigma_{1-\epsilon, 1-\epsilon}$ that contracts  $\tH_h$ and $\tH_u$ to $[2Q]$ and $[T]$ respectively and is isomorphic elsewhere. 

From Definition \ref{def:F(a,b)} and Theorem \ref{thm:k-moduli-walls}, we see that  $\sF(a,b)\cong \ofM^{\GIT}$ if $a,b\in [1,+\infty)$, and $\sF(1-\epsilon,b_1)\cong \tfM^{\GIT}$ if $0<\epsilon\ll 1$ and $b_1\in (0,1)$.
\end{defn}

\begin{thm}\label{thm:finitegeneration}
    For any $a,b\in \bQ_{>0}$, the section ring $R(\sF, \lambda+\frac{a}{2}H_h+\frac{b}{2}H_u)$ is finitely generated, and  $\sF(a,b)\cong\Proj~ R(\sF, \lambda+\frac{a}{2}H_h+\frac{b}{2}H_u)$. In particular, every $\sF(a,b)$ is projective.%\footnote{modify this proof to remove weighted blow up/use gluing approach instead \YL{done}} 
\end{thm}

\begin{proof}
We split into four cases based on values of $a$ and $b$.

\textbf{Case 1: $a,b\in (0,1)$.}
By Lemma \ref{lem:CM-perturb}, $\psi_{a,b}:\sF\dashrightarrow\sF(a,b)$ is an isomorphism in codimension $1$. Thus it suffices to show that the $\bQ$-divisor $\lambda(a,b)+\frac{a}{2}H_h(a,b)+\frac{b}{2}H_u(a,b)$ is ample on $\sF(a,b)$ for $a,b\in (0,1)\cap\bQ$. 
%By Lemma \ref{lem:small-variation}, this is true if $a\in (0,\frac{1}{9})$ and $b\in (0, \frac{1}{2})$. We first extend to $a\in (0,1)$ and $b\in (0, \frac{1}{2})$, then extend to $a,b\in (0,1)$.
%\textbf{Case 1: $a\in (0,1)$ and $b\in (0, \frac{1}{2})$.}
%Hence we just need to prove the finite generation conjecture. The case for $a=0$ is clear. Let us assume $a\in (0,1)$, as $a=1$ case will follow easily. It is clear that $\psi_{a,b}:\sF\dashrightarrow \sF(a)$ is an isomorphism in codimension $1$. Thus it suffices to show that $(\psi_{a,b})_*(\lambda+a\Delta)=(\psi_{a,b})_*(\lambda+\frac{a}{2}H_h+\frac{a}{2}H_u)$ is $\bQ$-Cartier and ample. 

By Lemma \ref{lem:CM-perturb}(1)(2), there exists $b'=b'(a)\in (0,1)$ such that $\lambda(a,b)+\frac{a}{2}H_h(a,b)+\frac{b'}{2}H_u(a,b)$ is ample. By Lemma \ref{lem:CM-perturb}(3), we know that $\lambda(a,b)+\frac{a}{2}H_h(a,b)$ is nef. Since a strict convex combination of a nef $\bQ$-divisor and an ample $\bQ$-divisor is ample, it suffices to show the nefness of $\lambda(a,b)+\frac{a}{2}H_h(a,b)+\frac{1}{2}H_u(a,b)$. Moreover, since $\frac{b'}{2}<\frac{1}{2}$ and  $H_h(a,b)\cap H_u(a,b)=\emptyset$, it suffices to show that  
\[
(\lambda(a,b)+\frac{a}{2}H_h(a,b)+\frac{1}{2}H_u(a,b))|_{H_u(a,b)}=(\lambda(a,b)+\frac{1}{2}H_u(a,b))|_{H_u(a,b)}
\]
is nef.
Indeed, by Definition \ref{def:F(a,b)}, Theorem \ref{thm:T_0-replace}(3), and Lemma \ref{lem:small-variation}, the birational map $\hsF\dashrightarrow \sF(a,b)$ is isomorphic near $\hH_u$ and $H_u(a,b)$. Since $H_u$ is a big open subset of $\hH_u$ and $(\lambda+\frac{H_u}{2})|_{H_u}=0$ by \eqref{eq:bu-1}, we know that $(\hat{\lambda}+\frac{\hH_u}{2})|_{\hH_u}=0$. This implies $(\lambda(a,b)+\frac{1}{2}H_u(a,b))|_{H_u(a,b)}$ is $\bQ$-linearly trivial and hence nef. Thus Case 1 is proved.

\textbf{Case 2: $a\in (0,1)$ and $b\in [1,+\infty)$.} 
We have the diagram
$\sF(a,1-\epsilon)\xrightarrow{\rho_a} \sF(a,1)\xleftarrow{\cong} \sF(a,b)$ from  Definition \ref{def:F(a,b)}, where $\rho_a$ is the birational morphism that contracts the divisor $H_u(a,1-\epsilon)$ to a point. We first show that $\lambda(a,1)+\frac{a}{2}H_h(a,1)$ is ample on $\sF(a,1)$. Indeed, from Case 1 we have the ampleness of $\lambda(a,1-\epsilon)+\frac{a}{2}H_h(a,1-\epsilon)+\frac{1-\epsilon}{2}H_u(a,1-\epsilon)$ on $\sF(a,1-\epsilon)$. By letting $\epsilon\to 0$, we know that  $\lambda(a,1-\epsilon)+\frac{a}{2}H_h(a,1-\epsilon)+\frac{1}{2}H_u(a,1-\epsilon)$ is big and nef, whose restrict to $H_u(a,1-\epsilon)$ is $\bQ$-linearly trivial. Thus
\[
\lambda(a,1-\epsilon)+\frac{a}{2}H_h(a,1-\epsilon)+\frac{1}{2}H_u(a,1-\epsilon)=\rho_a^*(\lambda(a,1)+\frac{a}{2}H_h(a,1)).
\]
This shows that $\lambda(a,1)+\frac{a}{2}H_h(a,1)$ is big and nef. By the Nakai-Moishezon criterion, to show the ampleness of $\lambda(a,1)+\frac{a}{2}H_h(a,1)$, it suffices to show that $(\lambda(a,1)+\frac{a}{2}H_h(a,1))|_V$ is big for any positive dimensional closed subvariety $V\subset \sF(a,1)$. Let $\tV\subset \sF(a,1-\epsilon)$ be the birational transform of $V$. Then clearly $\tV\not\subset H_u(a,1-\epsilon)$ as $\rho_a$ contracts $H_u(a,1-\epsilon)$ to a point. Since $(\lambda(a,1-\epsilon)+\frac{a}{2}H_h(a,1-\epsilon)+\frac{1-\epsilon}{2}H_u(a,1-\epsilon))|_{\tV}$ is ample, we know that 
\[
(\lambda(a,1-\epsilon)+\frac{a}{2}H_h(a,1-\epsilon)+\frac{1}{2}H_u(a,1-\epsilon))|_{\tV}=\rho_a^*((\lambda(a,1)+\frac{a}{2}H_h(a,1))|_V)
\]
is big. Thus $(\lambda(a,1)+\frac{a}{2}H_h(a,1))|_V$ is big which implies the ampleness of $\lambda(a,1)+\frac{a}{2}H_h(a,1)$. 

Since $\sF(a,1-\epsilon)$ and $\sF$ are isomorphic in codimension $1$, Case 2 reduces to showing
\[
\Proj~ R(\sF(a,1-\epsilon),\lambda(a,1-\epsilon)+\frac{a}{2}H_h(a,1-\epsilon)+\frac{b}{2}H_u(a,1-\epsilon))\cong \sF(a,1).
\]
This is true because $\lambda(a,1)+\frac{a}{2}H_h(a,1)$ is ample, $H_u(a,1-\epsilon)$ is $\rho_a$-exceptional, and
\[
\lambda(a,1-\epsilon)+\frac{a}{2}H_h(a,1-\epsilon)+\frac{b}{2}H_u(a,1-\epsilon)=\rho_a^*(\lambda(a,1)+\frac{a}{2}H_h(a,1))+\frac{b-1}{2} H_u(a, 1-\epsilon).
\]
Thus Case 2 is proved.

\textbf{Case 3: $a\in [1,+\infty)$ and $b\in (0,1)$.}%\footnote{\YL{need to modify the proof}}
%By Proposition \ref{prop:F(a,b)proper} we know that $\ofM^{\GIT} \cong \ofM_{c}^{\K}\xrightarrow{\sigma_{a,b}'}\sF(a,b)$.
Since $\sF$ and $\tfM^{\GIT}$ are isomorphic in codimension $1$, it suffices to show that 
\begin{equation}\label{eq:Proj-case3}
\Proj~ R(\tfM^{\GIT}, \tlambda+\frac{a}{2}\tH_h+\frac{b}{2}\tH_u)\cong \sF(a,b).
\end{equation}
By Definition \ref{def:F(a,b)}, we know that $\tfM^{\GIT}$ is the gluing of $U_u$ and $\ofM^{\K}_{\frac{1}{3}+\epsilon}\setminus \{[(\bP^3, T)]\}$, while $\sF(a,b)$ is the gluing of $U_u$ and $\ofM^{\K}_{c}\setminus \{[(\bP^3, T)]\}$ with $c=\frac{1}{1+2a}\leq \frac{1}{3}$. Thus we have $\ofM_c^{\K}\cong \ofM^{\GIT}$ by Theorem \ref{thm:firstwall}. Thus the birational morphism $\trho: \tfM^{\GIT} \to \ofM^{\GIT}$ can be decomposed into 
\[
\tfM^{\GIT} \xrightarrow{\rho_1} \sF(a,b) \xrightarrow{\rho_2} \ofM^{\GIT},
\]
where $\rho_1$ and $\rho_2$ contracts $\tH_h$ and $(\rho_1)_* \tH_u$ respectively. Note that $\rho_2$ is also induced by $\sigma_{a,b}$.

%Clearly, $\rho_1:\tfM^{\GIT}\to \mathrm{Bl}_{[T]}\ofM^{\GIT}$ is the Kirwan weighted blow-up at $[2Q]$ extracting $\tH_h$.
We first show ampleness of $(\rho_1)_* (\tlambda+\frac{b}{2}\tH_u)$. Indeed, by \eqref{eq:GIT-lambda} we know that $(\rho_1)_* (\tlambda+\frac{1}{2}\tH_u)$ is the pull-back of the ample $\bQ$-line bundle $\Lambda^{\GIT}$ on $\ofM^{\GIT}$ under %the Kirwan blow-up
$\rho_2$. %:\mathrm{Bl}_{[T]}\ofM^{\GIT}\to \ofM^{\GIT}$.
Since $(\rho_1)_* \tH_u$ is %$\rho_2$-exceptional, it is 
$\rho_2$-anti-ample by Lemma \ref{lem:antiample}, we have that $(\rho_1)_* (\tlambda+\frac{1-\epsilon}{2}\tH_u)$ is ample for $0<\epsilon\ll 1$. On the other hand, we know that $(\tlambda +\frac{1}{2}\tH_h)|_{\tH_h}$ is $\bQ$-linearly trivial by \eqref{eq:bh-2}. Thus 
$(\rho_1)^*(\rho_1)_* \tlambda=\tlambda+\frac{1}{2}\tH_h$. By Definition \ref{def:blow-up-GIT} and Lemma \ref{lem:CM-perturb}(3), we know that $\tfM^{\GIT}= \sF(1-\epsilon, 1-\epsilon)$ and $\tlambda+\frac{1-\epsilon}{2}\tH_h$ is nef. By letting $\epsilon\to 0$ we obtain nefness of $\tlambda+\frac{1}{2}\tH_h$ which implies the nefness of $(\rho_1)_* \tlambda$. This together with ampleness of $(\rho_1)_* (\tlambda+\frac{1-\epsilon}{2}\tH_u)$ implies  that $(\rho_1)_* (\tlambda+\frac{b}{2}\tH_u)$ is ample.

Now we prove \eqref{eq:Proj-case3}. From the above arguments, we have that $(\rho_1)_* (\tlambda+\frac{b}{2}\tH_u)$ is ample on $\sF(a,b)$, %$\mathrm{Bl}_{[T]}\ofM^{\GIT}$,
and 
\[
\tlambda+\frac{a}{2}\tH_h+\frac{b}{2}\tH_u = (\rho_1)^*(\rho_1)_* (\tlambda+\frac{b}{2}\tH_u) + \frac{a-1}{2} \tH_h.
\]
Since $a\geq 1$ and $\tH_h$ is $\rho_1$-exceptional,  \eqref{eq:Proj-case3} follows. Thus Case 3 is proved.

\textbf{Case 4: $a,b\in [1,+\infty)$.} %\footnote{\YL{need to modify the proof}}
By Definition \ref{def:blow-up-GIT} we know that $\sF(a,b)\cong \ofM^{\GIT}$. Similarly to Case 3, it suffices to show that 
\begin{equation}\label{eq:Proj-case4}
\Proj~ R(\tfM^{\GIT}, \tlambda+\frac{a}{2}\tH_h+\frac{b}{2}\tH_u)\cong \ofM^{\GIT}.
\end{equation}
By \eqref{eq:GIT-lambda}, we know that $\trho^*\Lambda^{\GIT}=\tlambda+\frac{1}{2}\tH_h+\frac{1}{2}\tH_u$. Since $\Lambda^{\GIT}$ is ample on $\ofM^{\GIT}$, both $\tH_h$ and $\tH_u$ are $\trho$-exceptional, and $a,b\geq 1$, we conclude that $\eqref{eq:Proj-case4}$ holds. Thus Case 4 is proved.
\end{proof}

\begin{remark}\label{rmk:weightedblowup}
 From the identification of the K-moduli spaces with $\sF(a,b)$, and the results of \cite[Sections 5.1 and 5.2]{LO18b}, where Laza-O'Grady prove that $\trho: \tfM^{\GIT}=\sF(1-\epsilon, 1-\epsilon)\to \ofM^{\GIT}$ (denoted by $\tfM\to \fM$ in their notation) is a composition of  weighted blowups of Kirwan type at $[2Q]$ and $[T]$, we can conclude that the morphisms $\rho$ in Theorem \ref{thm:firstwall}(3)  and $\rho_u$ in Definition \ref{def:F(a,b)} are weighted blowups at $[2Q]$ and $[T]$ respectively.
\end{remark}

\begin{proof}[Proof of Theorem \ref{mthm:fg}]
This follows from Definition \ref{def:F(a,b)}, Theorems \ref{thm:k-moduli-walls}, and \ref{thm:finitegeneration}.
\end{proof}

\begin{proof}[Proof of Corollary \ref{cor:LO}]
This is a direct consequence of Theorem \ref{mthm:fg} by letting $a=b$.
\end{proof}

% We  fix $b\in (0,\frac{1}{2})$ and let $a$ increase from $0$ to $1$. From the above,  $\lambda(a,b)+\frac{a}{2}H_h(a,b)+\frac{b}{2}H_u(a,b)$ is ample if $a\in (0, \frac{1}{9})$. 
% By Lemma \ref{lem:CM-perturb}, we know that there exists $b'=b'(a)>\frac{a}{2}$ such that  $\lambda(a,b)+\frac{a}{2}H_h(a,b)+\frac{b'}{2} H_u(a,b)$ is ample. Hence by interpolation it suffices to show that $(\psi_{a,b})_*(\lambda+\frac{a}{2}H_h)$ is $\bQ$-Cartier and nef on $\sF(a)$ for $a\in (0,1)$. Since  $(\psi_{a,b})_*H_{u}$ is $\bQ$-Cartier as it is the exceptional divisor of Kirwan blow-up, we know that $(\psi_{a,b})_*(\lambda+\frac{a}{2}H_h)$ is always $\bQ$-Cartier. 
% To summarize, using induction on $a_i$ we have shown the nefness of $\lambda_{a}+\frac{a}{2}H_h(a,b)$ on $\sF(a)$ for $a\in (0,1)$ which implies the ampleness of $\lambda_{a}+\frac{a}{2}(H_h(a,b)+H_u(a,b))$. For the $a=1$ case, notice that $\phi_8^{-}:\sF(1-\epsilon)\to \sF(1)\cong \ofM^{\GIT}$ is precisely the Kirwan blow up at $[2Q]$ and $[T]$, and by \cite{LO19} we know that  $\lambda_{1-\epsilon}+\frac{1}{2}(H_{h,1-\epsilon}+H_{u,1-\epsilon})$ is $\phi_8^{-}$-trivial. Hence the section ring $R(\sF, \lambda+\Delta)$ is isomorphic to $R(\ofM^{\GIT}, (\psi_1)_*(\lambda+\Delta))$. From \cite{LO19} we know that $(\psi_1)_*(\lambda+\Delta)$ is ample on  $\ofM^{\GIT}$. Hence the proof is finished.

% The wall crossing result follows from the K-moduli wall crossing (see Theorem ???).
%\end{proof}

\begin{rem}
%\textcolor{blue}{YL: Using similar methods, we should be able to construct ample model of $\lambda+\frac{a}{2}H_h+\frac{b}{2}H_u$ on $\sF$ for any non-negative $a,b$. The walls are determined by $a\in \{\beta_i\}$ or $b\in \{0,1\}$. Think about how to treat the case where either $a$ or $b$ is zero. Is $\sF^*$ of picard rank $1$? If so, there might be intermediate models between $\sF^*$ and $\hsF$ corresponding to $a=0$ or $b=0$.}
It is reasonable to expect that $R(\sF, \lambda+\frac{a}{2} H_h +\frac{b}{2}H_u)$ is finitely generated if one of $a$ or $b$ is zero and the other is positive. Taking $\Proj$ of these rings should extend the wall-crossing picture as described in Theorem \ref{mthm:fg} to $a,b\in \bQ_{\geq 0}$ where $\sF(0,0)=\sF^*$. 
\end{rem}

\subsection{Quartic double solids}

Recall that a smooth quartic double solid $Y$ is a double cover of $\bP^3$ branched along a smooth quartic surface $S$. Denote the double cover map by $\pi:Y\to \bP^3$. A quartic double solid is the same as a del Pezzo threefold of degree $2$ (see e.g. \cite{Fuj90}). By \cite[Example 4.2]{Der16}, we know that $Y$ is K\"ahler-Einstein and hence K-stable as $\Aut(Y)$ is finite. Since any smooth deformation of a del Pezzo threefold is still del Pezzo of the same degree, there exists an open substack $\sY$ of $\cM_{3, 16}^{\rm Kss}$ parametrizing all smooth quartic double solids. Let $\osY$ be the Zariski closure of $\sY$ in $\cM_{3, 16}^{\rm Kss}$ with reduced structure. By definition, we know that $\osY$ is also a closed substack of $\ocM_{3,16}^{\rm sm, Kss}$. Let $\ofY$ be the good moduli space of $\osY$, then $\ofY$ is a closed subscheme of $\oM_{3,16}^{\rm sm, Kps}$ and $M_{3,16}^{\rm Kps}$. Since $\sY$ parametrizes K-stable pairs, it is a saturated Deligne-Mumford open substack of $\osY$. Hence $\sY$ admits a coarse moduli space $\fY$ as an open subscheme of $\ofY$. We call $\ofY$ the \emph{K-moduli space of quartic double solids}. We know that $\ofY$ is an irreducible component of $M_{3,16}^{\rm Kps}$ since $\fY$ is open in $M_{3,16}^{\rm Kps}$.

\begin{prop}\label{prop:qds}
There exists a bijective morphism $\iota:\ofM_{\frac{1}{2}}^{\K}\to \ofY$.
\end{prop}

\begin{proof}
Consider the Hilbert scheme  $\mathbb{H}$  of pairs $(X,D) \subset \bP^9$ with Hilbert polynomial $\chi(\bP^3, \calO(4))$ and $D \sim -K_X$.  Let $\mathcal{H}$ denote the locally closed subscheme of $\mathbb{H}$ parameterizing K-semistable pairs $(X,\frac{1}{2}D) \subset \bP^9$ (see \cite[Definition 3.7]{ADL19} or \cite[Theorem 2.21]{ADL20}).  Because $c = \frac{1}{2}$, the wall crossing results in Section 
\ref{sec:wallcrossings} prove that, if $[(X, \frac{1}{2}D)] \in \mathcal{H}$, then $X \cong \bP^3, X_h$, or $\bP(1,1,2,4)$.

By Lemma \ref{lem:Xh-unobstructed}, the $\bQ$-Gorenstein deformations of $X = \bP^3$, $X_h$, or $\bP(1,1,2,4)$ are unobstructed in $\ocM_{3,64}^{{\rm sm},\delta\geq \epsilon_0}$, and there is no torsion in the class group for $\calX_t = \bP^3, X_h,$ or $\bP(1,1,2,4)$ (see e.g. \cite[Proposition 3.14]{Kol13}). Hence $\cO_{\bP^9}(1)|_X$ is the unique Weil divisor class that is $\bQ$-linearly equivalent to $-\frac{1}{2}K_X$. % which implies that any divisor in the linear system $|-\frac{1}{2}(K_X)| =  |\cO_{\bP^9}(1)|_X|$ is Cartier, so does not cause further obstructions.  
Therefore, $\mathcal{H}$ is smooth, so the quotient stack $[ \mathcal{H} / \PGL(10, \mathbb{C})]$ is smooth.

By construction of $\mathcal{H}$ which parameterizes K-semistable pairs, universality of K-moduli gives a map $[\mathcal{H}/\PGL(10, \mathbb{C})] \to \ofM_{\frac{1}{2}}^\K$.   This map is separated, stabilizer preserving, and bijective on $\mathbb{C}$-points.  Therefore, by \cite[Theorem A.5]{AI}, we have that $[\mathcal{H}/\PGL(10, \mathbb{C})] \cong \osM_{\frac{1}{2}}^\K$.

We now construct a morphism $\iota: \ofM_{\frac{1}{2}}^{\K}\to \ofY$. %exists a $\bmu_2$-gerbe $\sZ\to \osM_{\frac{1}{2}}^{\K}$ that admits a morphism $\sZ \to \osY$.
Consider the universal family $(\cX, \cD) \to \osM_{\frac{1}{2}}^{\K}$. By the isomorphism of $[\mathcal{H}/\PGL(10, \mathbb{C})]$ and $\osM_{\frac{1}{2}}^\K$, there is a line bundle $\cO_{\cX}(1)$ on $\cX$ obtained as the pull-back of the line bundle $\calO_{\bP^9}(1)$ from the universal family on the Hilbert scheme. Since $\cO_{\cX_t}(2)\sim -K_{\cX_t}\sim \cD_t$ on each fiber $\cX_t$ for any $t \in |\osM_{\frac{1}{2}}^{\K}|$, %namely the divisor   %Note that there is no torsion in the class group for $\calX_t = \bP^3, X_h,$ or $\bP(1,1,2,4)$ (see e.g. \cite[Proposition 3.14]{Kol13}), so this restriction is $ \calN = \frac{1}{2}\calD$.
we know that $\cO_{\cX}(\cD) \otimes \cO_{\cX}(-2)$ is trivial on every fiber $\cX_t$, which implies that it is the pull-back of a line bundle $\cF$ on $\osM_{\frac{1}{2}}^\K$. Let $\phi_{\sZ}:\sZ \to \osM_{\frac{1}{2}}^\K$ be the $\bmu_2$-gerbe obtained as the second root stack of $\cF$ (see e.g. \cite[Appendix B.1]{AGV08}), i.e. $\sZ := \osM_{\frac{1}{2}}^\K \times_{B\bG_m} B_{\bG_m}$ where $\osM_{\frac{1}{2}}^\K \to {B\bG_m}$ is the classifying morphism of $\cF$, and $B\bG_m\to B\bG_m$ is the second power map. Hence there is a line bundle $\cG$ on $\sZ$ such that $\cG^{\otimes 2}$ is the pull-back of $\cF$ on $\sZ$. Denote by $\pi_{\sZ}:(\cX_{\sZ}, \cD_{\sZ})\to \sZ$ the base change of the family $(\cX,\cD)$ to $\sZ$. Then there is a line bundle $\cN_{\sZ}:= \cO_{\cX_{\sZ}}(1)\otimes \pi_{\sZ}^* \cG$ on $\sZ$ satisfying $\cN_{\sZ}^{\otimes 2} \cong \cO_{\cX_{\sZ}}(\cD_{\sZ})$.

Consider the double cover $\cY_{\sZ}$ of $\cX_{\sZ}$ branched along $\cD_{\sZ}$, i.e. 
\[
\cY_{\sZ}:= \mathrm{Spec}_{\cX_{\sZ}} \cO_{\cX_{\sZ}} \oplus \cN_{\sZ}^{\otimes -1},
\]
where the $\cO_{\cX_{\sZ}}$-algebra structure is induced by the sheaf homomorphism $\cN_{\sZ}^{\otimes -2}\xrightarrow{\cdot s} \cO_{\cX_{\sZ}}$ where $s$ is a section of $\cN_{\sZ}^{\otimes 2}$ such that $(s=0)=\cD_{\sZ}$.
%, which we denote by $\cZ$. Since all fibers $\cX_t$ are Fano with klt singularities, by Kodaira vanishing $H^i(\cX_t, \cN) = 0$ for $i > 0$. Therefore, by cohomology and base change, 
%\footnote{\YL{I don't think we need cohomology and base change here, as $\cN_{\sZ}$ is flat}}
Since $\cN_{\sZ}$ is locally free, the double cover $\cY_{\sZ} \to \cX_{\sZ}$ is also a fiberwise double cover. In particular, each fiber $(\cX_t,\frac{1}{2}\cD_t)$ for $t\in |\osM_{\frac{1}{2}}^{\K}|$ is the $\bmu_2$-quotient of $\cY_z$ where $z\in |\sZ|$ is the unique point lying over $t$. Thus, by \cite{LZ20, Zhu20} we know that $\cY_{\sZ} \to \sZ$ is a $\bQ$-Fano family with K-semistable fibers, where a general fiber is a smooth quartic double solid. Therefore, %$\cZ \to \osM_{\frac{1}{2}}^{\K}$ gives a family of K-semistable quartic double solids over $\osM_{\frac{1}{2}}^{\K}$. 
this gives $\sZ \to \osY$ by the universality of the K-moduli stack.

Next we prove that the composition  $\sZ \xrightarrow{\phi_{\sZ}} \osM_{\frac{1}{2}}^{\K} \to \ofM_{\frac{1}{2}}^{\K}$ provides a good moduli space of $\sZ$. By \cite{alper} it suffices to show that $(\phi_{\sZ})_* \cO_{\sZ} = \cO_{\osM_{\frac{1}{2}}^{\K}}$ and $\phi_{\sZ}$ is cohomologically affine. The first statement follows from the fact that $\phi_{\sZ}$ is a $\bmu_2$-gerbe. For the second statement, applying \cite[Proposition 3.10(vii)]{alper} to the Cartesian diagram in the fiber product construction of $\sZ$, it suffices to show that the second power map $f: B\bG_m\to B\bG_m$ is cohomologically affine. A quasi-coherent sheaf $V$ over $B\bG_m$ corresponds via the weight decomposition to a family $(V_i)_{i\in \bZ}$ of $\bC$-vector spaces. It is clear that $W:=f_* V $ corresponds to $(W_j)_{j\in \bZ}$ where $W_j = V_{2j}$. Since $V\mapsto V_i$ is exact for every $i\in \bZ$, we know that $f_*$ is exact. Hence $\phi_{\sZ}$ is cohomologically affine, which implies that  $\sZ \to \ofM_{\frac{1}{2}}^{\K}$ is a good moduli space morphism.
Descending the map $\sZ \to \osY$ to level of good moduli spaces gives the desired morphism $\iota:\ofM_{\frac{1}{2}}^{\K} \to \ofY$.

 %We let $\cD_{\cY}$ denote the pull-back of $\cD$ to the index one cover.  

%Consider the universal family $(\cX,\cD)\to \osM_{\frac{1}{2}}^{\K}$. By Theorem \ref{thm:k-moduli-walls} we know that any fiber $\cX_t$ for $t\in |\osM_{\frac{1}{2}}^{\K}|$ is isomorphic to $\bP^3$, $X_h$, or $\bP(1,1,2,4)$. In particular, there exists a $\bQ$-Cartier Weil divisorial sheaf $\cL_t$ on $\cX_t$ such that $-K_{\cX_t}\sim 4\cL_t$ and $\cN_t:=2\cL_t$ is Cartier.  Since $\osM_{\frac{1}{2}}^{\K}$ is a smooth stack, there exists a $\bmu_2$-gerbe $\sZ\to \osM_{\frac{1}{2}}^{\K}$ and a line bundle $\cN_{\sZ}$ on $\cX_{\sZ}$ where $(\cX_{\sZ},\cD_{\sZ}):=(\cX,\cD)\times_{\osM_{\frac{1}{2}}^{\K}}\sZ$, such that $-K_{\cX_{\sZ}/\sZ}\sim_{\sZ} 2\cN_{\sZ}\sim_{\sZ} \cD_{\sZ}$. Thus we can take a fiberwise double cover $\cY_{\sZ}\to \cX_{\sZ}$ branched along $\cD_{\sZ}$. By \cite{LZ20, Zhu20} we know that $\cY_{\sZ}\to \sZ$ has K-semistable fibers, where a general fiber is a smooth quartic double solid. Thus we obtain a morphism $\sZ\to \osY$ which descends to $\iota:\fZ\to \ofY$ where $\fZ$ is the good moduli space of $\sZ$. Since $\sZ$ is a $\bmu_2$-gerbe over $\osM_{\frac{1}{2}}^{\K}$, we have $\fZ \cong \ofM_{\frac{1}{2}}^{\K}$, and we identify them for simplicity. 

Next, we will show that $\iota:\ofM_{\frac{1}{2}}^{\K}\to \ofY$ is bijective. Clearly, $\fY$ is contained in the image of $\iota$, so properness of $\ofM_{\frac{1}{2}}^{\K}$ implies the surjectivity of $\iota$.
It suffices to show injectivity of $\iota$, i.e. for any two points $[(X,D)],[(X',D')]\in \ofM_{\frac{1}{2}}^{\K}$, if their double covers $Y$ and $Y'$ are isomorphic, then $(X,D)\cong (X',D')$. 
First of all, we know that $X$ (resp. $X'$) is isomorphic to $\bP^3$, $X_h$, or $\bP(1,1,2,4)$ by Theorem \ref{thm:k-moduli-walls}, where in the latter two cases $D$ (resp. $D'$) does not pass through the cone vertex. By Lemma \ref{lem:L-construct}, we know that there exist ample $\bQ$-Cartier Weil divisorial sheaves $L$ and $L'$ on $X$ and $X'$ respectively, such that $-K_X=4L$ and $-K_{X'}=4L'$. In addition, if $X$ (resp. $X'$) is a cone, then $L$ (resp. $L'$) is Cartier away from the cone vertex, and it has Cartier index $2$ at the cone vertex. Let $\tL$ and $\tL'$ be the pull-back of $L$ and $L'$ to $Y$ and $Y'$ respectively. If $X$ (resp. $X'$) is smooth, then clearly $Y$ (resp. $Y'$) has local complete intersection singularities. If $X$ (resp. $X'$) is singular, then $Y$ (resp. $Y'$) has precisely two singularities, the preimage of the cone vertex, that are not local complete intersections,  since $D$ (resp. $D'$) is away from the cone vertex by Theorem \ref{thm:k-moduli-walls}. Since $Y\cong Y'$, we know that $X$ is smooth if and only if $X'$ is smooth. We split into two cases. For simplicity we assume $Y=Y'$, and denote the double cover maps by $\pi:Y\to X$ and $\pi':Y \to X'$.

\textbf{Case 1: $X\cong X'\cong \bP^3$.}
In this case, both $\tL$ and $\tL'$ are Cartier on $Y$, and $-K_Y=2\tL=2\tL'$ which implies that $\tL-\tL'$ is a torsion Cartier divisor on $Y$. Since $Y$ is $\bQ$-Fano, it is rationally connected by \cite{Zha06} and hence simply connected. Thus any torsion line bundle on $Y$ is trivial which implies $\tL=\tL'$.  By \cite[Definition 2.50]{KM98} we know that $\pi_*\cO_Y(\tL)\cong\cO_X(L)\oplus \cO_X(-L)$ and  $\pi'_*\cO_Y(\tL')\cong\cO_{X'}(L')\oplus \cO_{X'}(-L')$. Thus 
\begin{equation}\label{eq:qds-1}
     H^0(X, \cO_X(L))\cong H^0(Y, \tL)=H^0(Y, \tL')\cong H^0(X', \cO_{X'}(L')).
\end{equation}
Hence the linear system $|\tL|$ induces a map $Y\to \bP^3$ isomorphic to both $\pi$ and $\pi'$. By taking ramification divisors, we obtain $(X,D)\cong (X',D')$.

\textbf{Case 2: both $X$ and  $X'$ are cones.}
In this case, denote the unique non-lci singularity in $X$ and $X'$ by $x$ and $x'$, respectively. Then $\pi^{-1}(x)=\pi'^{-1}(x')=:\{y_1,y_2\}$. From the geometry of $X_h$ and $\bP(1,1,2,4)$, we know that $\Pic(x\in X)\cong \Pic(x'\in X')\cong \bZ/2\bZ$ where $L$ and $L'$ are generators respectively. Since $\pi$ (resp. $\pi'$) is \'etale over a neighborhood of $x$ (resp. of $x'$),  we know that $\tL-\tL'$ is a torsion Cartier divisor on $Y$, and as above we conclude that $\tL=\tL'$. We  have \eqref{eq:qds-1}, and also $\pi_*\cO_Y(2\tL)=\cO_X(2L)\oplus\cO_X$ and $\pi'_*\cO_Y(2\tL')\cong\cO_{X'}(2L')\oplus \cO_{X'}$ which implies 
\begin{equation}\label{eq:qds-2}
    H^0(X,\cO_X(2L))\oplus \bC \cong H^0(Y, 2\tL)=H^0(Y,2 \tL')\cong H^0(X', \cO_{X'}(2L'))\oplus \bC.
\end{equation}
By choosing a basis $(s_0, s_1,s_2,s_3)$ of $H^0(X,L)$ and an element $s_4\in H^0(Y, 2\tL)\setminus H^0(X,\cO_X(2L))$, we obtain a morphism $[s_0,\cdots, s_4]:Y\to\bP(1^4,2)$ which is isomorphic to $\pi$ after taking the image. Similarly, we have $[s_0',\cdots,s_4']:Y\to \bP(1^4,2)$ isomorphic to $\pi'$ after taking the image. From the construction, \eqref{eq:qds-1}, and \eqref{eq:qds-2}, we know that $\pi$ and $\pi'$ only differ by an automorphism of $\bP(1^4,2)$, so they are isomorphic to each other. Thus $(X,D)\cong (X',D')$. 
\end{proof}

\begin{proof}[Proof of Theorem \ref{mthm:doublesolid}]
The first statement follows from Proposition \ref{prop:qds}. The diagram follows from Theorem \ref{thm:k-moduli-walls} where $p=[(\bP(1,1,2,4), (x_3^2=x_2^4))]$. Then $\tau(p)$ represents the weighted hypersurface $(x_4^2=x_3^2-x_2^4)\subset\bP(1,1,2,4,4)$ which is isomorphic to $(x_2^4=x_3x_4)\subset\bP(1,1,2,4,4)$ after an automorphism of $\bP(1,1,2,4,4)$.
\end{proof}

\begin{rem}
If we take quadruple cyclic covers instead of double covers, then similar arguments show that there is a finite birational morphism $\ofM_{\frac{3}{4}}^{\K}\to M_{3, 4}^{\rm Kps}$ whose image is the Zariski closure of the moduli space of smooth quartic threefolds that are cyclic covers of $\bP^3$ under linear projections. Note that smooth quartic threefolds are known to be K-stable by \cite{Che01, CP02, Fuj19b}, while their K-moduli compactification is currently unknown. 
\end{rem}

\subsection{Gorenstein $\bQ$-Fano degenerations of $\bP^3$}

In this subsection, we prove Theorem \ref{mthm:gordeg}.

\begin{proof}[Proof of Theorem \ref{mthm:gordeg}]
Let $X$ be a Gorenstein $\bQ$-Fano degeneration of $\bP^3$. By the effective non-vanishing theorem of Ambro \cite[Main Theorem]{Amb99} and Kawamata \cite[Theorem 5.1]{Kaw00}, there exists an effective Cartier divisor $S\in |-K_X|$ such that $(X,S)$ is plt. Hence Theorem \ref{thm:uKs-almost-logCY} implies that $(X, (1-\epsilon)S)$ is uniformly K-stable for $0<\epsilon\ll 1$. Let $\pi:\cX\to B$ be a $\bQ$-Fano family over a smooth pointed curve $0\in B$ such that $\cX_0\cong X$ and $\cX_b\cong \bP^3$ for any $b\in B\setminus \{0\}$. By Lemma \ref{lem:L-construct} we know that $\pi_*\omega_{\cX/B}^{\vee}$ is a vector bundle over $B$ whose fiber over $b\in B$ is precisely $H^0(\cX_b, \omega_{\cX_b}^{\vee})$. In particular, we know that $(X,S)$ admits a $\bQ$-Gorenstein smoothing to $(\bP^3, S_b)$ where $S_b$ is a smooth quartic surface. Hence $[(X,S)]\in \ofM_{1-\epsilon}^{\K}$, and the statement follows from Proposition \ref{prop:no-other-3-fold}.
\end{proof}

\bibliographystyle{alpha}
\bibliography{quarticK3}

\begin{thebibliography}{ABHLX20}

\bibitem[AB19]{ABK3}
Kenneth Ascher and Dori Bejleri.
\newblock {Compact moduli of elliptic K3 surfaces}.
\newblock {\em Geom. Topol., to appear}, 2019.
\newblock \href{https://arxiv.org/abs/1902.10686}{\textsf{arXiv:1902.10686}}.

\bibitem[ABHLX20]{ABHLX19}
Jarod Alper, Harold Blum, Daniel Halpern-Leistner, and Chenyang Xu.
\newblock Reductivity of the automorphism group of {K}-polystable {F}ano
  varieties.
\newblock {\em Invent. Math.}, 222(3):995--1032, 2020.

\bibitem[ADL19]{ADL19}
Kenneth Ascher, Kristin DeVleming, and Yuchen Liu.
\newblock {Wall crossing for K-moduli spaces of plane curves}.
\newblock 2019.
\newblock \href{https://arxiv.org/abs/1909.04576}{\textsf{arXiv:1909.04576}}.

\bibitem[ADL20]{ADL20}
Kenneth Ascher, Kristin DeVleming, and Yuchen Liu.
\newblock {K-moduli of curves on a quadric surface and K3 surfaces}.
\newblock {\em J. Inst. Math. Jussieu, to appear}, 2020.
\newblock \href{https://arxiv.org/abs/2006.06816}{\textsf{arXiv:2006.06816}}.

\bibitem[AE21]{AE21}
Valery Alexeev and Philip Engel.
\newblock {Compact moduli of K3 surfaces}, 2021.
\newblock \href{https://arxiv.org/abs/2101.12186}{\textsf{arXiv:2101.12186}}.

\bibitem[AET19]{AET}
Valery {Alexeev}, Philip {Engel}, and Alan {Thompson}.
\newblock {Stable pair compactification of moduli of K3 surfaces of degree 2}.
\newblock 2019.
\newblock \href{https://arxiv.org/abs/1903.09742}{\textsf{arXiv:1903.09742}}.

\bibitem[AFS17]{AFS17}
Jarod Alper, Maksym Fedorchuk, and David~Ishii Smyth.
\newblock Second flip in the {H}assett-{K}eel program: existence of good moduli
  spaces.
\newblock {\em Compos. Math.}, 153(8):1584--1609, 2017.

\bibitem[AGV08]{AGV08}
Dan Abramovich, Tom Graber, and Angelo Vistoli.
\newblock Gromov-{W}itten theory of {D}eligne-{M}umford stacks.
\newblock {\em Amer. J. Math.}, 130(5):1337--1398, 2008.

\bibitem[AI19]{AI}
Shamil Asgarli and Giovanni Inchiostro.
\newblock The {P}icard group of the moduli of smooth complete intersections of
  two quadrics.
\newblock {\em Trans. Amer. Math. Soc.}, 372(5):3319--3346, 2019.

\bibitem[Alp13]{alper}
Jarod Alper.
\newblock Good moduli spaces for {A}rtin stacks.
\newblock {\em Ann. Inst. Fourier (Grenoble)}, 63(6):2349--2402, 2013.

\bibitem[Alt97]{Altmann}
Klaus Altmann.
\newblock The versal deformation of an isolated toric {G}orenstein singularity.
\newblock {\em Invent. Math.}, 128(3):443--479, 1997.

\bibitem[Amb99]{Amb99}
Florin Ambro.
\newblock Ladders on {F}ano varieties.
\newblock {\em J. Math. Sci. (New York)}, 94(1):1126--1135, 1999.
\newblock Algebraic geometry, 9.

\bibitem[BCHM10]{BCHM10}
Caucher Birkar, Paolo Cascini, Christopher~D. Hacon, and James McKernan.
\newblock Existence of minimal models for varieties of log general type.
\newblock {\em J. Amer. Math. Soc.}, 23(2):405--468, 2010.

\bibitem[BHLLX21]{BHLLX20}
Harold Blum, Daniel Halpern-Leistner, Yuchen Liu, and Chenyang Xu.
\newblock On properness of {K}-moduli spaces and optimal degenerations of
  {F}ano varieties.
\newblock {\em Selecta Math. (N.S.)}, 27(4):Paper No. 73, 2021.

\bibitem[Bir19]{Bir19}
Caucher Birkar.
\newblock Anti-pluricanonical systems on {F}ano varieties.
\newblock {\em Ann. of Math. (2)}, 190(2):345--463, 2019.

\bibitem[BJ20]{BJ17}
Harold Blum and Mattias Jonsson.
\newblock Thresholds, valuations, and {K}-stability.
\newblock {\em Adv. Math.}, 365:107062, 57, 2020.

\bibitem[BL22]{BL18b}
Harold Blum and Yuchen Liu.
\newblock Openness of uniform {K}-stability in families of {$\Bbb Q$}-{F}ano
  varieties.
\newblock {\em Ann. Sci. \'{E}c. Norm. Sup\'{e}r. (4)}, 55(1):1--41, 2022.

\bibitem[BLX22]{BLX19}
Harold Blum, Yuchen Liu, and Chenyang Xu.
\newblock {Openness of K-semistability for Fano varieties}.
\newblock {\em Duke Math. J.}, 171(13):2753 -- 2797, 2022.

\bibitem[BX19]{BX18}
Harold Blum and Chenyang Xu.
\newblock Uniqueness of {K}-polystable degenerations of {F}ano varieties.
\newblock {\em Ann. of Math. (2)}, 190(2):609--656, 2019.

\bibitem[Cam92]{Cam92}
F.~Campana.
\newblock Connexit\'{e} rationnelle des vari\'{e}t\'{e}s de {F}ano.
\newblock {\em Ann. Sci. \'{E}cole Norm. Sup. (4)}, 25(5):539--545, 1992.

\bibitem[Che01]{Che01}
I.~A. Cheltsov.
\newblock Log canonical thresholds on hypersurfaces.
\newblock {\em Mat. Sb.}, 192(8):155--172, 2001.

\bibitem[CP02]{CP02}
I.~A. Cheltsov and Jihun Park.
\newblock Global log-canonical thresholds and generalized {E}ckardt points.
\newblock {\em Mat. Sb.}, 193(5):149--160, 2002.

\bibitem[CP21]{CP18}
Giulio Codogni and Zsolt Patakfalvi.
\newblock Positivity of the {CM} line bundle for families of {K}-stable klt
  {F}ano varieties.
\newblock {\em Invent. Math.}, 223(3):811--894, 2021.

\bibitem[Der16]{Der16}
Ruadha\'i Dervan.
\newblock On {K}-stability of finite covers.
\newblock {\em Bull. Lond. Math. Soc.}, 48(4):717--728, 2016.

\bibitem[DeV22]{dV}
Kristin DeVleming.
\newblock Moduli of surfaces in ${{\mathbb {P}}}^{3}$.
\newblock {\em Compositio Mathematica}, 158(6):1329–1374, 2022.

\bibitem[Dol82]{Dol82}
Igor Dolgachev.
\newblock Weighted projective varieties.
\newblock In {\em Group actions and vector fields ({V}ancouver, {B}.{C}.,
  1981)}, volume 956 of {\em Lecture Notes in Math.}, pages 34--71. Springer,
  Berlin, 1982.

\bibitem[EL20]{EL19}
Lawrence Ein and Robert Lazarsfeld.
\newblock Tangent developable surfaces and the equations defining algebraic
  curves.
\newblock {\em Bull. Amer. Math. Soc. (N.S.)}, 57(1):23--38, 2020.

\bibitem[EVdV81]{EVdV81}
D.~Eisenbud and A.~Van~de Ven.
\newblock On the normal bundles of smooth rational space curves.
\newblock {\em Math. Ann.}, 256(4):453--463, 1981.

\bibitem[FO18]{FO16}
Kento Fujita and Yuji Odaka.
\newblock On the {K}-stability of {F}ano varieties and anticanonical divisors.
\newblock {\em Tohoku Math. J. (2)}, 70(4):511--521, 2018.

\bibitem[FR06]{FR06}
J.~Fine and J.~Ross.
\newblock A note on positivity of the {CM} line bundle.
\newblock {\em Int. Math. Res. Not.}, pages Art. ID 95875, 14, 2006.

\bibitem[Fuj90]{Fuj90}
Takao Fujita.
\newblock On singular del {P}ezzo varieties.
\newblock In {\em Algebraic geometry ({L}'{A}quila, 1988)}, volume 1417 of {\em
  Lecture Notes in Math.}, pages 117--128. Springer, Berlin, 1990.

\bibitem[Fuj19a]{Fuj19b}
Kento Fujita.
\newblock K-stability of {F}ano manifolds with not small alpha invariants.
\newblock {\em J. Inst. Math. Jussieu}, 18(3):519--530, 2019.

\bibitem[Fuj19b]{Fuj19a}
Kento Fujita.
\newblock Uniform {K}-stability and plt blowups of log {F}ano pairs.
\newblock {\em Kyoto J. Math.}, 59(2):399--418, 2019.

\bibitem[Fuj19c]{Fuj19}
Kento Fujita.
\newblock A valuative criterion for uniform {K}-stability of {$\Bbb Q$}-{F}ano
  varieties.
\newblock {\em J. Reine Angew. Math.}, 751:309--338, 2019.

\bibitem[GKKP11]{GKKP11}
Daniel Greb, Stefan Kebekus, S\'{a}ndor~J. Kov\'{a}cs, and Thomas Peternell.
\newblock Differential forms on log canonical spaces.
\newblock {\em Publ. Math. Inst. Hautes \'{E}tudes Sci.}, (114):87--169, 2011.

\bibitem[GMGS21]{GMGS}
Patricio Gallardo, Jesus Martinez-Garcia, and Cristiano Spotti.
\newblock Applications of the moduli continuity method to log {K}-stable pairs.
\newblock {\em J. Lond. Math. Soc. (2)}, 103(2):729--759, 2021.

\bibitem[Hac01]{Hac01}
Paul Hacking.
\newblock {\em A compactification of the space of plane curves}.
\newblock 2001.
\newblock Thesis (Ph.D.)-- Cambridge Univserity.

\bibitem[Har77]{Har77}
Robin Hartshorne.
\newblock {\em Algebraic geometry}.
\newblock Springer-Verlag, New York-Heidelberg, 1977.
\newblock Graduate Texts in Mathematics, No. 52.

\bibitem[HH13]{hassett2013log}
Brendan Hassett and Donghoon Hyeon.
\newblock Log minimal model program for the moduli space of stable curves: the
  first flip.
\newblock {\em Annals of Mathematics}, 177:1--58, 2013.

\bibitem[Huy16]{huybrechts}
Daniel Huybrechts.
\newblock {\em Lectures on {K}3 surfaces}, volume 158 of {\em Cambridge Studies
  in Advanced Mathematics}.
\newblock Cambridge University Press, Cambridge, 2016.

\bibitem[IP99]{IP99}
V.~A. Iskovskikh and Yu.~G. Prokhorov.
\newblock Fano varieties.
\newblock In {\em Algebraic geometry, {V}}, volume~47 of {\em Encyclopaedia
  Math. Sci.}, pages 1--247. Springer, Berlin, 1999.

\bibitem[Jia20]{Jia17}
Chen Jiang.
\newblock Boundedness of {$\Bbb Q$}-{F}ano varieties with degrees and
  alpha-invariants bounded from below.
\newblock {\em Ann. Sci. \'{E}c. Norm. Sup\'{e}r. (4)}, 53(5):1235--1248, 2020.

\bibitem[JMR16]{JMR16}
Thalia Jeffres, Rafe Mazzeo, and Yanir~A. Rubinstein.
\newblock K\"{a}hler-{E}instein metrics with edge singularities.
\newblock {\em Ann. of Math. (2)}, 183(1):95--176, 2016.

\bibitem[Kaw00]{Kaw00}
Yujiro Kawamata.
\newblock On effective non-vanishing and base-point-freeness.
\newblock {\em Asian J. Math.}, 4(1):173--181, 2000.
\newblock Kodaira's issue.

\bibitem[KKL16]{KKL16}
Anne-Sophie Kaloghiros, Alex K\"{u}ronya, and Vladimir Lazi\'{c}.
\newblock Finite generation and geography of models.
\newblock In {\em Minimal models and extremal rays ({K}yoto, 2011)}, volume~70
  of {\em Adv. Stud. Pure Math.}, pages 215--245. Math. Soc. Japan, [Tokyo],
  2016.

\bibitem[KM76]{KM76}
Finn~Faye Knudsen and David Mumford.
\newblock The projectivity of the moduli space of stable curves. {I}.
  {P}reliminaries on ``det'' and ``{D}iv''.
\newblock {\em Math. Scand.}, 39(1):19--55, 1976.

\bibitem[KM98]{KM98}
J\'{a}nos Koll\'{a}r and Shigefumi Mori.
\newblock {\em Birational geometry of algebraic varieties}, volume 134 of {\em
  Cambridge Tracts in Mathematics}.
\newblock Cambridge University Press, Cambridge, 1998.
\newblock With the collaboration of C. H. Clemens and A. Corti, Translated from
  the 1998 Japanese original.

\bibitem[KMM92]{KMM92}
J\'{a}nos Koll\'{a}r, Yoichi Miyaoka, and Shigefumi Mori.
\newblock Rational connectedness and boundedness of {F}ano manifolds.
\newblock {\em J. Differential Geom.}, 36(3):765--779, 1992.

\bibitem[Kni73]{Kni73}
Carol~M. Knighten.
\newblock Differentials on quotients of algebraic varieties.
\newblock {\em Trans. Amer. Math. Soc.}, 177:65--89, 1973.

\bibitem[Kol97]{SingularitiesOfPairs}
J\'{a}nos Koll\'{a}r.
\newblock Singularities of pairs.
\newblock In {\em Algebraic geometry---{S}anta {C}ruz 1995}, volume~62 of {\em
  Proc. Sympos. Pure Math.}, pages 221--287. Amer. Math. Soc., Providence, RI,
  1997.

\bibitem[Kol13]{Kol13}
J\'{a}nos Koll\'{a}r.
\newblock {\em Singularities of the minimal model program}, volume 200 of {\em
  Cambridge Tracts in Mathematics}.
\newblock Cambridge University Press, Cambridge, 2013.
\newblock With a collaboration of S\'{a}ndor Kov\'{a}cs.

\bibitem[Kol18]{Kol18}
J\'{a}nos Koll\'{a}r.
\newblock Mumford divisors.
\newblock 2018.
\newblock \href{https://arxiv.org/abs/1803.07596}{\textsf{arXiv:1803.07596}}.

\bibitem[Kol19]{Kol19}
J\'{a}nos Koll\'{a}r.
\newblock Families of divisors.
\newblock 2019.
\newblock \href{https://arxiv.org/abs/1910.00937}{\textsf{arXiv:1910.00937}}.

\bibitem[Laz04]{Positivity1}
Robert Lazarsfeld.
\newblock {\em Positivity in algebraic geometry. {I}}, volume~48 of {\em
  Ergebnisse der Mathematik und ihrer Grenzgebiete. 3. Folge. A Series of
  Modern Surveys in Mathematics [Results in Mathematics and Related Areas. 3rd
  Series. A Series of Modern Surveys in Mathematics]}.
\newblock Springer-Verlag, Berlin, 2004.
\newblock Classical setting: line bundles and linear series.

\bibitem[Laz16]{laza2012ksba}
Radu Laza.
\newblock The {KSBA} compactification for the moduli space of degree two {$K3$}
  pairs.
\newblock {\em J. Eur. Math. Soc. (JEMS)}, 18(2):225--279, 2016.

\bibitem[Li17]{Li17}
Chi Li.
\newblock K-semistability is equivariant volume minimization.
\newblock {\em Duke Math. J.}, 166(16):3147--3218, 2017.

\bibitem[Liu22]{Liu20}
Yuchen Liu.
\newblock K-stability of cubic fourfolds.
\newblock {\em J. Reine Angew. Math.}, 786:55--77, 2022.

\bibitem[LL19]{LL16}
Chi Li and Yuchen Liu.
\newblock K\"{a}hler-{E}instein metrics and volume minimization.
\newblock {\em Adv. Math.}, 341:440--492, 2019.

\bibitem[LLX20]{LLX18}
Chi Li, Yuchen Liu, and Chenyang Xu.
\newblock A guided tour to normalized volume.
\newblock In {\em Geometric analysis}, volume 333 of {\em Progr. Math.}, pages
  167--219. Birkh\"{a}user/Springer, Cham, 2020.

\bibitem[LMB00]{LMB00}
G\'{e}rard Laumon and Laurent Moret-Bailly.
\newblock {\em Champs alg\'{e}briques}, volume~39 of {\em Ergebnisse der
  Mathematik und ihrer Grenzgebiete. 3. Folge. A Series of Modern Surveys in
  Mathematics [Results in Mathematics and Related Areas. 3rd Series. A Series
  of Modern Surveys in Mathematics]}.
\newblock Springer-Verlag, Berlin, 2000.

\bibitem[LO18]{LO18b}
Radu Laza and Kieran~G. O'Grady.
\newblock G{IT} versus {B}aily-{B}orel compactification for quartic {$K3$}
  surfaces.
\newblock In {\em Geometry of moduli}, volume~14 of {\em Abel Symp.}, pages
  217--283. Springer, Cham, 2018.

\bibitem[LO19]{LO19}
Radu Laza and Kieran O'Grady.
\newblock Birational geometry of the moduli space of quartic {$K3$} surfaces.
\newblock {\em Compos. Math.}, 155(9):1655--1710, 2019.

\bibitem[LO21]{LO18a}
Radu Laza and Kieran O'Grady.
\newblock G{IT} versus {B}aily-{B}orel compactification for {$K3$}'s which are
  double covers of {$\Bbb P^1\times\Bbb P^1$}.
\newblock {\em Adv. Math.}, 383:107680, 63, 2021.

\bibitem[Loo86]{Loo86}
Eduard Looijenga.
\newblock New compactifications of locally symmetric varieties.
\newblock In {\em Proceedings of the 1984 {V}ancouver conference in algebraic
  geometry}, volume~6 of {\em CMS Conf. Proc.}, pages 341--364. Amer. Math.
  Soc., Providence, RI, 1986.

\bibitem[Loo03a]{Loo031}
Eduard Looijenga.
\newblock Compactifications defined by arrangements. {I}. {T}he ball quotient
  case.
\newblock {\em Duke Math. J.}, 118(1):151--187, 2003.

\bibitem[Loo03b]{Loo03}
Eduard Looijenga.
\newblock Compactifications defined by arrangements. {II}. {L}ocally symmetric
  varieties of type {IV}.
\newblock {\em Duke Math. J.}, 119(3):527--588, 2003.

\bibitem[LWX19]{LWX19}
Chi Li, Xiaowei Wang, and Chenyang Xu.
\newblock On the proper moduli spaces of smoothable {K}\"{a}hler-{E}instein
  {F}ano varieties.
\newblock {\em Duke Math. J.}, 168(8):1387--1459, 2019.

\bibitem[LWX21]{LWX18}
Chi Li, Xiaowei Wang, and Chenyang Xu.
\newblock Algebraicity of the metric tangent cones and equivariant
  {K}-stability.
\newblock {\em J. Amer. Math. Soc.}, 34(4):1175--1214, 2021.

\bibitem[LX14]{LX14}
Chi Li and Chenyang Xu.
\newblock Special test configuration and {K}-stability of {F}ano varieties.
\newblock {\em Ann. of Math. (2)}, 180(1):197--232, 2014.

\bibitem[LX19]{LX19}
Yuchen Liu and Chenyang Xu.
\newblock K-stability of cubic threefolds.
\newblock {\em Duke Math. J.}, 168(11):2029--2073, 2019.

\bibitem[LX20]{LX16}
Chi Li and Chenyang Xu.
\newblock Stability of valuations and {K}oll\'{a}r components.
\newblock {\em J. Eur. Math. Soc. (JEMS)}, 22(8):2573--2627, 2020.

\bibitem[LXZ22]{LXZ21}
Yuchen Liu, Chenyang Xu, and Ziquan Zhuang.
\newblock Finite generation for valuations computing stability thresholds and
  applications to {K}-stability.
\newblock {\em Ann. of Math. (2)}, 196(2):507--566, 2022.

\bibitem[LZ20]{LZ20}
Yuchen Liu and Ziwen Zhu.
\newblock Equivariant {K}-stability under finite group action.
\newblock 2020.
\newblock \href{https://arxiv.org/abs/2001.10557}{\textsf{arXiv:2001.10557}}.

\bibitem[LZ22]{LZ19}
Yuchen Liu and Ziquan Zhuang.
\newblock On the sharpness of {T}ian's criterion for {K}-stability.
\newblock {\em Nagoya Math. J.}, 245:41--73, 2022.

\bibitem[May72]{Mayer}
Alan~L. Mayer.
\newblock Families of {$K-3$} surfaces.
\newblock {\em Nagoya Mathematical Journal}, 48:1--17, 1972.

\bibitem[Mir81]{Mir81}
Rick Miranda.
\newblock The moduli of {W}eierstrass fibrations over {${\bf P}\sp{1}$}.
\newblock {\em Math. Ann.}, 255(3):379--394, 1981.

\bibitem[Nam97]{Namikawa}
Yoshinori Namikawa.
\newblock Smoothing {F}ano {$3$}-folds.
\newblock {\em J. Algebraic Geom.}, 6(2):307--324, 1997.

\bibitem[Oda13]{Oda13b}
Yuji Odaka.
\newblock The {GIT} stability of polarized varieties via discrepancy.
\newblock {\em Ann. of Math. (2)}, 177(2):645--661, 2013.

\bibitem[Oda15]{Oda15}
Yuji Odaka.
\newblock Compact moduli spaces of {K}\"{a}hler-{E}instein {F}ano varieties.
\newblock {\em Publ. Res. Inst. Math. Sci.}, 51(3):549--565, 2015.

\bibitem[OO21]{OO18}
Yuji Odaka and Yoshiki Oshima.
\newblock {\em {Collapsing K3 surfaces, tropical geometry and moduli
  compactifications of Satake, Morgan-Shalen type}}, volume~40 of {\em MSJ
  Memoirs}.
\newblock Mathematical Society of Japan, Tokyo, 2021.

\bibitem[OSS16]{OSS16}
Yuji Odaka, Cristiano Spotti, and Song Sun.
\newblock Compact moduli spaces of del {P}ezzo surfaces and
  {K}\"ahler-{E}instein metrics.
\newblock {\em J. Differential Geom.}, 102(1):127--172, 2016.

\bibitem[Pet22a]{Pet21}
Andrea Petracci.
\newblock On deformation spaces of toric singularities and on singularities of
  {K}-moduli of {F}ano varieties.
\newblock {\em Trans. Amer. Math. Soc.}, 375(8):5617--5643, 2022.

\bibitem[Pet22b]{Petracci}
Andrea Petracci.
\newblock On deformations of toric {F}ano varieties.
\newblock In {\em Interactions with lattice polytopes}, volume 386 of {\em
  Springer Proc. Math. Stat.}, pages 287--314. Springer, Cham, [2022]
  \copyright 2022.

\bibitem[Pos22]{Pos19}
Quentin Posva.
\newblock Positivity of the {CM} line bundle for {K}-stable log {F}anos.
\newblock {\em Trans. Amer. Math. Soc.}, 375(7):4943--4978, 2022.

\bibitem[Pro00]{Pro00}
Yu.~G. Prokhorov.
\newblock Blow-ups of canonical singularities.
\newblock In {\em Algebra ({M}oscow, 1998)}, pages 301--317. de Gruyter,
  Berlin, 2000.

\bibitem[PT06]{PT06}
Sean~Timothy Paul and Gang Tian.
\newblock C{M} stability and the generalized {F}utaki invariant {I}.
\newblock 2006.
\newblock
  \href{https://arxiv.org/abs/math/0605278}{\textsf{arXiv:math/0605278}}.

\bibitem[PT09]{PT09}
Sean~Timothy Paul and Gang Tian.
\newblock C{M} stability and the generalized {F}utaki invariant {II}.
\newblock {\em Ast\'{e}risque}, (328):339--354 (2010), 2009.

\bibitem[Ser06]{Sernesi}
Edoardo Sernesi.
\newblock {\em Deformations of algebraic schemes}, volume 334 of {\em
  Grundlehren der Mathematischen Wissenschaften [Fundamental Principles of
  Mathematical Sciences]}.
\newblock Springer-Verlag, Berlin, 2006.

\bibitem[Sha80]{Sha80}
Jayant Shah.
\newblock A complete moduli space for {$K3$} surfaces of degree {$2$}.
\newblock {\em Ann. of Math. (2)}, 112(3):485--510, 1980.

\bibitem[Sha81]{Sha81}
Jayant Shah.
\newblock Degenerations of {$K3$} surfaces of degree {$4$}.
\newblock {\em Trans. Amer. Math. Soc.}, 263(2):271--308, 1981.

\bibitem[Shi89]{Shin89}
Kil-Ho Shin.
\newblock {$3$}-dimensional {F}ano varieties with canonical singularities.
\newblock {\em Tokyo J. Math.}, 12(2):375--385, 1989.

\bibitem[SS17]{SS17}
Cristiano Spotti and Song Sun.
\newblock Explicit {G}romov-{H}ausdorff compactifications of moduli spaces of
  {K}\"{a}hler-{E}instein {F}ano manifolds.
\newblock {\em Pure Appl. Math. Q.}, 13(3):477--515, 2017.

\bibitem[{Sta}18]{stacks-project}
The {Stacks Project Authors}.
\newblock \textit{Stacks Project}.
\newblock \url{https://stacks.math.columbia.edu}, 2018.

\bibitem[Ste77a]{Ste77}
J.~H.~M. Steenbrink.
\newblock Mixed {H}odge structure on the vanishing cohomology.
\newblock In {\em Real and complex singularities ({P}roc. {N}inth {N}ordic
  {S}ummer {S}chool/{NAVF} {S}ympos. {M}ath., {O}slo, 1976)}, pages 525--563.
  Sijthoff and Noordhoff, Alphen aan den Rijn, 1977.

\bibitem[Ste77b]{Steenbrink}
Joseph Steenbrink.
\newblock Intersection form for quasi-homogeneous singularities.
\newblock {\em Compositio Math.}, 34(2):211--223, 1977.

\bibitem[Tia97]{Tia97}
Gang Tian.
\newblock K\"{a}hler-{E}instein metrics with positive scalar curvature.
\newblock {\em Invent. Math.}, 130(1):1--37, 1997.

\bibitem[Xu14]{Xu14}
Chenyang Xu.
\newblock Finiteness of algebraic fundamental groups.
\newblock {\em Compos. Math.}, 150(3):409--414, 2014.

\bibitem[Xu20]{Xu19}
Chenyang Xu.
\newblock A minimizing valuation is quasi-monomial.
\newblock {\em Ann. of Math. (2)}, 191(3):1003--1030, 2020.

\bibitem[Xu21]{Xu20}
Chenyang Xu.
\newblock K-stability of {F}ano varieties: an algebro-geometric approach.
\newblock {\em EMS Surv. Math. Sci.}, 8(1-2):265--354, 2021.

\bibitem[XZ20]{XZ19}
Chenyang Xu and Ziquan Zhuang.
\newblock On positivity of the {CM} line bundle on {K}-moduli spaces.
\newblock {\em Ann. of Math. (2)}, 192(3):1005--1068, 2020.

\bibitem[XZ21]{XZ20}
Chenyang Xu and Ziquan Zhuang.
\newblock Uniqueness of the minimizer of the normalized volume function.
\newblock {\em Camb. J. Math.}, 9(1):149--176, 2021.

\bibitem[Zha06]{Zha06}
Qi~Zhang.
\newblock Rational connectedness of log {${\bf Q}$}-{F}ano varieties.
\newblock {\em J. Reine Angew. Math.}, 590:131--142, 2006.

\bibitem[Zho21a]{Zho21a}
Chuyu Zhou.
\newblock {Log K-stability of GIT-stable divisors on Fano manifolds}.
\newblock 2021.
\newblock \href{https://arxiv.org/abs/2102.07458}{\textsf{arXiv:2102.07458}}.

\bibitem[Zho21b]{Zho21b}
Chuyu Zhou.
\newblock On wall crossing for {K}-stability.
\newblock 2021.
\newblock \href{https://arxiv.org/abs/2103.04925}{\textsf{arXiv:2103.04925}}.

\bibitem[Zhu21]{Zhu20}
Ziquan Zhuang.
\newblock Optimal destabilizing centers and equivariant {K}-stability.
\newblock {\em Invent. Math.}, 226(1):195--223, 2021.

\end{thebibliography}

\end{document}